\documentclass[12pt]{article}

\usepackage[T1]{fontenc}
\usepackage{lmodern}
\usepackage[utf8]{inputenc}

\usepackage[english]{babel}
\usepackage{fixltx2e,ellipsis}
\usepackage[tracking=true]{microtype}%
\DeclareMicrotypeSet*[tracking]{my}%
  { font = */*/*/sc/* }%
\SetTracking{ encoding = *, shape = sc }{ 45 }%

\usepackage[intlimits]{amsmath}
\usepackage{amsfonts}
\usepackage{amssymb}
\usepackage{enumerate}
\usepackage{times}

\makeatletter
\def\author#1{\gdef\autrun{\def\and{\unskip, }#1}\gdef\@author{#1}}
\def\address#1{{\def\and{\\\hspace*{18pt}}\renewcommand{\thefootnote}{}\footnote {#1}}}
\makeatother
\def\email#1{e-mail: #1}
\def\subjclass#1{{\renewcommand{\thefootnote}{}\footnote{\emph{Mathematics Subject Classification (2010):} #1}}}
\def\keywords#1{\par\medskip\noindent\textbf{Keywords.} #1}

\usepackage{amsthm}
\newtheorem{theorem}{Theorem}[section]
\newtheorem{lemma}[theorem]{Lemma}
\newtheorem{proposition}[theorem]{Proposition}
\newtheorem{corollary}[theorem]{Corollary}

\theoremstyle{definition}
\newtheorem*{definition}{Definition}

\newtheorem*{example*}{Example}

\newtheorem{remark}[theorem]{Remark}
\numberwithin{equation}{section}

\usepackage{graphicx}
\usepackage{tikz}
\usetikzlibrary{arrows,positioning}

\usepackage{mleftright}\mleftright
\usepackage[
	bookmarks=true,
	bookmarksnumbered=true,
	unicode=false,
	pdftoolbar=true,
	pdfmenubar=true,
	pdffitwindow=false,
	pdfstartview={FitH},
	pdftitle={},
	pdfauthor={},
	pdfsubject={},
	pdfcreator={},
	pdfproducer={},
	pdfkeywords={},
	pdfnewwindow=true,
	colorlinks=true,
	linkcolor=black,
	citecolor=black,
	filecolor=blue,
	urlcolor=black]{hyperref}

\makeatletter
	\def\@cite#1#2{[\textbf{#1}\if@tempswa , #2\fi]}	
\makeatother

\makeatletter
\newcommand\RedeclareMathOperator{%
  \@ifstar{\def\rmo@s{m}\rmo@redeclare}{\def\rmo@s{o}\rmo@redeclare}%
}
\newcommand\rmo@redeclare[2]{%
  \begingroup \escapechar\m@ne\xdef\@gtempa{{\string#1}}\endgroup
  \expandafter\@ifundefined\@gtempa
     {\@latex@error{\noexpand#1undefined}\@ehc}%
     \relax
  \expandafter\rmo@declmathop\rmo@s{#1}{#2}}
\newcommand\rmo@declmathop[3]{%
  \DeclareRobustCommand{#2}{\qopname\newmcodes@#1{#3}}%
}
\@onlypreamble\RedeclareMathOperator
\makeatother

\renewcommand{\phi}{\varphi}

\DeclareMathOperator{\B}{\mathbb{B}}
\DeclareMathOperator{\R}{\mathbb{R}}
\RedeclareMathOperator{\S}{\mathbb{S}}
\RedeclareMathOperator{\H}{\mathbb{H}}

\DeclareMathOperator*{\bd}{bd}
\DeclareMathOperator*{\conv}{conv}
\DeclareMathOperator{\interior}{int} 
\DeclareMathOperator{\relinterior}{rel int}
\DeclareMathOperator{\vol}{vol}
 
\DeclareMathOperator{\as}{as}
\DeclareMathOperator{\flag}{flag}
\RedeclareMathOperator{\vert}{vert}
\DeclareMathOperator{\face}{face}
\DeclareMathOperator{\id}{id}

\renewcommand{\d}{\,\mathrm{d}}

\begin{document}

\title{Flag Numbers and Floating Bodies}

\author{Florian Besau \and Carsten Schütt \and Elisabeth M.~Werner}

\date{}

\maketitle

\address{F.~Besau: Goethe-Universit\"at Frankfurt; \email{besau@math.uni-frankfurt.de} 
\and C.~Sch\"utt: Christian-Albrechts Universit\"at; \email{schuett@math.uni-kiel.de}
\and E.~M.~Werner: Case Western Reserve University; \email{elisabeth.werner@case.edu}}


\begin{abstract}
	We investigate weighted floating bodies of polytopes. We show that the weighted volume depends on the complete flags of the polytope.
	This connection is obtained by introducing flag simplices, which translate between the metric and combinatorial structure.
	
	Our results are applied in spherical and hyperbolic space. This leads to new asymptotic results for polytopes in these spaces. We also provide explicit examples of spherical and hyperbolic convex bodies whose floating bodies behave completely different from any convex body in Euclidean space.

\end{abstract}

A convex body $K$ in $\mathbb R^{n}$ is a compact, convex set with
nonempty interior. A \emph{convex floating body} $K_{\delta}$ is obtained from $K$ by cutting off all caps of a fixed volume $\delta>0$ \cite{BarLar:1988,SchuettWerner:1990}. 
Blaschke \cite{Blaschke:1923} proved a fundamental relation between the volume of floating bodies and the classical \emph{affine surface area} $\as(K)$ in dimension $2$ and $3$, which was later generalized by Leichtweiss \cite{Leichtweiss:1986a} to higher dimensions, for sufficiently smooth convex bodies $K$. By using the convex floating body $K_\delta$, in \cite{SchuettWerner:1990} the following extension of this relation to all convex bodies $K$ was established,
\begin{equation}\label{eqn:blaschke}
	\lim_{\delta\to 0^+}\frac{\vol_n(K)-\vol_{n}(K_{\delta})}{\delta^{\frac{2}{n+1}}} = c_n \, \as(K),
\end{equation}
where $c_n=(2\pi)^{\frac{1-n}{1+n}}\, \Gamma(\frac{n+3}{2})^{\frac{2}{1+n}}$.
Far reaching generalizations of affine surface areas have since been introduced, which include the $L_p$ affine surface areas \cite{Lutwak:1996, MeyerWerner:2000, Werner:2012, Zhao:2016} and more recently Orlicz affine surface areas \cite{HabPar:2014, LudRei:2010} (see also \cite{CagYe:2016,Ye:2016}). Furthermore, functional analogous were introduced \cite{AAKSW:2012, CFGLSW:2016, LSW:2017} and extensions of the convex floating body and \eqref{eqn:blaschke} were established in spherical \cite{BesWer:2015} and hyperbolic spaces \cite{BesWer:2016}. See also \cite{DKY:2018} for another recent success in extending Euclidean convex geometry to spherical and hyperbolic spaces. In  \cite{Werner:2002}, and also in \cite{BLW:2017}, \eqref{eqn:blaschke} was extended to weighted floating bodies, and a unifying framework to translate between different constant curvature spaces was introduced. Other important relatives of the floating body were considered in \cite{MordhorstWerner:2017, MordhorstWerner:2017a, SchuettWerner:2004}.

\smallskip
The volume derivative \eqref{eqn:blaschke} vanishes for polytopes and in \cite{BarLar:1988} the correct asymptotic order was determined. 
In \cite{Schuett:1991} it was shown that
\begin{equation}\label{eqn:deriv_inner_poly}
	\lim_{\delta\to 0^+} \frac{\vol_{n}(P)-\vol_{n}(P_{\delta})}{\delta\left(\ln\frac{1}{\delta}\right)^{n-1}} = \frac{\left|\flag(P)\right|}{n!\,n^{n-1}},
\end{equation}
where $\left|\flag(P)\right|$ is the \emph{total number of complete flags} (or towers) of $P$ (see Section 2.1 for the definition).

\medskip
In this article we generalize \eqref{eqn:deriv_inner_poly} to a weighted setting. Our main results are stated in the next section.
Our generalizations provide extensions of \eqref{eqn:deriv_inner_poly} to constant curvature spaces, in particular for $n$-dimensional spherical (Section 5.1) and hyperbolic space (Section 5.2). We also provide examples in spherical space and the hyperbolic plane, of convex bodies that realize asymptotic behavior of order $\delta$,  a behavior that is impossible for Euclidean convex bodies.

A major tool in our proofs is the concept of flag simplices which we introduce  in  Section 2. It is connected  with simplex subdivisions of convex polytopes, and should prove useful in other contexts. In particular, we believe that generalizations of bounds on the approximation of convex bodies by polytopes \cite{Schuett:1991}, and asymptotic results on random approximation of polytopes \cite{BarBu:1993} are now well within reach.

In Subsection 2.3 we give a brief survey on lower and upper bounds of the total number of complete flags. It is very intriguing, that the conjectured 
(by Kalai \cite{Kalai:1989}) minimizers  for the total number of complete flags are exactly the same as in Mahler's conjecture for the volume product. We hope that our generalization of \eqref{eqn:deriv_inner_poly} may shed new light on possible connections between the volume product of polytopes and the total number of flags.

\section{Statement of the Main Results}

Let $P$ be an $n$-polytope in $\R^n$, that is, $P$ is the convex hull of a finite number of points and such that $P$ has non-empty interior.
For continuous functions $\phi,\psi:P\to (0,\infty)$ we denote by $\Phi$, respectively $\Psi$, the measure with density $\phi$, respectively $\psi$, i.e., $\Phi(A) = \int_A \phi$ and $\Psi(A)=\int_A \psi$ for every Borel $A\subset P$.
The \emph{weighted floating body} is defined in \cite{Werner:2002} by
\begin{equation}\label{eqn:float_def_mu}
	P_\delta^\phi  = \bigcap \big\{ H^- : \Phi(P\cap H^+)\leq \delta \big\},
\end{equation}
for every sufficiently small $\delta>0$ (cf.\ \cite[Eq.\ (3)]{BLW:2017}).
Clearly, if $\phi\equiv 1$, then the weighted floating body is the convex floating body, i.e., $P_\delta^\phi=P_\delta$.

Our main theorem is the following.
\begin{theorem}\label{thm:main_theorem}
	Let $P$ be an $n$-dimensional convex polytope in $\R^n$ and let $\phi,\psi:P\to (0,\infty)$ be continuous functions. Then
	\begin{equation}\label{eqn:deriv_inner_poly_weighted}
		\lim_{\delta\to 0^+} \frac{\Psi(P)-\Psi(P_\delta^\phi)}{\delta \left(\ln\frac{1}{\delta}\right)^{n-1}} 
		= \sum_{v\in\vert P} \frac{\psi(v)}{\phi(v)} \frac{\left|\flag_v(P)\right|}{n!\,n^{n-1}},
	\end{equation}
	where $\vert P$ is the set of vertices of $P$ and $\left|\flag_v(P)\right|$ is the number of complete flags that have $v$ as a vertex.
\end{theorem}
We will prove this theorem in Section 4.  In the next section we recall some basic facts from convex geometry and introduce the notion of \emph{flag simplices}. In Section 3 we  review parts of the proof \cite{Schuett:1991} of \eqref{eqn:deriv_inner_poly} and establish concentration results with respect to flag simplices. Finally, in Section 4 we introduce weight functions and employ the properties of flag simplices obtained in Section 2 to derive Theorem \ref{thm:main_theorem}.

The following special case of Theorem \ref{thm:main_theorem} is of particular interest:
\begin{corollary}\label{cor:main_theorem}
	Let $P$ be an $n$-dimensional convex polytope in $\R^n$ and let $\phi:P\to (0,\infty)$ be a continuous function. Then
	\begin{equation}\label{eqn:deriv_inner_poly_weighted_2}
		\lim_{\delta\to 0^+} \frac{\Phi(P) - \Phi(P_\delta^\phi)}{\delta \left(\ln\frac{1}{\delta}\right)^{n-1}} = \frac{\left|\flag(P)\right|}{n!\,n^{n-1}},
	\end{equation}
	where $\left|\flag(P)\right|$ is the total number of complete flags of $P$.
\end{corollary}
It is remarkable that the volume derivative of the weighted floating body, i.e., the left hand side of \eqref{eqn:deriv_inner_poly_weighted_2}, is  independent of the actual weight function. This  does not happen for the volume derivative of the weighted floating body for general convex bodies (cf.\ \cite [Thm.\ 5]{Werner:2002} and \cite[Thm.\ 1.1]{BLW:2017}). In view of \eqref{eqn:deriv_inner_poly_weighted}, one explanation for this behavior in the polytopal case might be, that the ``curvature'' is concentrated at the vertices of the polytope.

\smallskip
In Section 5 we apply Corollary \ref{cor:main_theorem} to spherical and hyperbolic space and derive the following theorems. There $\vol_n^s$, respectively  $\vol_n^h$,   denote the spherical, respectively hyperbolic, volume and $P_\delta^s$, respectively $P_\delta^h$, is the spherical, respectively hyperbolic, analogue of the convex floating body (see Section 5 for details).
\begin{theorem}\label{cor:main_theorem_sphere}
	If $P$ is a spherically convex polytope contained in an open halfsphere of the Euclidean unit sphere $\S^n\subset\R^{n+1}$, then
	\begin{equation}
		\lim_{\delta\to 0^+} \frac{\vol^s_n(P)- \vol_n^s(P_\delta^s)}{\delta\left(\ln\frac{1}{\delta}\right)^{n-1}} 
		= \frac{\left|\flag(P)\right|}{n!\,n^{n-1}}.
	\end{equation}
\end{theorem}
\begin{theorem}\label{cor:main_theorem_hyperbolic}
	If $P$ is a compact geodesically convex polytope in hyperbolic $n$-space $\H^n$, then
	\begin{equation}
		\lim_{\delta\to 0^+} \frac{\vol^h_n(P)- \vol_n^h(P_\delta^h)}{\delta\left(\ln\frac{1}{\delta}\right)^{n-1}} 
		= \frac{\left|\flag(P)\right|}{n!\,n^{n-1}}.
	\end{equation}
\end{theorem}
In Section 5 we also give examples in spherical space and in the hyperbolic plane for convex subsets where the volume difference between the set and its floating body is of order $\delta$, something that is not possible in Euclidean space.

\section{Convex Polytopes and Flag Simplices}

An $n$-polytope $P$ is the convex hull of a finite number of points in $\R^n$ and such that $P$ has non-empty interior. We denote by $\mathcal{P}(\R^n)$ the space of $n$-polytopes in $\R^n$. As general reference on convex bodies we refer to \cite{Gardner:2006, Gruber:2007,Schneider:2014} and for references on convex polytopes we may recommend \cite{Ewald:1996, Grunbaum:2003, Ziegler:1995}.

\subsection{Faces and flags of a polytope}

In the following we write $\conv A$ for the \emph{convex hull} of $A\subset \R^n$ and the convex hull of a finite number of points $v_0,\dotsc,v_m\in \R^n$ is
\begin{equation*}
	[v_0,\dotsc, v_m] := \conv \{v_0,\dotsc,v_m\}.
\end{equation*}
In particular, $[x,y]$ will denote the closed affine segment spanned by $x,y\in \R^n$.
The standard orthonormal basis in $\R^n$ is denoted by $e_1,\dotsc,e_n$ and for convenience we set $e_0:=0$. The \emph{standard simplex} $T_n$ is defined by $T_n:=[e_0,\dotsc,e_n]$.

We recall the following basic definitions (cf.\ \cite[Sec.~2.1]{Schneider:2014}, \cite[Ch.~3]{Grunbaum:2003}, \cite[Ch.~2]{Ziegler:1995} or \cite[Sec.~I.4]{Ewald:1996}).
\begin{definition}[faces and complete flags]
	For $P\in\mathcal{P}(\R^n)$,  a \emph{face $F$ of $P$} is the intersection of $P$ with a supporting hyperplane $H$, i.e., $F=P\cap H$.
	We call $F$ a $i$-face, $i=0,\ldots,n-1$, of $P$ if $F$ spans a $i$-dimensional affine subspace. 
	The empty set and $P$ are  called  \emph{improper faces} of $P$ with dimension $-1$, respectively $n$.
	The set of all faces, respectively $i$-faces, of $P$ is denoted by $\face(P)$, respectively $\face_i(P)$.
	
	A \emph{complete flag of $P$} is an ordered sequence $\mathbf{F}=(F_0,\dotsc,F_{n-1})$ of faces $F_0\subset \dotsc \subset F_{n-1}$ with $F_i\in\face_i(P)$ for $i=0,\ldots,n-1$.
	The set of complete flags of $P$ is denoted by $\flag(P)$.
\end{definition}

We write $\vert(P)$ for the set of all vertices. Of course $\face_0(P) = \{ \{v\} : v\in \vert(P)\}$.
Given a fixed vertex $v\in\vert P$, we denote by $\face_i(P;v)$ the subset of $i$-faces of $P$ that contain $v$.
As usual we call $1$-faces edges and $(n-1)$-faces facets. 

The set $\face(P)$ of all proper and improper faces of $P$ is called the \emph{face-lattice}. It is a graded lattice with respect to the partial order induced by set inclusion on the faces and the rank function is given by the dimension of the face.
Note that what we call flag is always a \emph{complete flag} of $P$. Those  are also called \emph{towers} or \emph{full flags} in the literature, and they correspond to maximal chains in the face lattice.

\begin{remark}[$f$-vector and flag vector]
The $f$-vector 
$(f_0(P),\dotsc,f_{n-1}(P))\in\mathbb{N}^{n}$,
where $f_i(P)=\left|\face_i(P)\right|$,  denotes the number of all $i$-faces of $P$, is one of the most important notions in polyhedral combinatorics. Many questions on characterizing the set of all possible $f$-vectors of $n$-polytopes are still open. The two-dimensional case, $n=2$, is trivial and for $n=3$ a complete characterization was obtained by Steinitz in 1906 \cite{Steinitz:1906}. A characterization for the $f$-vector of simplicial $n$-polytopes, that is, polytopes whose facets are simplices, was conjectured by McMullen in 1971 \cite{McMullen:1971},  the so-called \emph{$g$-conjecture}. The proof of McMullen's conjecture was established by Billera and Lee \cite{BilleraLee:1980} and Stanley \cite{Stanley:1980}--a crowning achievement. The cases $n\geq 4$ are in general still open, although there is much progress in the lower dimensional cases, in particular for $n=4$.  See \cite{SjobergZiegler:2017,SjobergZiegler:2018} and the references therein.

To attack the higher dimensional cases a generalization of the $f$-vector has been introduced \cite{BayerBillera:1985}, the \textit{flag vector}, with entries $f_S(P)$. Here  $S=\{i_0,\dotsc,i_s\}\subset \{0,1,\dotsc,n-1\}:=[n]$, such that $i_0<\dotsc<i_s$, is the \emph{rank set} and $f_S(P)$ is the number of all sequences of faces $F_{0}\subset \dotsc \subset F_{s}$ such that $F_{j}\in\face_{i_j}(P)$ for $j=0,\dotsc,s$. Clearly $f_{\{i\}}(P)=f_i(P)$ and $f_{[n]}(P)=\left|\flag(P)\right|$ is exactly the number of complete flags.

Note that $\left|\flag(P)\right|=f_{[n]}(P)$ can be expressed by flag numbers with lower rank by the generalized Dehn--Sommerville equations, which were established by Bayer and Billera \cite{BayerBillera:1985}. In particular, for dimension $n=2$, we have $\left|\flag(P)\right|=2f_0(P)=2f_1(P)$, for $n=3$ we have $\left|\flag(P)\right|=4f_1(P)$, and for $n=4$ we have $\left|\flag(P)\right|= 4 f_{02}(P)$ (cf.\ \cite{Bayer:1987}).
Also, for every simplicial $n$-polytope $P$ we have $\left|\flag(P)\right| = n! f_{n-1}(P)$, and for every simple $n$-polytope $P$ we have $\left|\flag(P)\right|=n!f_0(P)$.
\end{remark}

\medskip
The following lemma is an easy exercise, but we include a short proof for the reader's convenience. It shows that the faces of a complete flag stack in a convex way.
\begin{lemma}\label{lem:flag_dissection}
	Let $P\in\mathcal{P}(\R^n)$, $(F_0,\dotsc,F_{n-1})\in\flag(P)$ and set $F_n:= P$. If $I\subset \{0,\dotsc,n\}$, then 
	\begin{equation}\label{eqn:dissect}
		\conv\left(\bigcup_{i\in I} \relinterior F_i\right) = \biguplus_{i\in I} \relinterior F_i,
	\end{equation}
	where by $\uplus$ we indicate that the sets are pairwise disjoint.
\end{lemma}
\begin{proof}
	Proof by induction on $k:=|I|$. The case $k=1$ is trivial. Assume that the statement holds true for $k-1$.
	Let $I\subset \{0,\dotsc,n\}$ with $|I| = k$. Set $i_0:=\max \{i:i\in I\}$ and $J:=I\setminus\{i_0\}$. By the induction hypothesis, since $|J|=|I|-1=k-1$, we have that $C:=\biguplus_{j\in J}\relinterior F_j$ is convex.
	Since
	\begin{equation*}
		F_i\subset F_{i+1}\setminus \relinterior F_{i+1},
	\end{equation*}
	for all $i=0,\dotsc,n-1$, we conclude $\relinterior F_i\subset F_{i_0} \setminus \relinterior F_{i_0}$ for all $i\in J$, since $i<i_0$.
	This yields $C\cap \relinterior F_{i_0} = \emptyset$ and $C\cup \relinterior F_{i_0} \subset F_{i_0}$. Since $C$ and $\relinterior F_{i_0}$ are convex, we conclude
	\begin{equation*}
		\conv(C\cup \relinterior F_{i_0}) = C\uplus \relinterior F_{i_0} = \biguplus_{i\in I} \relinterior F_i. \qedhere
	\end{equation*}
\end{proof}

\subsection{Flag simplices}
In this subsection we develop the notion of flag simplices.
Let us start with an introductory example: the \emph{barycenter subdivision} of a convex polytope $P\in\mathcal{P}(\R^n)$.
This subdivision generates a simplicial complex associated with $P$, i.e., it gives a set of $n$-simplices $S_i$, $i=1,\dotsc,m$, such that 
\begin{enumerate}
	\item[i)] $S_i\cap S_j$ is either empty or a common face of both simplices, and
	\item[ii)] $P=\bigcup \{S_i:i=1,\dotsc,m\}$.
\end{enumerate}
In particular, i) implies that $\interior S_i\cap \interior S_j=\emptyset$ if $i\neq j$. One may define the barycenter subdivision inductively as follows:
If $n=1$, then $P$ is just a segment $[a,b]\subset \R$ and the barycenter subdivision is given by the segments $[a,c]$ and $[c,b]$ where $c=\frac{a+b}{2}$ is the barycenter (midpoint) of the segment.

Now if $n>1$, we again set $c$ as the barycenter (centroid) of $P$ and for any facet of $P$ we apply the barycenter subdivision in the corresponding affine hyperplane. Then the barycenter subdivision of $P$ is the set of all $n$-simplices that are obtained as the convex hull of $c$ and the $(n-1)$-simplices that decompose the facets (see Figure \ref{fig:barycenter}). The important observation here is, that there is a bijection between 
the simplices constructed in this way and 
 the complete flags of $P$.  In particular,  $m=\left|\flag(P)\right|$.

\begin{figure}[t]
	\centering
	\begin{tikzpicture}
		\draw[thick, fill=black!10] (0,0) -- (3,1) -- (2,4) -- (-2,4) -- (-3,3) -- cycle;
		\draw (0,0) -- (0.21,2.29) -- (3,1);
		\draw (2,4) -- (0.21,2.29) -- (-2,4);
		\draw (-3,3) -- (0.21,2.29);
		\draw (1.5,0.5) -- (0.21,2.29) -- (2.5,2.5);
		\draw (0,4) -- (0.21,2.29) -- (-2.5,3.5);
		\draw (-1.5,1.5) -- (0.21,2.29);
	\end{tikzpicture}
	\caption{Barycenter subdivision of a convex polygon.}
	\label{fig:barycenter}
\end{figure}
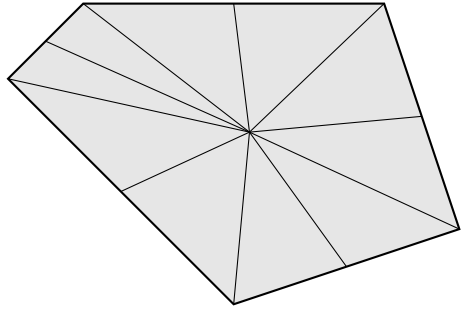

Of course, instead of the barycenter one may consider any sequence of interior points in the $i$-dimensional faces of $P$ and obtain another subdivision of $P$. To be more precise, for an $n$-polytope the barycenter subdivision is determined by the centroids in the faces, i.e., the map $\mathcal{C}^P:\face(P)\to P$, where $\mathcal{C}^P(F)$ is the centroid of the proper face $F$, completely describes the subdivision ($\mathcal{C}^P(\emptyset)$ may be chosen arbitrarily). The simplices in the subdivision are determined by
\begin{equation}\label{eqn:simplex_sub}
	S(\mathcal{C}^P,\mathbf{F}) = [\mathcal{C}^P(F_0),\dotsc,\mathcal{C}^P(F_{n-1}),\mathcal{C}^P(P)],
\end{equation}
for $\mathbf{F}=(F_0,\dotsc,F_{n-1})\in \flag(P)$. In the same way, any map $\mathcal{S}^P:\face(P)\to P$ with the property that $\mathcal{S}^P(F)\in\relinterior F$ for all $F\in\face(P)$, $F\neq \emptyset$, determines a subdivision of $P$ into a simplicial complex where the simplices $S(\mathcal{S}^P,\mathbf{F})$ are determined by \eqref{eqn:simplex_sub} and enumerated by the flags $\mathbf{F}$ of $P$. 
The barycenter subdivision $\mathcal{C}^P$ and,  more generally,  subdivisions $\mathcal{S}^P$ are  particular  triangulations of $P$.  An extensive exposition on general triangulations and subdivision can be found in the book \cite{DeLRS:2010}.
The simplices that appear in barycenter-type subdivisions $\mathcal{S}^P$ are the main inspiration behind the following definition.

\begin{definition}[flag simplex]
	For $P\in\mathcal{P}(\R^n)$ we call a simplex $S=[v_0,\dotsc,v_n]$ a \emph{flag simplex of $P$}, if and only if,
	there is a flag $(F_0,\dotsc,F_{n-1})\in\flag(P)$ and a permutation $\sigma$ of $\{0,\dotsc, n\}$ such that
	\begin{equation*}
		v_{\sigma(i)} \in \relinterior F_i,
	\end{equation*}
	for $i=0,\dotsc,n-1$ and $v_{\sigma(n)} \in \interior P$.
\end{definition}

Since an $n$-polytope $P$ is dissected by the relative interior of its faces, i.e., $P=\biguplus \{ \relinterior F : F\in\face(P)\}$ (cf.~\cite[Thm.~2.1.2]{Schneider:2014}), for every vertex $v_i$ of the flag simplex, there is a uniquely determined face $F_i$ such that $v_i\in\relinterior F_i$. Thus, if $S$ is a flag simplex of $P$, then the associated flag $(F_0,\dotsc,F_{n-1})$ and the permutation $\sigma$ are uniquely determined. 
So,  given a flag simplex $S$,  we call $(F_0,\dotsc, F_{n-1})$ the \emph{flag determined by $S$} and in the sequel we always assume that the vertices of a flag simplex $S=[v_0,\dotsc,v_n]$ are ordered, such that 
\begin{equation*}
	v_i\in\relinterior F_i,\quad \text{for $i=0,\dotsc,n-1$,}
\end{equation*}
and $v_n\in \interior P$ (see Figure \ref{fig:flag_simplex}).

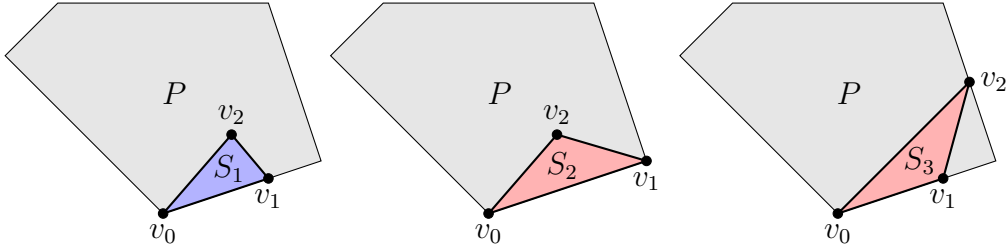
\begin{figure}[t]
	\centering
	\hfill
	\begin{tikzpicture}[scale=0.7]
		\draw[fill=black!10] (0,0) -- (3,1) -- (2,4) -- (-2,4) -- (-3,3) -- cycle;
		\draw[thick, fill=blue!30] (0,0) -- (2,{2/3}) -- (1.3,1.5)--cycle;
		\draw[fill=black] (0,0) circle(0.09);
		\node[below] at (0,0) {$v_0$};
		\draw[fill=black] (2,{2/3}) circle(0.09);
		\node[below] at (2,{2/3}) {$v_1$};
		\draw[fill=black] (1.3,1.5) circle(0.09);
		\node[above] at (1.3,1.5) {$v_2$};
		\node at (0.21,2.29) {$P$};
		\node at (1.25,{5/9+0.3}) {$S_1$};
	\end{tikzpicture}\hfill
	\begin{tikzpicture}[scale=0.7]
		\draw[fill=black!10] (0,0) -- (3,1) -- (2,4) -- (-2,4) -- (-3,3) -- cycle;
		\draw[thick, fill=red!30] (0,0) -- (3,1) -- (1.3,1.5)--cycle;
		\draw[fill=black] (0,0) circle(0.09);
		\node[below] at (0,0) {$v_0$};
		\draw[fill=black] (3,1) circle(0.09);
		\node[below] at (3,1) {$v_1$};
		\draw[fill=black] (1.3,1.5) circle(0.09);
		\node[above] at (1.3,1.5) {$v_2$};
		\node at (0.21,2.29) {$P$};
		\node at (1.4,{5/9+0.35}) {$S_2$};
	\end{tikzpicture}\hfill
	\begin{tikzpicture}[scale=0.7]
		\draw[fill=black!10] (0,0) -- (3,1) -- (2,4) -- (-2,4) -- (-3,3) -- cycle;
		\draw[thick, fill=red!30] (0,0) -- (2,{2/3}) -- (2.5,2.5)--cycle;
		\draw[fill=black] (0,0) circle(0.09);
		\node[below] at (0,0) {$v_0$};
		\draw[fill=black] (2,{2/3}) circle(0.09);
		\node[below] at (2,{2/3}) {$v_1$};
		\draw[fill=black] (2.5,2.5) circle(0.09);
		\node[right] at (2.5,2.5) {$v_2$};
		\node at (0.21,2.29) {$P$};
		\node at (1.55,1.) {$S_3$};
	\end{tikzpicture}\hfill $ $
	\caption{The $2$-simplex $S_1$ is a flag simplex of the polygon $P$, but $S_2$ and $S_3$ are not.}
	\label{fig:flag_simplex}
\end{figure}

Flag simplices may also be characterized as $n$-simplices that meet the relative interior of a uniquely determined flag.
\begin{proposition}\label{prop:flag simplex}
	An $n$-simplex $S\subset P$ is a flag-simplex of $P$ if and only if 
	\begin{enumerate}
		\item[i)] for every $i\in\{0,\dotsc,n-1\}$ there is a unique $i$-face $F_i^*$ such that $\relinterior F_i^*\cap S\neq \emptyset$, and
		\item[ii)] $(F_0^*,\dotsc,F_{n-1}^*)\in\flag(P)$.
	\end{enumerate}
	In particular, if $S=[v_0,\dotsc,v_n]$ is a flag simplex associated with the flag $(F_0,\dotsc,$ $F_{n-1})\in\flag(P)$ and such that $v_i\in\relinterior F_i$ for $i=0,\dotsc n$, where $F_n:=P$, then
	\begin{equation}\label{eqn:simplex-flag}
		[v_0,\dotsc,v_i] \subset \biguplus_{j=0}^i \relinterior F_j
	\end{equation}
	for $i=0,\dotsc,n$.
\end{proposition}
\begin{proof}
	Let $S=[v_0,\dotsc,v_n]$ be a flag simplex with flag $(F_0,\dotsc,F_{n-1})\in\flag(P)$. Then $v_i \in S\cap \relinterior F_i$ for $i=0,\dotsc,n-1$ and $v_n\in \interior P$. Thus, by Lemma \ref{lem:flag_dissection},
	\begin{equation*}
		S=[v_0,\dotsc,v_n] \subset \conv\left(\interior P \cup \bigcup_{i=0}^{n-1} \relinterior F_i\right) 
		= \interior P\uplus \biguplus_{i=0}^n \relinterior F_i,
	\end{equation*}
	and therefore $S$ only meets the relative interior of the $i$-face $F_i$ for $i=0,\dotsc,n-1$.
	
	For the converse assume that $S\subset P$ is an $n$-simplex spanned by the vertices $v_0,\dotsc,v_n$ and assume that $S$ satisfies the conditions i) and ii). We set $F_n^*=P$. By induction we  show that there is a permutation $\sigma$ of $\{0,\dotsc,n\}$ such that $v_{\sigma(i)}\in\relinterior F_i^*$, for $i=0,\dotsc,n-1$ and $v_{\sigma(n)}\in \interior P$, as  by Lemma \ref{lem:flag_dissection}, 
	\begin{equation*}
		[v_{\sigma(0)},\dotsc,v_{\sigma(i)}] \subset \conv(\bigcup_{j=0}^i \relinterior F_j^*) = \biguplus_{j=0}^i \relinterior F_j^*,
	\end{equation*}
	for $i=0,\dotsc,n$. Hence, $S$ is a flag simplex with flag $(F_0^*,\dotsc,F_{n-1}^*)$.
\end{proof}

Notice that \eqref{eqn:simplex-flag} implies that for a flag simplex $S=[v_0,\dotsc,v_n]$, with associated $(F_0,\dotsc,F_{n-1})\in\flag(P)$, 
the faces of $S$ defined by
\begin{equation*}
	S_i := [v_0,\dotsc,v_i], \quad \text{for $i=0,\dotsc,n-1$},
\end{equation*}
are $i$-simplices such that 
\begin{enumerate}
	\item[i)] $\relinterior S_i \subset \relinterior F_i$, for $i=0,\dotsc,n-1$, and
	\item[ii)] $(S_0,\dotsc,S_{n-1})\in\flag(S)$.
\end{enumerate}

In the next proposition we establish two important properties for the set of all flag simplices associated with a fixed flag of $P$.
\begin{proposition}\label{prop:flag simplices}
	Let $P\in\mathcal{P}(\R^n)$ and $\mathbf{F}\in\flag(P)$.
	\begin{enumerate}
		\item[i)] 
		If $S_1$ and $S_2$ are two flag simplices of $P$ associated with $\mathbf{F}$, then there exists another flag simplex $S$ of $P$ such that $S \subset S_1\cap S_2$.
		\item[ii)] If $S$ is a flag simplex of $P$ associated with $\mathbf{F}$, then there exists a flag simplex $T$ of $P$ with $S\subset T\subset P$ and $S$ is a flag simplex of $T$.
	\end{enumerate}
\end{proposition}
\begin{proof}
	\emph{i)} Set $F_n=P$. For $j\in\{1,2\}$ let $S_j=[v_0^j,\dotsc, v_n^j]$ be a flag simplex such that $v_i^j\in \relinterior F_i$ for $i=0,\dotsc,n$.
	We show by induction that for $i=0,\dotsc,n$ we may choose
	\begin{equation*}
		v_i \in \relinterior\, [v_0,\dotsc,v_{i-1},v_i^1] \cap \relinterior\, [v_0,\dotsc,v_{i-1},v_i^2]\subset \relinterior F_i.
	\end{equation*}
	The statement is trivial for $i=0$.
	Assume that the statement holds true for $i-1$. Then
	\begin{equation*}
		\relinterior\, [v_0,\dotsc,v_{i-1},v_i^1]\cap \relinterior[v_0,\dotsc,v_{i-1},v_i^2] \neq \emptyset,
	\end{equation*}
	because $[v_0,\dotsc,v_{i-1}]$ is a $(i-1)$-simplex in $F_{i-1} \subset F_i\setminus \relinterior F_i$ and $v_i^1,v_i^2\in\relinterior F_i$.
	We define $S:=[v_0,\dotsc,v_n]$. Then $S$ is a flag simplex associated with $\mathbf{F}$ and $S\subset S_1\cap S_2$.
	
	\smallskip
	\emph {ii)} 
	Let $S=[v_0,\dotsc,v_n]$ be associated with the flag $(F_0,\dotsc,F_{n-1})\in\flag(P)$.
	We prove the statement by induction on $n$: The case $n=1$ is obvious.
	Assume that the statement holds true in dimension $n-1$. 
	Then $S':=[v_0,\dotsc,v_{n-1}] = S\cap F_{n-1}$ is a flag simplex of $F_{n-1}$ associated with the flag $(F_0,\dotsc,F_{n-1})$.
	Hence, by the induction hypothesis there exists a flag simplex $T'=[w_0,\dotsc,w_{n-1}]$ such that $S'\subset T'\subset F_{n-1}$. 
	Since $S'$ is a flag simplex of $T'$ we have $v_i\in\relinterior\, [w_0,\dotsc,w_i]\in \face_{i} T'$, for $i=0,\dotsc,n-1$.
	
	Since $v_n\in\interior P$ there is $\varepsilon>0$ such that the closed ball $B(v_n,\varepsilon)$ is contained in $\interior P$. Fix $x'\in \interior S'$ and set $w_n := x' + (1+\varepsilon) (v_n-x')$. We set
	\begin{equation*}
		T:=\conv\left(\{w_n\}\cup T'\right) = [w_0,\dotsc,w_n].
	\end{equation*}
	Then $w_n\in B(v_n,\varepsilon)\subset\interior P$ and since also $w_i\in\relinterior F_i$ for $i=0,\dotsc,n-1$, we see that $T$ is flag simplex of $P$ associated with $\mathbf{F}$ and $S = \conv(S'\cup\{v_n\}) \subset \conv(T'\cup\{v_n\})\subset T$. 
	
	To finish the proof, we only need to verify that $S$ is a flag simplex of $T$. 
	Set $T_i=[w_0,\dotsc,w_i]$ for $i=0,\dotsc,n-1$. Then $(T_0,\dotsc,T_{n-1})\in\flag(T)$ and $v_i\in\relinterior\, T_i$ for $i=0,\dotsc,n-1$ by the induction hypothesis for $T'$. Thus, we only need to show that $v_n\in \interior T$.
	Since $x'\in\relinterior S'$ there is $\delta>0$ such that $B':=B(x',\delta)\cap F_{n-1}\subset S'$ and therefore
	\begin{equation*}
		v_n\in \interior \conv\left(\{w_n\}\cup B'\right).
	\end{equation*}
	Since also $\conv\left(\{w_n\}\cup B'\right)\subset T$, we conclude $v_n\in\interior T$.
\end{proof}

\medskip
If the origin is in the interior of $P\in\mathcal{P}(\R^n)$, then the \emph{polar body} $P^\circ$ is defined by
\begin{equation*}
	P^\circ = \{y\in \R^n :\text{$x\cdot y \leq 1$ for all $x\in P$}\}.
\end{equation*}
The polar body of an $n$-polytope is again an $n$-polytope, that is $P^\circ\in\mathcal{P}(\R^n)$.
If $P\subset Q$ and $0\in\interior P$, then $P^\circ \supset Q^\circ$, i.e., the polar map reverses inclusion.
For $F\in\face(P)$ the \emph{conjugate face} $\widehat{F}$ is defined by
\begin{equation*}
	\widehat{F} = \{x\in P^\circ:\text{$x\cdot y = 1$ for all $y\in F$}\}\in\face(P^\circ).
\end{equation*}
The map \,$\widehat{ }:\face(P)\to \face(P^\circ)$ is an antimorphism, i.e., a bijection that reverses the inclusion relation 
(cf.\ \cite[p.~120]{Schneider:2014}). In particular, if $F$ is a $k$-face of $P$, then $\widehat{F}$ is a $(n-1-k)$-face of $P^\circ$ and we have
$\widehat{\phantom{\rule{0.7pt}{1.85ex}}\hspace*{-0.5ex}\smash{\widehat{F}}}=F$ (cf.\ \cite[Thm.\ 2.1]{Ewald:1996}). Consequently for the face number we have $f_{k}(P)=f_{n-1-k}(P^\circ)$ (cf.\ \cite[Thm.\ 2.5]{Ewald:1996}).
For a flag $\mathbf{F}=(F_0,\dotsc,F_{n-1})\in\flag(P)$ we define the \emph{conjugate flag} by
\begin{equation*}
	\widehat{\mathbf{F}} = (\widehat{F}_{n-1}, \dotsc, \widehat{F}_0)\in\flag(P^\circ).
\end{equation*}

Let $S$ be a flag simplex of $P$ associated with $(F_0, \dotsc, F_{n-1})\in\flag(P)$.
We label the vertices of $S$ by $v_0,\dotsc,v_n$ such that $v_i\in \relinterior F_i$. Then $S_i:=[v_0,\dotsc,v_i]\in\face_i(S)$ and, by Proposition \ref{prop:flag simplex},
\begin{equation}\label{eqn:flag simplex_rel}
	S_i=[v_0,\dotsc,v_i]\subset \biguplus_{j=0}^i\relinterior F_j \subset F_i,\quad \text{for $i=0,\dotsc,n$.}
\end{equation}
If $0\in\interior S$, then $S^\circ$ is an $n$-simplex and $S^\circ \supset P^\circ$. 
The $n$-simplex $S^\circ$ is determined by the closed half-spaces $H_i^+:=\{y\in S^\circ: y\cdot v_i\leq 1\}$, i.e., $S^\circ=\bigcap_{i=0}^n H_i^+$.
Furthermore, \eqref{eqn:flag simplex_rel} implies
\begin{equation}\label{eqn:flag simplex_polar}
	\widehat{F}_i \subset \widehat{S}_i = H_0\cap \cdots \cap H_i, \quad \text{for all $i=0,\dotsc,n-1$}.
\end{equation}
 
Now if $T$ is a flag simplex of $P^\circ$ associated with $\widehat{\mathbf{F}}$, then $T\subset P^\circ \subset S^\circ$.
Furthermore, we may label the vertices $w_0,\dotsc,w_n$ of $T$ such that 
\begin{equation*}
	w_i\in\relinterior \widehat{F}_{n-1-i} \overset{\eqref{eqn:flag simplex_polar}}{\subset} \relinterior \widehat{S}_{n-1-i}.
\end{equation*}
We see that $T$ is a flag simplex of $S^\circ$ associated with the flag $(\widehat{S}_{n-1},\dotsc,\widehat{S}_0)$.
If $0\in T$, then we may apply polarity once more and find $S\subset P \subset T^\circ$ and $S$ is flag simplex of $T^\circ$ associated with the flag $(\widehat{G}_{n-1},\dotsc,\widehat{G}_0)\in\flag(T^\circ)$, where $G_i=[w_0,\dotsc,w_i]\in\face_i(T)$.

\begin{lemma}\label{lem:upper_simplex}
	If $P\in\mathcal{P}(\R^n)$ and $S$ is a flag simplex of $P$, then there exists an $n$-simplex $T\supset P$ such that $S$ is also a flag simplex of $T$.
\end{lemma}
\begin{proof}
	This follows from the preceding arguments, but can also be proved directly by induction on $n$.
\end{proof}

In the next lemma we show, that given a family of flag simplices there is always a possibly smaller family of flag simplices that is pairwise disjoint.
\begin{lemma}\label{lem:disjoint}
	Let $P\in\mathcal{P}(\R^n)$ and $v\in \vert P$. If $(T(\mathbf{F}))_{\mathbf{F}\in\flag_v(P)}$ is a family of flag simplices such that $T(\mathbf{F})$ determines the flag $\mathbf{F}$, then there exists a family $(S(\mathbf{F}))_{\mathbf{F}\in\flag_v(P)}$ of flag simplices such that
	\begin{enumerate}
		\item[i)] $S(\mathbf{F}) \subset T(\mathbf{F})$ for all $\mathbf{F}\in\flag_v(P)$, and
		\item[ii)] $\interior S(\mathbf{F}) \cap \interior S(\mathbf{F}') = \emptyset$ for all $\mathbf{F},\mathbf{F}'\in \flag_v(P)$, $\mathbf{F}\neq \mathbf{F}'$.
	\end{enumerate}
\end{lemma}
\begin{proof}
	Let $(T'(\mathbf{F}))_{\mathbf{F}\in\flag_v(P)}$ be the sub-family of flag simplices containing $v$ obtained from the barycenter subdivision of $P$.
	Then $\interior T'(\mathbf{F}) \cap \interior T'(\mathbf{F}')=\emptyset$ for all $\mathbf{F},\mathbf{F}'\in \flag_v(P)$, $\mathbf{F}\neq \mathbf{F}'$.
	By Proposition \ref{prop:flag simplices}, for each $\mathbf{F}\in\flag_v(P)$ we may choose a flag simplex $S(\mathbf{F})$ such that $S(\mathbf{F}) \subset T(\mathbf{F})\cap T'(\mathbf{F})$. Then the sequence $(S(\mathbf{F}))_{\mathbf{F}\in\flag_v(P)}$ satisfies i) and ii).
\end{proof}

A \emph{wedge} $W$ is the intersection of two closed half-spaces,  $W=H_0^+\cap H_1^+$. 
The next lemma is  crucial  to conclude that the difference volume of $P$ and $P_\delta$ is concentrated in 
 the flag simplices.
This is proved in the next sections. 

We show that the symmetric difference of two flag simplices $S_1$ and $S_2$ that are associated with the same flag of an $n$-simplex $T$ can be covered by special wedges. It follows from \cite[Lem.\ 1.4]{Schuett:1991} that these wedges can be disregarded and  concentration of volume  in arbitrarily small flag simplices of $T$ will follow.  See Lemma \ref{lem:eucl_first}.
\begin{lemma}\label{lem:flag_simplex_switch}
	Let $T$ be an $n$-dimensional simplex with vertices $z_0,\dotsc,z_n$. Furthermore, let $S_1$ and $S_2$ be two flag simplices of $T$ associated with the flag $\mathbf{F}=(F_0,\dotsc,F_{n-1})$, where
	$F_i=[z_0,\dotsc,z_i]$ for all $i=0,\dotsc, n-1$.
	Then there are wedges $(W_i)_{i=0}^{n-1}$ such that
	\begin{equation*}
		S_1\triangle S_2 \subset \bigcup_{i=0}^{n-1} W_i
	\end{equation*}
	and the half-spaces $H_{i0}^+,H_{i1}^+$ that determine the wedge $W_i=H_{i0}^+\cap H_{i1}^+$, satisfy
	\begin{enumerate}
		\item[i)] $z_i \in (\interior H_{i0}^+)\setminus H_{01}^+$ and $z_{i+1}\in (\interior H_{i1}^+)\setminus H_{i0}^+$,
		\item[ii)] $z_j \in H_{i0}\cap H_{i1}$ for all $j \in \{0,\dotsc,n\}\setminus \{i, i+1\}$,
	\end{enumerate}
	for $i=0,\dotsc,n-1$.
\end{lemma}
\begin{figure}[t]
	\centering
	\hfill
	\begin{tikzpicture}[scale=3.7]
		\path[use as bounding box] (-0.2,-0.2) rectangle (1.2,1.2);
		\node[below] at (1,0) {$z_1$};
		\node[left] at (0,1)  {$z_2$};
		\node[below left] at (0,0) {$z_0$};
		\draw (1,0) -- (0,1);

		\fill[black!20] (0,0) -- ({24/75},{8/75}) -- (0.4,0) -- (0.6,0) -- (0.6,0.2) -- ({24/75},{8/75}) -- (0.1,0.4) --cycle;
		\draw[dashed] (0,0) -- ({24/75},{8/75}) -- (0.4,0);
		\draw (0.6,0) -- (0.6,0.2) -- ({24/75},{8/75}) -- (0.1,0.4) -- (0,0);
		\node[above] at ({24/75+0.05},{8/75+0.1}) {$S_1\triangle S_2$};
		
		\draw[thick] (0,0) -- (1,0) -- (0,1) -- cycle;
	\end{tikzpicture}\hfill
	\begin{tikzpicture}[scale=3.7]
		\path[use as bounding box] (-0.2,-0.2) rectangle (1.2,1.2);
		\node[below] at (1,0) {$z_1$};
		\node[left] at (0,1)  {$z_2$};
		\node[below left] at (0,0) {$z_0$};
		\draw (1,0) -- (0,1);
		
		\fill[black!20] (0,0) -- ({24/75},{8/75}) -- (0.4,0) -- (0.6,0) -- (0.6,0.2) -- ({24/75},{8/75}) -- (0.1,0.4) --cycle;
		\draw[dashed] (0,0) -- ({24/75},{8/75}) -- (0.4,0);
		\draw (0.6,0) -- (0.6,0.2) -- ({24/75},{8/75}) -- (0.1,0.4) -- (0,0);
		
		\draw[dashed] (0,1) -- (0.6,0);
		\draw[dashed] (0,1) -- (0.75,0);
		\draw[dashed] (0,1) -- (0.4,0);
		\draw[dashed] (0,1) -- ({1/6},0);
		\draw[thick] (0,0) -- (1,0) -- (0,1) -- cycle;
	\end{tikzpicture}\hfill $ $\\
	\hfill
	\begin{tikzpicture}[scale=3.7]
		\path[use as bounding box] (-0.2,-0.2) rectangle (1.2,1.2);
		\node[below] at (1,0) {$z_1$};
		\node[left] at (0,1)  {$z_2$};
		\node[below left] at (0,0) {$z_0$};
		\draw (1,0) -- (0,1);
		
		\draw[dashed] (0,0) -- ({24/75},{8/75}) -- (0.4,0);
		\draw (0.6,0) -- (0.6,0.2) -- ({24/75},{8/75}) -- (0.1,0.4) -- (0,0);
		
		\fill[blue,opacity=0.5] (0,1) -- (0.75,0) -- ({1/6},0)--cycle;
		\draw[blue,thick] (0,1) -- (0.75,0) node[below] {$H_{00}$} -- ({1/6},0) node[below] {$H_{01}$} --cycle;
		
		\node[blue,below] at (0.5,0) {$W_0$};
		\draw[thick] (0,0) -- (1,0) -- (0,1) -- cycle;
	\end{tikzpicture}\hfill
	\begin{tikzpicture}[scale=3.7]
		\path[use as bounding box] (-0.2,-0.2) rectangle (1.2,1.2);
		\node[below] at (1,0) {$z_1$};
		\node[left] at (0,1)  {$z_2$};
		\node[below left] at (0,0) {$z_0$};
		\draw (1,0) -- (0,1);
		
		\draw[dashed] (0,0) -- ({24/75},{8/75}) -- (0.4,0);
		\draw (0.6,0) -- (0.6,0.2) -- ({24/75},{8/75}) -- (0.1,0.4) -- (0,0);
		
		\fill[blue,opacity=0.5] (0,1) -- (0.75,0) -- ({1/6},0)--cycle;
		\draw[blue, thick] (0,1) -- (0.75,0) node[below] {$H_{00}$} -- ({1/6},0) node[below] {$H_{01}$} --cycle;
		
		\fill[red,opacity=0.5] (0,0) -- (0.75,0.25) -- (0.2,0.8)--cycle;
		\draw[red,thick] (0,0) -- (0.75,0.25) node[above right]{$H_{11}$} -- (0.2,0.8) node[above right]{$H_{10}$}--cycle;
		
		\node[blue,below] at (0.5,0) {$W_0$};
		\node[red, above right] at (0.5,0.5) {$W_1$};
		\draw[thick] (0,0) -- (1,0) -- (0,1) -- cycle;
	\end{tikzpicture}\hfill $ $
	\caption{Sketches for the proof of Lemma \ref{lem:flag_simplex_switch}. In the first sketch in the upper left we see the symmetric difference of two flag simplices $S_1$ and $S_2$ that determine the same flag. In the second sketch on the upper right we see the projection we use to reduce the dimension. Finally, in the lower two pictures we see the construction of the closed wedges $W_0$ and $W_1$ that cover the symmetric difference $S_1\triangle S_2$.}
\end{figure}
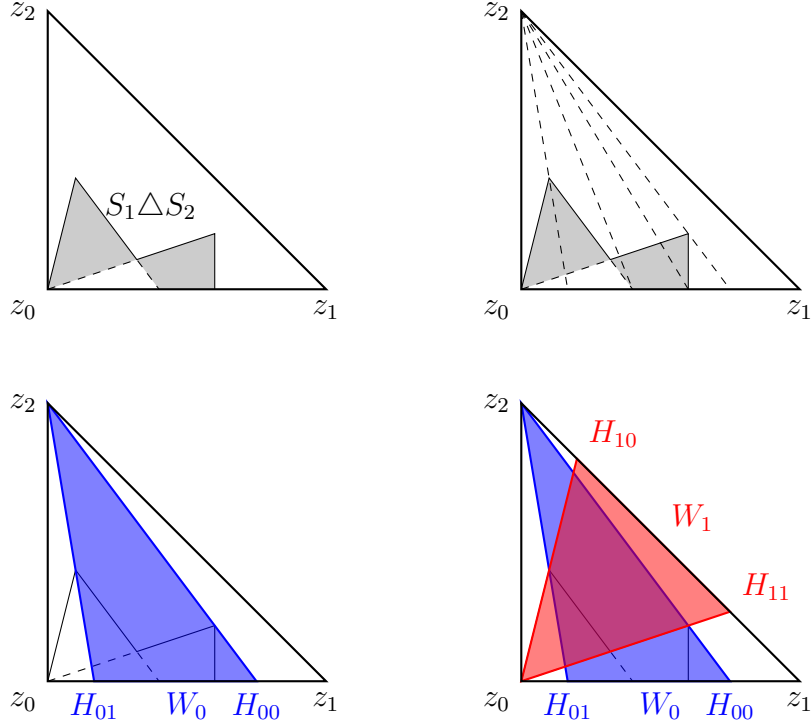
\begin{proof} The proof is done by 
	induction on the dimension $n$. The case $n=1$: $T$, $S_1$, and $S_2$ are closed intervals and we may assume that
	$T=[z_0,z_1]$, $S_1=[z_0,v_1]$, $S_2=[z_0,w_1]$ for $z_0<v_1\leq w_1<z_1$. Set 
	\begin{align*}
		H_{00}^+ &= \left(-\infty, v_1\right], &
		H_{01}^+ &= \left[w_1,\infty\right),
	\end{align*}
	and $W_0 = H_{00}^+\cap H_{01}^+=[v_1, w_1]$. Then
	\begin{align*}
		S_1\triangle S_2 = (v_1,w_1] \subset W_0,
	\end{align*}
	and $H_{00}^+,H_{01}^+$ satisfy the condition i). Condition ii) holds true trivially.
	
	Now let $n\geq 2$ and assume that the statement holds true in dimension $n-1$.
	It will be sufficient to consider the standard simplex $T_n$ with vertices $e_0,\dotsc,e_n$, where we set $e_0:=0$, since any $n$-simplex $T$ can be mapped to $T_n$ by an affine transformation. We may also assume that the flag determined by $S_1$ and $S_2$ is $\mathbf{F}=(F_0,\dotsc,F_{n-1})$ where $F_i = [e_0,e_1,\dotsc, e_i]$.
	
	Let $E$ be the affine $(n-2)$-dimensional plane that is spanned by $e_2,\dotsc,e_n$.
	We consider the projective transformation $\alpha$ that maps a point $(x_1,\dotsc,x_n)\in T_n\setminus E$, i.e., $\sum_{i=2}^n x_i<1$, to the line $L$ spanned by $e_1$ through $E$. By identifying the line $L$ with $\R$, we may write $\alpha$ as
	\begin{equation*}
		\alpha\big((x_1,\dotsc,x_n)\big) = \frac{x_1}{1-\sum_{i=2}^n x_i}.
	\end{equation*}
	This means that $x$ is a convex combination of 
	\begin{equation*}
		(\alpha(x),0,\dotsc,0)\quad  \text{ and }\quad \frac{(0,x_2,\dotsc,x_n)}{\sum_{i=2}^n x_i}.
	\end{equation*}
	Since for $x\in T_n\setminus E$ we have $\sum_{i=1}^n x_i \leq 1$, we conclude that $\alpha$ maps $T_n\setminus E$ to the interval $[0,1]$. 
	The function $\alpha$ is quasilinear, that is, for any $v,w\in T_n\setminus E$ we have
	\begin{equation*}
		\min\{\alpha(v),\alpha(w)\} \leq \alpha((1-\lambda)v+\lambda w) \leq \max\{ \alpha(v),\alpha(w)\},
	\end{equation*}
	for all $\lambda\in (0,1)$. One may easily verify this by checking that $\lambda \mapsto \alpha((1-\lambda)v+\lambda w)$ is monotone.
	
	Denote by $(v_i)_{i=0}^n$ the vertices of $S_1$. $S_1$ is associated with $\mathbf{F}$ and therefore we may assume $e_0=v_0$ and $v_i \in\relinterior F_i$ for $i=1,\dotsc,n$. Since $\relinterior F_i=\relinterior [e_0,\dotsc,e_i]$, we conclude, for $i=1,\dotsc,n$, that $(v_i)_1>0$ and $\sum_{j=1}^n (v_i)_j<1$, or equivalently $\alpha(v_i) \in (0,1)$. The same argument holds true for the vertices $(w_i)_{i=0}^n$ of $S_2$, i.e., $\alpha(w_i)\in(0,1)$ for $i=1,\dotsc,n$ and $w_0=e_0$.
	We set 
	\begin{align*}
		a&:=\min_{1\leq i\leq n} \min(\alpha(v_i),\alpha(w_i)),&
		b&:=\max_{1\leq i\leq n} \max(\alpha(v_i),\alpha(w_i)),
	\end{align*}
	and note that $0<a\leq b < 1$.
	Since $\alpha$ is quasilinear, the maximum over $S_1$, respectively $S_2$, is assumed at the vertices of $S_1$, respectively $S_2$. This yields
	\begin{equation}\label{eqn:max_flag}
		\max_{y\in S_1\cup S_2} \alpha(y) \leq b.
	\end{equation}
	Analogously, the minimum over the $(n-1)$-simplex $[v_1,\dotsc,v_n]$, respectively $[w_1,\dotsc,w_n]$, is assumed at some vertex and therefore
	\begin{equation*}
		\min_{y\in [v_1,\dotsc,v_n]\cup[w_1,\dotsc,w_n]} \alpha(y)\geq a.
	\end{equation*}

	We define $H_{00}^+$, respectively $H_{01}^+$, as the closed half-space that contains $e_0$, respectively $e_1$, in the interior and whose boundary hyperplane is spanned by $E$ and $be_1$, respectively $ae_1$. Hence
	\begin{align}\label{eqn:h00}
	\begin{split}
		H_{00}^+ &= \left\{(x_1,\dotsc,x_n)\in \R^n: \sum_{i=2}^n x_i \leq 1-\frac{x_1}{b}\right\},\\
		H_{01}^+ &= \left\{(x_1,\dotsc,x_n)\in \R^n: \sum_{i=2}^n x_i \geq 1-\frac{x_1}{a}\right\}.
	\end{split}
	\end{align}
	Then $H_{00}^+$ and $H_{01}^+$ satisfy conditions i) and ii) for $i=0$.
	We further set 
	\begin{equation*}
		W_0:=H_{00}^+\cap H_{01}^+ = \left\{(x_1,\dotsc,x_n)\in\R^n : 1-\frac{x_1}{a}\leq \sum_{i=2}^n x_i \leq 1-\frac{x_1}{b}\right\}.
	\end{equation*}
	Clearly, $W_0\setminus E$ is mapped by $\alpha$ to $[a,b]\subset \R$. 
	For $x\in S_1\cap S_2$ we conclude by \eqref{eqn:max_flag} that $\alpha(x) \leq b$, or equivalently $x\in H_{00}^+$. Thus
	\begin{equation}\label{eqn:symdiff1}
		(S_1\triangle S_2)\cap H_{01}^+ \subset H_{00}^+\cap H_{01}^+ = W_0.
	\end{equation}
	
	Next, we set
	\begin{align*}
		T'&:= H_{01}\cap T_n,&
		S_1'&:=H_{01}\cap S_1,&
		S_2'&:=H_{01}\cap S_2.
	\end{align*}
	Then $T'$ is a $(n-1)$-dimensional simplex spanned by the vertices $ae_1, e_2, \dotsc, e_n$.
	Since $v_i,w_i \in H_{01}^+$ for $i=1,\dotsc,n$ and $e_0=v_0=w_0\not\in H_{01}^+$,  it follows that 
	$S_1'$ as well as $S_2'$ are flag simplices of $T'$ associated with the flag $\mathbf{G}=(G_0,\dotsc, G_{n-2})\in\flag(T)'$, where
	$G_j = F_{j+1}\cap H_{01} = [ae_1,e_2,\dotsc,e_{j+1}]$ for $j=0,\dotsc,n-2$.
	Furthermore
	\begin{equation}\label{eqn:symdiff2}
		(S_1\triangle S_2)\setminus H_{01}^+ \subset \conv\left(\{e_0\}\cup (S_1'\triangle S_2')\right),
	\end{equation}
	since $S_1\setminus H_{01}^+ \subset \conv(\{e_0\}\cup S_1')$ and $S_2\setminus H_{01}^+ \subset \conv(\{e_0\}\cup S_2')$.
	By the induction hypothesis we obtain wedges $W_j'$, $j=0,\dotsc, n-2$, such that
	\begin{equation}\label{eqn:symdiff3}
		S_1'\triangle S_2' \subset \bigcup_{j=0}^{n-2} W_j',
	\end{equation}
	and $W_j'=\bar{H}_{j0}^+\cap \bar{H}_{j1}^+$. 
	For $i=1,\dotsc,n-1$, we define $H_{i0}^+$, respectively $H_{i1}^+$, as the closed half-space that contains $e_i$, respectively $e_{i+1}$, and whose boundary hyperplane is spanned by $\bar{H}_{(i-1)0}$, respectively $\bar{H}_{(i-1)1}$, and the origin $e_0$.
	We easily verify that $H_{i0}^+$ and $H_{i1}^+$ satisfy condition i) and ii) and set $W_i:= H_{i0}^+\cap H_{i1}^+$, for $i=1,\dotsc,n-1$.
	Then $\conv(\{e_0\}\cup W_{i-1}')\subset W_i$ and by \eqref{eqn:symdiff1}, \eqref{eqn:symdiff2} and \eqref{eqn:symdiff3} we conclude
	\begin{equation*}
		S_1\triangle S_2
		\subset W_0 \cup \conv\left(\{e_0\}\cup (S_1'\triangle S_2')\right)
		\subset W_0 \cup \bigcup_{i=1}^{n-1} W_i.
	\end{equation*}
	Thus the induction step is complete and the lemma follows.
\end{proof}

\subsection{Lower and upper bounds on \texorpdfstring{$\left|\flag(P)\right|$}{|flag(P)|}}
For $P\in\mathcal{P}(\R^n)$ we have that
\begin{equation*}
	\left|\flag(P)\right| 
	= \sum_{F\in\face_{n-1}(P)} \left|\flag(F)\right| 
	\geq (n+1)\, \min \{\left|\flag(F)\right| : F\in\face_{n-1}(P)\}.
\end{equation*}
Thus, by induction, we find
\begin{equation}\label{eqn:flag_lower_bound}
	\left|\flag(P)\right| \geq \left|\flag(T)\right| = (n+1)! , 
\end{equation}
where $T$ is an $n$-dimensional simplex. For centrally symmetric convex polytopes $P$ it is an open conjecture by Kalai \cite{Kalai:1989}, that
\begin{equation}\label{eqn:kalai}
	 \left|\flag(P)\right|\overset{?}{\geq} \left|\flag(C)\right| = n!\, 2^n , 
\end{equation}
where $C$ is an $n$-dimensional cube.
B{\'a}r{\'a}ny and Lov{\'a}sz \cite[Cor.\ 3]{BarLov:1982} showed that $f_{n-1}(P)\geq 2^n$ for any simplicial centrally symmetric $n$-polytope $P$ (Stanley \cite{Stanley:1987} later generalized their results), which proves Kalai's conjecture \eqref{eqn:kalai} for simplicial, or simple, centrally symmetric polytopes.
Much less is known for general centrally symmetric $n$-polytopes.
The $2$-dimensional case is trivial. For dimension $n=3$ the famous $3^n$ conjecture by Kalai \cite{Kalai:1989} (see \cite[Sec.\ 2]{SWZ:2009} for a quick proof) gives $f_0(P)+f_1(P)+f_2(P)\geq 3^3-1 = 26$ for every centrally symmetric $3$-polytope $P$. This together with Euler's polyhedral formula, $f_0(P)-f_1(P)+f_2(P)=2$, leads to 
\begin{equation*}
	\left|\flag(P)\right| = 4 f_1(P) \geq 3!\,2^3.
\end{equation*}
Apparently, for dimension $n\geq 4$ it is still an open question if the $n$-cube minimizes $\left|\flag(P)\right|$ for all centrally symmetric $n$-polytopes.
It was observed that $\left|\flag(H)\right|=n!\, 2^n$ holds true for any $n$-dimensional Hanner polytope \cite{Hanner:1956} (cf.\ \cite{Kim:2014}).
Thus, incidentally the conjectured minimizers in Mahler's conjecture are exactly the same as in Kalai's conjecture.

\begin{remark}[Mahler's conjecture]
	It is well-known that the classical \emph{affine isoperimetric inequality} (cf.\ \cite{Lutwak:1991}), which provides an upper bound in terms of Euclidean 	balls for the affine surface area, is connected to the \emph{Blaschke--Santal{\'o} inequality} (cf.\ \cite{Lutwak:1985}), which gives an upper bound in terms of Euclidean balls for the volume product of a convex body and its polar body (cf.\ \cite{BBF:2014, Boroczky:2010} and \cite[Sec.\ 10.5]{Schneider:2014}). Many extension of both inequalities were since discovered.  See for instance \cite{CagWer:2014, DPP:2016, HaberlSchuster:2009, WernerYe:2008, Stancu:2012}.
	A lower bound for the volume product is a well-known open conjecture by Mahler \cite{Mahler:1939, Mahler:1939a}. 
	For general convex bodies it states that the lower bound of the volume product is achieved by an $n$-dimensional simplex and for centrally symmetric convex bodies $K$, i.e., $K=-K$, it is conjectured that the minimum is achieved by an $n$-dimensional cube or more generally by the Hanner polytopes. The conjecture holds true in dimension two, as proved by Mahler, and the equality cases were obtained in \cite{Meyer:1991,Reisner:1986}. Very recently a proof in dimension three for the symmetric case has appeared \cite{IriShi:2017}. In all other cases the conjecture is still open and the best lower bound is given by Kuperberg \cite{Kuperberg:2008} who improved a previous bound by Bourgain and Milman \cite{BouMil:1987}. Mahler's conjecture was also shown to hold true for many special classes of convex bodies. We  refer to \cite{AFZ:2017, BMMR:2013, GPV:2014, Ivaki:2015, Kim:2014} and the references therein for  recent expositions.
\end{remark}

\begin{remark}[Subdivisions and the volume product]
	As mentioned in the beginning of subsection 2.2, a subdivision of $P$ into a simplicial complex of flag simplices is determined by a map $\mathcal{S}^P:\face(P)\to P$ such that $\mathcal{S}^P(F)\in\relinterior F$ for all $F\in\face(P)$, see \eqref{eqn:simplex_sub}.
	We may therefore express the volume of $P$ by the volume of the simplices in the subdivision, that is,
	\begin{equation*}
		\vol_n(P) = \sum_{\mathbf{F}\in\flag(P)} \vol_n(S(\mathcal{S}^P,\mathbf{F})).
	\end{equation*}
	Now, since $\left|\flag(P)\right| = \left|\flag(P^\circ)\right|$ and by the Cauchy--Schwarz inequality, we find
	\begin{align*}
		\vol_n(P)\vol_n(P^\circ) 
 		&= \left(\sum_{\mathbf{F}\in\flag(P)} \vol_n(S(\mathcal{S}^P,\mathbf{F}))\right)\left(\sum_{\widehat{\mathbf{F}}\in\flag(P^\circ)} \vol_n(S(\mathcal{S}^{P^\circ},\widehat{\mathbf{F}}))\right)\\
 		&\geq \left(\sum_{\mathbf{F}\in\flag(P)} \vol_n(S(\mathcal{S}^P,\mathbf{F}))^{\frac{1}{2}}\vol_n(S(\mathcal{S}^{P^\circ},\widehat{\mathbf{F}}))^{\frac{1}{2}}\right)^2\\
 		&\geq \left|\flag(P)\right|^2 \min_{\mathbf{F}\in\flag(P)} \vol_n(S(\mathcal{S}^P,\mathbf{F})) \vol_n(S(\mathcal{S}^{P^\circ},\widehat{\mathbf{F}})),
	\end{align*}
	where $\mathcal{S}^P$ and $\mathcal{S}^{P^{\circ}}$ are subdivision schemes for $P$, respectively $P^\circ$.

	A construction like this is used in \cite{Kim:2014} on Hanner polytopes to show that they are local minimizers of the volume product. A crucial step is the observation, that the barycenter subdivision of a Hanner polytope gives simplices of equal volume \cite[Prop.\ 4]{Kim:2014}. Hence, if $H$ is a Hanner polytope, then so is the polar $H^\circ$ and
	\begin{equation*}
		\vol_n(H)\vol_n(H^\circ) = \left|\flag H\right|^2 \vol_n(S(\mathcal{C}^H,\mathbf{F})) \vol_n(S(\mathcal{C}^{H^\circ},\widehat{\mathbf{F}})),
	\end{equation*}
	for all $\mathbf{F}\in\flag(H)$, where $\mathcal{C}^H$ and $\mathcal{C}^{H^\circ}$ are the subdivision schemes given by the centroids of the faces of $H$, respectively $H^\circ$. 
\end{remark}

\medskip
It was proved by Figiel, Lindenstrauss and Milman \cite[Thm.\ 3.4]{FLM:1977} that there is an absolute constant $c>1$ such that
\begin{equation*}
	\max\{f_0(P),f_{n-1}(P)\} \geq c^{\sqrt{n}},
\end{equation*}
for every centrally symmetric $n$-polytope. A trivial lower bound is obtained by applying \eqref{eqn:flag_lower_bound} to the flags of a facet of a centrally symmetric polytope $P$, that is 
\begin{equation*}
	\left|\flag(P)\right| = \sum_{F\in\face_{n-1}(P)} \left|\flag(F)\right| \geq n!\, f_{n-1}(P).
\end{equation*}
Since $\left|\flag(P)\right|=\left|\flag(P^\circ)\right|$ and $f_{n-1}(P)=f_0(P^\circ)$, we combine these two inequalities to conclude, that there exists a absolute constant $c>1$ such that
\begin{equation*}
	\left|\flag(P)\right| \geq n! \, c^{\sqrt{n}},
\end{equation*}
for every centrally symmetric $n$-polytope $P$. Comparing this with the conjectured lower bound by Kalai \eqref{eqn:kalai}, we notice that the exponential factor is of order $\sqrt{n}$ instead of $n$.

\medskip
An upper bound for the number of complete flags with respect to the number of vertices follows by the general upper bound theorem (cf.\ \cite[Cor.\ 6.6]{BilleraEhrenborg:2000}), that is,
\begin{equation*}
	\left|\flag(P)\right| \leq \left|\flag C_n(f_0(P))\right|,
\end{equation*}
where $C_n(k)$ is the \emph{cyclic $n$-polytope} with $k$ vertices, i.e., the convex hull of $k$ vertices chosen from the moment generating curve $(t,t^1,\dotsc,t^n)\in\R^n$. The cyclic $n$-polytope is simplicial and $\lfloor \frac{n}{2}\rfloor$-neighborly, see \cite{McMuShep:1971},
hence
\begin{align*}
	\left|\flag C_n(k)\right| &= n! f_{n-1}(C_n(k)) \\
	&=\begin{cases}
	  	(2\ell)! \frac{k}{k-\ell} \binom{k-\ell}{\ell}, &\text{if $n=2\ell$,}\\
	  	2(2\ell+1)!\binom{k-\ell-1}{\ell}, &\text{if $n=2\ell+1$,}
	  \end{cases}
\end{align*}
(cf.\ \cite[p.\ 25]{Ziegler:1995}). Upper bounds for the face numbers of centrally symmetric polytopes were obtained in \cite{BLN:2013} using the symmetric moment curve introduced in \cite{BN:2008}.

\section{Results for Uniform Weights}

In this section we restate  parts of the original proof for the uniform case, i.e., $\phi\equiv\psi\equiv 1$, contained in \cite{Schuett:1991} in terms of  flag simplices. We will build on these results in the next section for the proof of our main Theorem \ref{thm:main_theorem}.  We hope that the introduction of flag simplices will  help to make the proof more transparent.

Let $P\in\mathcal{P}(\R^n)$ and $v\in\vert P$.  Let $\delta >0$. We define
\begin{align}\label{eqn:def_a_eucl}
		A^+(P,v,\delta) 
		&= \bigcup\left\{P\cap \interior H^+ :\, 
			\parbox[c]{0.4\textwidth}{$\lambda_n(P\cap H^+)\leq \delta$, $v\in\interior H^+$, and\\ $w\not\in H^+$ for all $w\in\vert P \setminus\{v\}$}\right\},
\end{align}
where  $\lambda_n$ is the Lebesgue measure on $\mathbb{R}^n$ and $H^+$ denotes as usual a closed half-space of $\R^n$. 
Thus  $A^+(P,v,\delta)$  is the union of all open halfspaces of volume smaller or equal $\delta$ that cut off one vertex from $P$.
Furthermore, we define
\begin{align}\label{eqn:def_b_eucl}
		B^+(P,v,w,\delta) 
		&= \bigcup\left\{P\cap \interior H^+ :\, \text{$\lambda_n(P\cap H^+)\leq \delta$ and $v,w\in H^+$}\right\},
\end{align}
for $v,w\in\vert P$, $v\neq w$. 
$B^+(P,v,w,\delta)$  is the union of all open halfspaces of volume smaller or equal $\delta$ that cut off at least two vertices from $P$.
Then
\begin{align}\label{eqn:float_split}
\begin{split}
		P\setminus P_\delta 
		&= \bigcup \{P\cap H^+ : \lambda_n(P\cap H^+)< \delta\}\\
		&= \bigcup_{v\in\vert P} A^+(P,v,\delta) \cup \bigcup_{\substack{v,w\in\vert P\\i\neq j}} B^+(P,v,w,\delta).
\end{split}
\end{align}
We may think of $A^+(P,v,\delta)$ as the part of $P\setminus P_\delta$ that is determined by $v\in\vert P$. The sets $B^+(P,v,w,\delta)$ cover what is left after removing $A^+(P,v,\delta)$ for all vertices and are related to the part of $P\setminus P_\delta$ that is determined by the edge $[v,w]$.

First we recall  known facts for $n$-simplices.
\begin{lemma}[{\cite{Schuett:1991}}]\label{lem:eucl_first}
	Let $T$ be an $n$-dimensional simplex in $\R^n$ spanned by the vertices $z_0,\dotsc,z_n$.
	Then for all $\delta \in (0,\frac{\lambda_n(T)}{2})$ we have
	\begin{align}
		\label{eqn:lambda_A_lower}
		\lambda_n(A^+(T,z_i,\delta)) &\geq \frac{\delta}{n^{n-1}} \left[\ln\left(\frac{\lambda_n(T)}{\delta}\right)\right]^{n-1},\\
		\label{eqn:lambda_A_upper}
		\lambda_n(A^+(T,z_i,\delta)) &\leq \frac{\delta}{n^{n-1}} \left[\ln\left(\frac{n^n\lambda_n(T)}{\delta}\right)\right]^{n-1} + n\delta \left[\ln\left(\frac{n!\lambda_n(T)}{\delta}\right)\right]^{n-2},
	\end{align}
	and
	\begin{equation}
		\label{eqn:lambda_B_upper}
		\limsup_{\delta\to 0^+} \frac{\lambda_n(B^+(T,z_i,z_j,\delta))}{\delta\left(\ln\frac{1}{\delta}\right)^{n-1}} = 0.
	\end{equation}
	Furthermore, if $H_0^+$ is a closed half-space such that $z_i\not \in H_0^+$, then
	\begin{equation}\label{eqn:lambda_A_upper2}
		\limsup_{\delta\to 0^+} \frac{\lambda_n(A^+(T,z_i,\delta)\cap H_0^+)}{\delta\left(\ln\frac{1}{\delta}\right)^{n-1}} = 0.
	\end{equation}
	Finally, if $S$ is a flag simplex of $T$ such that $z_i\in S$, then
	\begin{equation}\label{eqn:lambda_concentrate}
		\lim_{\delta\to 0^+} \frac{\lambda_n(S \cap A^+(T,z_i,\delta))}{\delta\left(\ln\frac{1}{\delta}\right)^{n-1}} = \frac{1}{n!\,n^{n-1}}.
	\end{equation}
\end{lemma}
\begin{proof} 
	The inequalities \eqref{eqn:lambda_A_lower} as well as \eqref{eqn:lambda_A_upper} are exactly as in \cite[Lem.~1.3(i)]{Schuett:1991}.
	In \cite[Lem.~1.3(ii)]{Schuett:1991} it is stated that
	\begin{equation*}
		\lambda_n(B^+(T,z_i,z_j,\delta)) \leq c_n \delta \left[\ln\left(\frac{\lambda_n(T)}{\delta}\right)\right]^{n-2}, \quad \text{for $\delta<\frac{\lambda_n(T)}{2}$},
	\end{equation*}
	for a positive constant $c_n$ depending only on $n$. Clearly this implies \eqref{eqn:lambda_B_upper}.

	\smallskip
	Next we show that \eqref{eqn:lambda_A_lower} implies 
	\begin{equation}\label{eqn:lambda_A_lower2}
		\liminf_{\delta\to 0^+} \frac{\lambda_n(A^+(T,z_i,\delta)\cap H_0^+)}{\delta \left(\ln \frac{1}{\delta}\right)^{n-1}} \geq \frac{1}{n^{n-1}},
	\end{equation}
	for any closed half-space $H_0^+$ with $z_i\in\interior H_0^+$. Since $z_i\in \interior H_0^+$,  there is $\varepsilon>0$ such that 
	$B(z_i,\varepsilon)\subset\interior H_0^+$. Let $H_1$ be the hyperplane spanned by the points of the intersection of the boundary of $B(z_i,\varepsilon)$ with the segments $[z_i,z_j]$ for all  $j\neq i$ and denote by $H_1^+$ the closed half-space that contains $z_i$ in the interior. Then $T':=T\cap H_1^+$ is an $n$-simplex and $T'\subset T\cap H_0^+$. Furthermore, if $H^+$ is a closed halfspace such that $H$ separates $z_i$ from all other vertices of $T'$, then $T'\cap H^+ = T \cap H^+ \cap H_0^+$ and therefore $A^+(T',z_i,\delta)\subset A^+(T,z_i,\delta)\cap H_0^+$. This yields
	\begin{equation*}
		\liminf_{\delta\to 0^+} \frac{\lambda_n(A^+(T,z_i,\delta)\cap H_0^+)}{\delta \left(\ln \frac{1}{\delta}\right)^{n-1}} 
		\geq \liminf_{\delta\to 0^+} \frac{\lambda_n(A^+(T',z_i,\delta))}{\delta \left(\ln \frac{1}{\delta}\right)^{n-1}}
		\overset{\eqref{eqn:lambda_A_lower}}{\geq}  \frac{1}{n^{n-1}}.
	\end{equation*}

	Next, if $z_i\not\in H_0^+$, then $z_i\in\interior H_0^-$ and by \eqref{eqn:lambda_A_upper} and \eqref{eqn:lambda_A_lower2} we conclude
	\begin{align*}
		\limsup_{\delta\to 0^+}\! \frac{\lambda_n\big(\!A^+(T,z_i,\delta)\!\cap\! H^+_0\!\big)}{\delta \left(\ln \frac{1}{\delta}\right)^{n-1}}
		&\!\leq\! \limsup_{\delta\to 0^+}\! \frac{\lambda_n\big(\!A^+(T,z_i,\delta)\!\big)\!-\!\lambda_n\big(\!A^+(T,z_i,\delta)\!\cap\! H^-_0\!\big)}{\delta \left(\ln \frac{1}{\delta}\right)^{n-1}}\\
		&\!\leq\! \frac{1}{n^{n-1}} - \frac{1}{n^{n-1}} = 0.
	\end{align*}
	Thus \eqref{eqn:lambda_A_upper2} follows.
	
	\smallskip
	Finally, to prove \eqref{eqn:lambda_concentrate} we first note that it is sufficient to consider the standard simplex $T_n$ spanned by $0,e_1,\dotsc, e_n$.
	We denote by $S_{\id}$ the flag simplex spanned by the vertices $0,w_1,\dotsc,w_n$, where $w_i = \frac{1}{2i} \sum_{k=1}^{i} e_k$.
	Let $\Sigma(n)$ be the set of permutations of $\{1,\dotsc,n\}$. For $\sigma \in \Sigma(n)$ there is an orthogonal transformation $\alpha_{\sigma}\in\mathrm{O}(n)$ such that
	$\alpha_{\sigma}(e_i) = e_{\sigma(i)}$. We set $S_\sigma:= \alpha_{\sigma}(S_{\id})$. Then 
	\begin{enumerate}
		\item[i)] $\interior S_{\sigma} \cap \interior S_{\sigma'} =\emptyset$ for all $\sigma\neq \sigma'$, and
		\item[ii)] $\frac{1}{2}T_n = \bigcup \{S_{\sigma}:\sigma\in \Sigma(n)\}$.
	\end{enumerate}
	We have
	\begin{align*}
		\lambda_n\big(S_{\id}\cap A^+(T_n,0,\delta)\big) 
		&=\frac{1}{|\Sigma(n)|} \sum_{\sigma\in \Sigma(n)} \lambda_n\big(S_\sigma\cap A^+(T_n,0,\delta)\big) \\
		&=\frac{1}{n!} \lambda_n\left(\frac{1}{2}T_n\cap A^+(T_n,0,\delta)\right).
	\end{align*}
	This yields
	\begin{equation}\label{eqn:sid_lower}
		\liminf_{\delta\to 0^+}\! \frac{\lambda_n\big(\!S_{\id}\!\cap\! A^+(T_n,0,\delta)\!\big)}{\delta\left(\ln \frac{1}{\delta}\right)^{n-1}}
		\!=\! \frac{1}{n!} \liminf_{\delta\to 0^+}\! \frac{\lambda_n\big(\!\frac{1}{2}T_n\!\cap\! A^+(T_n,0,\delta)\!\big)}{\delta\left(\ln \frac{1}{\delta}\right)^{n-1}}
		\!\overset{\eqref{eqn:lambda_A_lower2}}{\geq}\! \frac{1}{n!\,n^{n-1}}
	\end{equation}
	and
	\begin{equation}\label{eqn:sid_upper}
		\limsup_{\delta\to 0^+}\! \frac{\lambda_n\big(\!S_{\id}\!\cap\! A^+(T_n,0,\delta)\!\big)}{\delta\left(\ln \frac{1}{\delta}\right)^{n-1}}
		\!=\! \frac{1}{n!} \limsup_{\delta\to 0^+}\! \frac{\lambda_n\big(\!\frac{1}{2}T_n\!\cap\! A^+(T_n,0,\delta)\!\big)}{\delta\left(\ln \frac{1}{\delta}\right)^{n-1}}
		\!\overset{\eqref{eqn:lambda_A_upper}}{\leq}\! \frac{1}{n!\,n^{n-1}}
	\end{equation}
	Now assume that $S$ is a flag simplex that is associated with the same flag as $S_{\id}$.
	By Lemma \ref{lem:flag_simplex_switch} there is a sequence of wedges $(W_i)_{i=0}^{n-1}$ such that $\bigcup_{i=0}^{n-1} W_i\supset S\triangle S_{\id}$ and 
	\begin{equation*}
		\limsup_{\delta\to 0^+} \frac{\lambda_n(W_i\cap A^+(T_n,0,\delta))}{\delta\left(\ln\frac{1}{\delta}\right)^{n-1}} = 0,
	\end{equation*}
	for $i=0,\dotsc,n-1$, by \cite[Lem.~1.4]{Schuett:1991}.
	Hence
	\begin{equation}\label{eqn:wi_cover}
		\limsup_{\delta\to 0^+} \frac{\lambda_n\big(\!(S_{\id}\triangle S)\!\cap\! A^+\!(T_n,0,\delta)\!\big)}
			{\delta \left(\ln \frac{1}{\delta}\right)^{n-1}}
		\!\leq\! \limsup_{\delta\to 0^+} \sum_{i=0}^n \frac{\lambda_n\big(\!W_i\!\cap\! A^+\!(T_n,0,\delta)\!\big)}
			{\delta\left(\ln\frac{1}{\delta}\right)^{n-1}} 
		= 0.
	\end{equation}
	By \eqref{eqn:sid_upper} and \eqref{eqn:wi_cover}, we derive
	\begin{multline*}
		\limsup_{\delta \to 0^+} \frac{\lambda_n(S\cap A^+(T_n,0,\delta))}{\delta \left(\ln \frac{1}{\delta}\right)^{n-1}}\\
		\!\leq\! \limsup_{\delta\to 0^+} \frac{\lambda_n\big(\!S_{\id}\!\cap\! A^+(T_n,0,\delta)\!\big)\!+\!\lambda_n\big(\!(S_{\id}\triangle S)\!\cap\! A^+(T_n,0,\delta)\!\big)}{\delta \left(\ln \frac{1}{\delta}\right)^{n-1}}
		\!\leq\! \frac{1}{n!\,n^{n-1}},
	\end{multline*}
	and by \eqref{eqn:sid_lower} and \eqref{eqn:wi_cover}, we obtain
	\begin{multline*}
		\liminf_{\delta \to 0^+} \frac{\lambda_n(S\cap A^+(T_n,0,\delta))}{\delta \left(\ln \frac{1}{\delta}\right)^{n-1}}\\
		\!\geq\! \liminf_{\delta\to 0^+} \frac{\lambda_n\big(\!S_{\id}\!\cap\! A^+(T_n,0,\delta)\!\big)\!-\!\lambda_n\big(\!(S_{\id}\triangle S)\!\cap\! A^+(T_n,0,\delta)\!\big)}{\delta \left(\ln\frac{1}{\delta}\right)^{n-1}}
		\!\geq\! \frac{1}{n!\,n^{n-1}},
	\end{multline*}
	Thus \eqref{eqn:lambda_concentrate} follows.
\end{proof}

The following concentration result is  implicitly contained in the proof presented in \cite{Schuett:1991}. It states that the volume of $P\setminus P_\delta$ is concentrated in 
the union of the flag simplices.
\begin{theorem}\label{thm:main_eucl_concentration}
	For $P\in\mathcal{P}(\R^n)$ and any family $(S(\mathbf{F}))_{\mathbf{F}\in\flag(P)}$ of flag simplices, where $S(\mathbf{F})$ is associated with $\mathbf{F}$, we have
	\begin{equation*}
		\lim_{\delta\to 0^+} \frac{\lambda_n\big( (P\setminus P_\delta)\setminus \bigcup\{S(\mathbf{F}):\mathbf{F}\in\flag(P)\}\big)}{\delta\left(\ln \frac{1}{\delta}\right)^{n-1}} = 0.
	\end{equation*}
\end{theorem}
For the reader's convenience we  include the proof of Theorem \ref{thm:main_eucl_concentration}. The first step is the following lemma.

\begin{lemma}\label{lem:main_eucl_concentration_simplex}
	Let $T$ be an $n$-simplex and let $v\in\vert T$.
	For any family of flag simplices $(S(\mathbf{F}))_{\mathbf{F}\in\flag_v(T)}$, where $S(\mathbf{F})$ is associated with the flag $\mathbf{F}$, we have
	\begin{equation*}
		\limsup_{\delta\to 0^+}\frac{\lambda_n\big(A^+(T,v,\delta)\setminus \bigcup \{ S(\mathbf{F}):\mathbf{F}\in\flag_v(T)\}\big)}{\delta\left(\ln\frac{1}{\delta}\right)^{n-1}} = 0.
	\end{equation*}
\end{lemma}
\begin{proof}
	By Lemma \ref{lem:disjoint}, given a family $(S(\mathbf{F}))_{\mathbf{F}\in\flag(T)}$ of flag simplices there always exists a family $(S'(\mathbf{F}))_{\mathbf{F}\in\flag(T)}$ that is disjoint and such that $S'(\mathbf{F})\subset S(\mathbf{F})$.
	Hence, without loss of generality, we may assume that $(S(\mathbf{F}))_{\mathbf{F}\in\flag(T)}$ is a disjoint family of flag simplices, that is, $\interior S(\mathbf{F})\cap \interior S(\mathbf{F}') = \emptyset$, for $\mathbf{F}\neq\mathbf{F}'$. Then
	\begin{multline*}
		\lambda_n\left(A^+(T, v, \delta)\setminus \bigcup \{S(\mathbf{F}) : \mathbf{F} \in \flag_v(T)\} \right) \\
		= \lambda_n\big(A^+(T, v, \delta)\big) -\sum_{\mathbf{F}\in\flag_v(T)}\lambda_n\big(\!S(\mathbf{F}) \cap A^+(T,v,\delta)\big).
	\end{multline*}
	By \eqref{eqn:lambda_A_upper}, \eqref{eqn:lambda_concentrate} and since $\left|\flag_v(T)\right| = n!$, we conclude
	\begin{multline*}
		\limsup_{\delta\to 0^+}\frac{\lambda_n\big(A^+(T,v,\delta)\setminus \bigcup \{ S(\mathbf{F}):\mathbf{F}\in\flag_v(T)\}\big)}
			{\delta\left(\ln\frac{1}{\delta}\right)^{n-1}}\\
		\leq \limsup_{\delta\to 0^+} \frac{\lambda_n(A^+(T,v,\delta))}{\delta\left(\ln\frac{1}{\delta}\right)^{n-1}}
			- \sum_{\mathbf{F}\in\flag_v(T)} \liminf_{\delta\to 0^+} \frac{\lambda_n(S(\mathbf{F})\cap A^+(T,v,\delta))}{\delta\left(\ln\frac{1}{\delta}\right)^{n-1}} \\
		= \frac{1}{n^{n-1}}-\frac{\left|\flag_v(T)\right|}{n!\,n^{n-1}} = 0. \tag*{$\qed$}
	\end{multline*}
	\renewcommand{\qedsymbol}{}
\end{proof}

\vspace{-0.5cm}
It is easy to improve Lemma \ref{lem:main_eucl_concentration_simplex} by combining \eqref{eqn:float_split} with \eqref{eqn:lambda_B_upper}.
\begin{corollary}\label{cor:euclidean_concentration_p}
	Let $T$ be an $n$-simplex.
	For any family $(S(\mathbf{F}))_{\mathbf{F}\in\flag(T)}$ of flag simplices, where $S(\mathbf{F})$ is associated with the flag $\mathbf{F}$, we have
	\begin{equation*}
		\limsup_{\delta\to 0^+}\frac{\lambda_n\big((T\setminus T_\delta)\setminus \bigcup \{ S(\mathbf{F}):\mathbf{F}\in\flag(T)\}\big)}
			{\delta\left(\ln\frac{1}{\delta}\right)^{n-1}} = 0.
	\end{equation*}
\end{corollary}

We are now ready to prove Theorem \ref{thm:main_eucl_concentration}. 
For any subdivison  $(S(\mathbf{F}))_{\mathbf{F}\in\flag(P)}$  of $P$, we
construct flag simplices $S'(\mathbf{F})$ and $T(\mathbf{F})$ such that $S'(\mathbf{F})\subset S(\mathbf{F})\subset T(\mathbf{F}) \subset P$. 
Then we apply  Corollary \ref{cor:euclidean_concentration_p} to $T(\mathbf{F})$ and the concentration statement follows for $S'(\mathbf{F})$ and hence for  
$S(\mathbf{F})$. The argument is slightly more complicated, because in general we may not assume that $(S(\mathbf{F}))_{\mathbf{F}\in\flag(P)}$ is a subdivision of $P$. To remedy this, we introduce the barycenter subdivision of $P$. But the same argument will work for any simplex subdivision.
\begin{proof}[Proof of Theorem \ref{thm:main_eucl_concentration}]
	For a polytope $P\in\mathcal{P}(\R^n)$ the barycenter subdivision gives a family of flag simplices $\bar{S}(\mathbf{F})$ associated with $\mathbf{F}\in\flag(P)$, such that $\interior \bar{S}(\mathbf{F})\cap \interior \bar{S}(\mathbf{F}') = \emptyset$ if $\mathbf{F}\neq \mathbf{F}'$ and 
	\begin{equation*}
		P=\bigcup_{\mathbf{F}\in\flag(P)} \bar{S}(\mathbf{F}).
	\end{equation*}
	By Proposition \ref{prop:flag simplices} ii), for any flag simplex $\bar{S}(\mathbf{F})$, we can choose a flag simplex $T(\mathbf{F})$ of $P$ associated with $F$ such that
	\begin{enumerate}
		\item[i)] $\bar{S}(\mathbf{F})\subset T(\mathbf{F})\subset P$ and
		\item[ii)] $\bar{S}(\mathbf{F})$ is a flag simplex of $T(\mathbf{F})$.
	\end{enumerate}
	Since $T(\mathbf{F}) \subset P$ we have $T(\mathbf{F})_\delta \subset P_\delta$ and therefore
	\begin{equation*}
		\bar{S}(\mathbf{F})\setminus P_\delta \subset \bar{S}(\mathbf{F})\setminus T(\mathbf{F})_\delta,
	\end{equation*}
	for each $\mathbf{F}\in\flag(P)$.
	Hence
	\begin{align*}
		\lambda_n\left((P\setminus P_\delta)\setminus \bigcup\{S(\mathbf{G}):\mathbf{G}\in\flag(P)\}\right)
		&\leq \sum_{\mathbf{F}\in\flag(P)}\!\!\! \lambda_n\left( (\bar{S}(\mathbf{F})\setminus P_\delta)\setminus S(\mathbf{F})\right)\\
		&\leq \sum_{\mathbf{F}\in\flag(P)}\!\!\! \lambda_n\left( (\bar{S}(\mathbf{F})\setminus T(\mathbf{F})_\delta)\setminus S(\mathbf{F})\right)
	\end{align*}
	By Proposition \ref{prop:flag simplices} i),  there is flag simplex $S'(\mathbf{F})$ of $P$ associated with $\mathbf{F}$ such that $S'(\mathbf{F})\subset \bar{S}(\mathbf{F})\cap S(\mathbf{F})$ and therefore $S'(\mathbf{F})$ is also a flag simplex of $T(\mathbf{F})$ associated with the same flag as $\bar{S}(\mathbf{F})$. Hence
	\begin{equation*}
		\limsup_{\delta\to 0^+} \frac{\lambda_n\left(\!(\bar{S}(\mathbf{F})\setminus T(\mathbf{F})_\delta)\setminus S(\mathbf{F})\!\right)}
			{\delta\left(\ln \frac{1}{\delta}\right)^{n-1}}
		\!\leq\! \limsup_{\delta\to 0^+} \frac{\lambda_n\left(\!(\bar{S}(\mathbf{F})\setminus T(\mathbf{F})_\delta)\setminus S'(\mathbf{F})\!\right)}
			{\delta\left(\ln \frac{1}{\delta}\right)^{n-1}}
		\!=\! 0,
	\end{equation*}
	by Corollary \ref{cor:euclidean_concentration_p}, because $\bar{S}(\mathbf{F})$ and $S'(\mathbf{F})$ are flag simplices of $T(\mathbf{F})$ associated with the same flag. This concludes the proof of the theorem.
\end{proof}

\section{Proof of the Main Theorem \ref{thm:main_theorem}}

First, let us prove some obvious bounds for the left-hand side of \eqref{eqn:deriv_inner_poly_weighted}. 
We set $c:= \min_{x\in P} \phi(x)$ and $C:=\max_{x\in P} \phi(x)$. Since $\phi:P\to (0,\infty)$ is continuous and $P$ is compact we have $c>0$.
Furthermore, $c\leq \phi(x) \leq C$ for all $x\in P$ and therefore,
\begin{equation}\label{eqn:wfloat_bound}
	P_{\delta/c} \subset P_\delta^\phi \subset P_{\delta/C},
\end{equation}
for all $\delta>0$ small enough.
By \eqref{eqn:deriv_inner_poly}, this yields
\begin{align*}
	\limsup_{\delta\to 0^+} \frac{ \Psi(P)-\Psi\big(P_\delta^\phi\big)}{\delta\, \left(\ln \frac{1}{\delta}\right)^{n-1}}
	&\leq \left(\max_{x\in P}  \psi(x)\right) \limsup_{\delta\to 0^+} \frac{\lambda_n(P)-\lambda_n\big(P_{\delta/c}\big)}{\delta\, \left(\ln\frac{1}{\delta}\right)^{n-1}}\\
	&= \frac{\max_{x\in P} \psi(x)}{\min_{x\in P} \phi(x)} \frac{\left|\flag(P)\right|}{n!\,n^{n-1}},
\end{align*}
and
\begin{equation*}
	\liminf_{\delta\to 0^+} \frac{ \Psi(P)-\Psi(P_\delta^\phi )}{\delta\, \left(\ln \frac{1}{\delta}\right)^{n-1}}
	\geq \frac{\min_{x\in P} \psi(x)}{\max_{x\in P} \phi(x)}\, \frac{\left|\flag(P)\right|}{n!\,n^{n-1}}.
\end{equation*}

We prove Theorem \ref{thm:main_theorem} in two steps: In Proposition \ref{thm:main_theorem_weighted_upper} we show that $\Psi(P\setminus P_\delta^\phi)$ is concentrated in the flag simplices and in Theorem \ref{thm:main_theorem_weighted_local} we calculate the limit for a fixed flag simplex.

The fact that $\Psi(P\setminus P_\delta^\phi)$ is concentrated in the flag simplices of $P$ follows from the result for uniform measures, Theorem \ref{thm:main_eucl_concentration}, by bounding the weight $\psi$ from above and the weight $\phi$ from below by positive constants.
\begin{proposition}[Flag-Simplex Concentration]\label{thm:main_theorem_weighted_upper}
	Let $P\in\mathcal{P}(\R^n)$ and let $\phi,\psi: P \to (0,\infty)$ be continuous.
	For any family $(S(\mathbf{F}))_{\mathbf{F}\in\flag(P)}$ of flag simplices of $P$, where $S(\mathbf{F})$ is associated with $\mathbf{F}$, we have
	\begin{equation*}
		\limsup_{\delta\to 0^+} \frac{\Psi\big((P\setminus P_\delta^\phi) \setminus \bigcup \{S(\mathbf{F}):\mathbf{F}\in\flag(P)\}\big)}
			{\delta\left(\ln\frac{1}{\delta}\right)^{n-1}} = 0.
	\end{equation*}
\end{proposition}
\begin{proof}
	We set $c:=\min_{x\in P} \phi(x)$. Then $c>0$ since $\phi$ is continuous and positive and $P$ is compact.
	By \eqref{eqn:wfloat_bound} we have $P\setminus P_{\delta}^\phi \subset P\setminus P_{\delta/c}$ and therefore
	\begin{multline*}
		\Psi\left((P\setminus P_\delta^\phi) \setminus \bigcup \{S(\mathbf{F}):\mathbf{F}\in\flag(P)\}\right)\\
		\leq \left(\max_{x\in P} \psi(x)\right) 
			\lambda_n\left((P\setminus P_{\delta/c}) \setminus \bigcup \{S(\mathbf{F}):\mathbf{F}\in\flag(P)\}\right).
	\end{multline*}
	Thus the statement follows by Theorem \ref{thm:main_eucl_concentration}.
\end{proof}

We now proceed to the second step, that is, finding the limit for a fixed flag simplex $S$ of $P$.
\begin{proposition}[Flag-Simplex Limit]\label{thm:main_theorem_weighted_local}
	Let $P\in\mathcal{P}(\R^n)$ and let $\phi,\psi :P\to (0,\infty)$ be continuous. 
	For $v\in\vert P$ and a flag simplex $S$ of $P$ with $v\in S$, we have
	\begin{equation*}
			\lim_{\delta\to 0^+} \frac{\Psi(S \setminus P_\delta^\phi)}{\delta\left(\ln\frac{1}{\delta}\right)^{n-1}}
			= \frac{1}{n!\,n^{n-1}} \frac{\psi(v)}{\phi(v)}.
	\end{equation*}
\end{proposition}

We prove this proposition by first considering $n$-simplices $T$ instead of general polytopes $P$.
In fact, by proving the proposition for the standard simplex $T_n$ with vertices $e_0, e_1,\dotsc,e_n$, we immediately see that the statement holds true for general $n$-simplices $T$ by the following argument: There is an affine transformation $\alpha$ such that $\alpha(T)=T_n$.
We denote by $\alpha\#\Phi$ the  push-forward measure of $\Phi$ by $\alpha$, i.e., for all Borel measurable sets $A$ we have
\begin{equation*}
	\alpha\#\Phi(A)=\Phi(\alpha^{-1}(A)) = \int_A \frac{\phi(\alpha^{-1}(x))}{\left|\det \alpha\right|}\d{\lambda_n}(x).
\end{equation*}
We see that $\alpha\#\Psi(A)$ is absolutely continuous with respect to $\lambda_n$ and the density function is given by $\frac{\mathrm{d}(\alpha\#\Phi)}{\mathrm{d}\lambda_n}(x) = \frac{\phi(\alpha^{-1}(x))}{\left|\det\alpha\right|}$.
Set $\widetilde{\phi} := \frac{\phi\circ\alpha^{-1}}{\left|\det\alpha\right|}$ and $\widetilde{\psi}:= \frac{\psi\circ\alpha^{-1}}{\left|\det\alpha\right|}$. Then $\widetilde{\phi}$, $\widetilde{\psi}$ are positive and continuous functions on $\alpha(T)=T_n$ and
\begin{align}\label{eqn:affine_trans}
	\begin{split}
	\alpha(T_\delta^\phi) 
	&= \bigcap\left\{\alpha(T\cap H^-): \Phi(T\cap H^+)\leq \delta\right\}  \\
	&= \bigcap\left\{T_n\cap H^-: \alpha\#\Phi(T_n \cap H^+)\leq \delta\right\}
	 = \left(T_n\right)_{\delta}^{\widetilde{\phi}}.
	\end{split}
\end{align}
Once we prove Proposition \ref{thm:main_theorem_weighted_local} for the standard simplex $T_n$ with positive and continuous weight functions we may conclude
\begin{equation*}
		\lim_{\delta\to 0^+}\!\frac{ \Psi\left(S \setminus T_\delta^\phi\right)}{\delta \left(\ln\frac{1}{\delta}\right)^{n-1}}
		\!=\! \lim_{\delta\to 0^+}\! \frac{ \alpha\#\Psi\left(\alpha(S) \setminus (T_n)_{\delta}^{\widetilde{\phi}}\right)}{\delta \left(\ln\frac{1}{\delta}\right)^{n-1}}
		\!=\! \frac{1}{n!\,n^{n-1}} \frac{\widetilde{\psi}(\alpha(v))}{\widetilde{\phi}(\alpha(v))}
		\!=\! \frac{1}{n!\,n^{n-1}} \frac{\psi(v)}{\phi(v)},
\end{equation*}
and therefore Proposition \ref{thm:main_theorem_weighted_local} will also hold true for general $n$-simplices $T$.

\smallskip
For $\varepsilon\in(0,1]$ and $i=0,\dotsc,n$, we put 
\begin{equation*}
	A_\phi^+(\varepsilon T_n, \varepsilon e_i,\delta) := \bigcup\left\{\varepsilon T_n \cap \interior H^+ :\, \parbox[c]{0.4\textwidth}{$\Phi(\varepsilon T_n\cap H^+)\leq \delta$, $\varepsilon e_i\in\interior H^+$\\ and $\varepsilon e_j\not\in H^+$ for $j\neq i$}\right\}.
\end{equation*}
 For $\phi\equiv 1$ we obtain the same sets as in \eqref{eqn:def_a_eucl}.
Set again $c:=\min_{x\in T_n} \phi(x)$. Then
\begin{equation*}
	c\lambda_n(T_n\cap H^+) \leq \Phi(T_n\cap H^+)
\end{equation*}
and therefore
\begin{equation}\label{eqn:a_bound}
	A_\phi^+(T_n,e_i,\delta) \subset A^+(T_n,e_i,\delta/c).
\end{equation}
Furthermore, for $\varepsilon\in (0,1]$, we have
\begin{equation}\label{eqn:a_lower}
	A^+_{\phi}(T_n,e_0,\delta) \supset A^+_{\phi}(\varepsilon T_n,e_0,\delta).
\end{equation}
This holds as for any half-space $H^+$ such that $e_0\in\interior H^+$ and $\varepsilon e_i \not\in H^+$ for $i=1,\dotsc,e_n$ we have 
$H^+\cap T_n = H^+\cap \varepsilon T_n$ and  $e_i\not\in H^+$ for $i=1,\dotsc,n$.

We set
\begin{equation*}
	\phi_{\varepsilon}^-(e_0) := \min_{x\in \varepsilon T_n} \phi(x) \text{ and }
	\phi_{\varepsilon}^+(e_0) := \max_{x\in \varepsilon T_n} \phi(x).
\end{equation*}
Then
\begin{equation*}
	\phi_{\varepsilon}^-(e_0) \lambda_n(\varepsilon T_n\cap H^+) 
	\leq \Phi(\varepsilon T_n\cap H^+) 
	\leq \phi_{\varepsilon}^+(e_0) \lambda_n(\varepsilon T_n\cap H^+),
\end{equation*}
which yields
\begin{equation}\label{eqn:a_eps_bound}
	A^+(\varepsilon T_n,e_0,\delta/\phi_{\varepsilon}^+(e_0)) \subset A^+_{\phi}(\varepsilon T_n,e_0,\delta) \subset A^+(\varepsilon T_n,e_0,\delta/\phi_{\varepsilon}^-(e_0))
\end{equation}

We will also use the sets 
\begin{equation*}
	B_\phi^+(T_n,e_i,\varepsilon e_j,\delta) 
	:= \bigcup\left\{T_n\cap \interior H^+ :\, \text{$\Phi(T_n\cap H^+)\leq \delta$ and $e_i,\varepsilon e_j\in H^+$}\right\},
\end{equation*}
for $\varepsilon\in (0,1]$ and $i,j=0,\dotsc,n$, $i\neq j$. For $\phi\equiv 1$ and $\varepsilon=1$ we obtain the same sets as in \eqref{eqn:def_b_eucl}.
As with \eqref{eqn:a_bound}, we find
\begin{equation}\label{eqn:B_bound}
	B_\phi^+(T_n,e_i, e_j,\delta)\subset B^+(T_n,e_i,e_j,\delta/c),
\end{equation}
for $i,j=0,\dotsc,n$, $i\neq j$.
Also, for $i=1,\dotsc,n$, we have
\begin{equation}\label{eqn:B_eps_bound}
	(\varepsilon T_n)\cap B_\phi^+(T_n,e_0,\varepsilon e_i,\delta) \subset B^+(\varepsilon T_n,e_0,\varepsilon e_i,\delta/c),
\end{equation}
because if $H^+$ is a halfspace that contains the segment $[e_0,\varepsilon e_i]$ and satisfies $\Phi(H^+\cap T_n) \leq \delta$, then
\begin{equation*}
	\delta \geq \Phi(H^+\cap T_n) \geq \Phi(H^+\cap \varepsilon T_n) \geq c\lambda_n(H^+\cap \varepsilon T_n).
\end{equation*}

\smallskip
To establish Proposition \ref{thm:main_theorem_weighted_local} for $T_n$, we first prove the following upper bound. 
\begin{lemma}\label{lem:main_simplex_upper_weighted}
	Let $T_n$ be the standard simplex in $\R^n$ with vertices $e_0,\dotsc, e_n$ and let $\phi,\psi: T_n\to (0,\infty)$ be continuous.
	If $H_0^+$ is a closed half-space such that $e_0\in\interior H^+$ and $e_i\not\in H^+$ for $i=1,\dotsc,n$, then
	\begin{equation*}
		\limsup_{\delta\to 0^+} \frac{\Psi\left( (H_0^+ \cap T_n)\setminus (T_n)_{\delta}^\phi\right)}{\delta \left(\ln\frac{1}{\delta}\right)^{n-1}} \leq 
		\frac{1}{n^{n-1}} \frac{\psi(e_0)}{\phi(e_0)}.
	\end{equation*}
\end{lemma}
\begin{proof}
	First note that, analogous to \eqref{eqn:float_split}, we have
	\begin{equation}\label{eqn:wfloat_split}
		T_n \setminus (T_n)_{\delta}^\phi = \bigcup_{i=0}^n  A_\phi^+(T_n,e_i,\delta) \cup \bigcup_{\substack{i,j=0\\i\neq j}}^n B_\phi^+(T_n,e_i,e_j,\delta).
	\end{equation}
	We set again $c:=\min_{x\in T_n} \phi(x)$.
	By \eqref{eqn:lambda_B_upper} and \eqref{eqn:B_bound}, 
	\begin{align}\label{eqn:b_weight}
		\begin{split}
		\limsup_{\delta\to 0^+}\frac{\Psi\big(H_0^+\cap B_\phi^+(T_n,e_i,e_j,\delta)\big)}{\delta\left(\ln\frac{1}{\delta}\right)^{n-1}} 
		\leq \limsup_{\delta\to 0^+}\frac{\Psi\big(B_\phi^+(T_n,e_i,e_j,\delta)\big)}{\delta\left(\ln\frac{1}{\delta}\right)^{n-1}}\\
		\leq \left(\max_{x\in P}\psi(x)\right) \limsup_{\delta\to 0^+}\frac{\lambda_n\big(B^+(T_n,e_i,e_j,\delta/c)\big)}{\delta\left(\ln\frac{1}{\delta}\right)^{n-1}}  
		= 0,
		\end{split}
	\end{align}
	for $i,j=0,\dotsc,n$, $i\neq j$.
	By \eqref{eqn:lambda_A_upper2} and \eqref{eqn:a_bound}
	\begin{equation}\label{eqn:a_weight}
		\limsup_{\delta \to 0^+}\! \frac{\Psi\big(\!H_0^+\!\cap\! A_\phi^+(T_n,e_i,\delta)\!\big)}{\delta\left(\ln\frac{1}{\delta}\right)^{n-1}}
		\!\leq\!\! \left(\!\max_{x\in P}\psi(x)\!\right) \limsup_{\delta\to 0^+}\!\frac{\lambda_n\big(\!H_0^+ \!\cap\! A^+(T_n,e_i,\frac{\delta}{c})\!\big)}{\delta\left(\ln\frac{1}{\delta}\right)^{n-1}} 
		\!=\!0,
	\end{equation}
	for $i=1,\dotsc,n$.
	By \eqref{eqn:wfloat_split}, \eqref{eqn:b_weight} and \eqref{eqn:a_weight} we conclude
	\begin{equation}\label{eqn:local_upper_bound}
		\limsup_{\delta\to 0^+} \frac{\Psi((H_0^+\cap T_n)\setminus (T_n)_\delta^{\phi})}{\delta\left(\ln\frac{1}{\delta}\right)^{n-1}} 
		\leq \limsup_{\delta\to 0^+} \frac{\Psi(H_0^+\cap A^+_\phi(T_n,e_0,\delta))}{\delta\left(\ln\frac{1}{\delta}\right)^{n-1}}.
	\end{equation}
	Now let $\varepsilon\in (0,1)$ be arbitrary and set
	\begin{align*}
		\phi^-_\varepsilon(e_0) &:= \min_{x\in \varepsilon T_n} \phi(x), &
		\phi^+_\varepsilon(e_0) &:= \max_{x\in \varepsilon T_n} \phi(x),\\
		\psi^-_\varepsilon(e_0) &:= \min_{x\in \varepsilon T_n} \psi(x), &
		\psi^+_\varepsilon(e_0) &:= \max_{x\in \varepsilon T_n} \psi(x).
	\end{align*}
	Then, for any Borel $A\subset \varepsilon T_n$, 
	\begin{align}\label{eqn:phi_eps}
		\begin{split}
		\phi^-_\varepsilon(e_0) \lambda_n(A) &\leq \Phi(A) \leq \phi^+_\varepsilon(e_0) \lambda_n(A),\\
		\psi^-_\varepsilon(e_0) \lambda_n(A) &\leq \Psi(A) \leq \psi^+_\varepsilon(e_0) \lambda_n(A).
		\end{split}
	\end{align}
	If $H^+$ is a hyperplane such that $e_0\in\interior H^+$, then either $\varepsilon e_i\not \in H^+$ for all $i\in\{1,\cdots n\}$, in which case $H^+\cap T_n=H^+\cap \varepsilon T_n$, or there is at least one $i\in\{1,\dotsc,n\}$ such that $[e_0,\varepsilon e_i]\in H^+$. Hence
	\begin{align*}
		(\varepsilon T_n) \cap A^+_\phi(T_n,e_0,\delta) 
		&\subset A^+_\phi(\varepsilon T_n, e_0,\delta) \cup \bigcup_{i=1}^n (\varepsilon T_n) \cap B^+_\phi(T_n,e_0,\varepsilon e_i,\delta)\\
		&\overset{\eqref{eqn:B_eps_bound}}{\subset} A^+_\phi(\varepsilon T_n, e_0,\delta) \cup \bigcup_{i=1}^n B^+(\varepsilon T_n, e_0, \varepsilon e_i, \delta/c).
	\end{align*}
	We therefore define the sets
	\begin{align*}
		A^+(\varepsilon) &:= A^+_\phi(\varepsilon T_n, e_0,\delta),\\
		B_i^+(\varepsilon) &:= B^+(\varepsilon T_n,e_0,\varepsilon e_i,\delta/c), \quad \text{for $i=1,\dotsc,n$,}\\
		C^+(\varepsilon) &:= A^+_\phi(T_n,e_0,\delta) \setminus (\varepsilon T_n).
	\end{align*}
	Then
	\begin{equation}\label{eqn:A_phi_cover}
		H_0^+\cap A^+_\phi(T_n,e_0,\delta)\subset A^+_\phi(T_n,e_0,\delta) \subset A^+(\varepsilon) \cup \bigcup_{i=1}^n B_i^+(\varepsilon) \cup C^+(\varepsilon).
	\end{equation}
	We have $A^+(\varepsilon)\subset \varepsilon T_n$ and therefore
	\begin{align*}
		\Psi(A^+(\varepsilon))
		&\overset{\eqref{eqn:phi_eps}}{\leq} \psi_\varepsilon^+(e_0)\, \lambda_n(A_\phi^+(\varepsilon T_n,e_0,\delta))\\
		&\overset{\eqref{eqn:a_eps_bound}}{\leq} \psi_\varepsilon^+(e_0)\, \lambda_n(A^+(\varepsilon T_n, e_0, \delta/\phi_\varepsilon^-(e_0)))\\
		&\overset{\eqref{eqn:a_lower}}{\leq} \psi_\varepsilon^+(e_0)\, \lambda_n(A^+(T_n, e_0, \delta/\phi_\varepsilon^-(e_0)))\\
		&\overset{\eqref{eqn:lambda_A_upper}}{\leq} \frac{\psi_\varepsilon^+(e_0)}{\phi_\varepsilon^-(e_0)}\,
			\left[ \frac{\delta}{n^{n-1}} \ln\left(\frac{n^n \phi_\varepsilon^-(e_0)}{n!\delta}\right)^{n-1} 
			+ n\delta \ln\left(\frac{\phi_\varepsilon^-(e_0)}{\delta}\right)^{n-2}\right].
	\end{align*}
	Hence
	\begin{equation}\label{eqn:A_eps}
		\limsup_{\delta\to 0^+} \frac{\Psi(A^+(\varepsilon))}{\delta \left(\ln\frac{1}{\delta}\right)^{n-1}} \leq \frac{1}{n^{n-1}} \frac{\psi_\varepsilon^+(e_0)}{\phi_\varepsilon^-(e_0)}.
	\end{equation}
	For $B_i^+(\varepsilon)$ we have
	\begin{equation}\label{eqn:B_eps}
		\limsup_{\delta\to 0^+} \frac{\Psi(B_i^+(\varepsilon))}{\delta\left(\ln\frac{1}{\delta}\right)^{n-1}} 
		\overset{\eqref{eqn:B_eps_bound}}{\leq} \left(\max_{x\in P}\psi(x)\right) \limsup_{\delta\to 0^+} 
			\frac{ \lambda_n(B^+(\varepsilon T_n,e_0,\varepsilon e_i,\delta/c))}{\delta\left(\ln\frac{1}{\delta}\right)^{n-1}} 
		\overset{\eqref{eqn:lambda_B_upper}}{=} 0.
	\end{equation}
	Finally, for $C^+(\varepsilon)$ we find
	\begin{equation*}
		C^+(\varepsilon) = A^+_\phi(T_n,e_0,\delta) \setminus (\varepsilon T_n) 
		\overset{\eqref{eqn:a_bound}}{\subset} A^+(T_n,e_0,\delta/c) \setminus (\varepsilon T_n),
	\end{equation*}
	which yields
	\begin{equation}\label{eqn:C_eps}
		\limsup_{\delta\to 0^+} \frac{\Psi(C^+(\varepsilon))}{\delta\left(\ln\frac{1}{\delta}\right)^{n-1}} 
		\leq \left(\max_{x\in P}\psi(x)\right) \limsup_{\delta\to 0^+} \frac{\lambda_n(A^+(T_n,e_0,\delta/c)\setminus (\varepsilon T_n))}{\delta\left(\ln\frac{1}{\delta}\right)^{n-1}}
		\overset{\eqref{eqn:lambda_A_upper2}}{=}0.
	\end{equation}
	Combining \eqref{eqn:A_eps}, \eqref{eqn:B_eps} and \eqref{eqn:C_eps} with \eqref{eqn:local_upper_bound} and \eqref{eqn:A_phi_cover}, we conclude
	\begin{equation*}
		\limsup_{\delta\to 0^+} \frac{\Psi\left((H_0^+\cap T_n)\setminus (T_n)_\delta^{\phi}\right)}{\delta\left(\ln\frac{1}{\delta}\right)^{n-1}} 
		\leq \frac{1}{n^{n-1}}\frac{\psi_\varepsilon^+(e_0)}{\phi_\varepsilon^-(e_0)}.
	\end{equation*}
	Since $\varepsilon\in(0,1)$ was arbitrary and since $\psi$ and $\phi$ are continuous in $e_0$, we have
	\begin{equation*}
		\limsup_{\delta\to 0^+} \frac{\Psi\left((H_0^+\cap T_n)\setminus (T_n)_\delta^{\phi}\right)}{\delta\left(\ln\frac{1}{\delta}\right)^{n-1}} 
		\leq \frac{1}{n^{n-1}}\liminf_{\varepsilon\to 0^+} \frac{\psi_\varepsilon^+(e_0)}{\phi_\varepsilon^-(e_0)} = \frac{1}{n^{n-1}}\frac{\psi(e_0)}{\phi(e_0)},
	\end{equation*}
	which concludes the proof.
\end{proof}

\begin{proof}[{Proof of Proposition \ref{thm:main_theorem_weighted_local}.}]
	Since $S$ is a flag simplex of $P$ with $v\in \vert P$ there is a uniquely determined flag $\mathbf{F}\in \flag_v(P)$ such that $S$ is associated to $\mathbf{F}$. First we show, that we may reduce the problem from $P$ to an $n$-dimensional simplex $T$.
	By Lemma \ref{lem:upper_simplex} there is an $n$-simplex $T_u$ such that
	\begin{enumerate}
		\item[i)] $P \subset T_u$,
		\item[ii)] $S$ is a flag simplex of $T_u$.
	\end{enumerate}
	By Urysohn's Lemma (or equivalently Tietze's extension theorem), we may choose a continuous extension $\tilde{\phi}:T_u\to (0,\infty)$ of $\phi$ from $P$ to $T_u$, i.e., we have $\tilde{\phi}(x)=\phi(x)$ for $x\in T_u\cap P$.
	Then $\Phi(H^+\cap P) \leq \widetilde{\Phi}(H^+\cap T_u)$ and therefore $P_{\delta}^\phi\subset (T_u)_{\delta}^{\widetilde{\phi}}$, which yields
	\begin{equation*}
		\liminf_{\delta\to 0^+} \frac{\Psi(S \setminus P_\delta^\phi)}{\delta\left(\ln\frac{1}{\delta}\right)^{n-1}} 
		\geq \liminf_{\delta\to 0^+} \frac{\Psi\left(S \setminus (T_u)_\delta^{\widetilde{\phi}}\right)}{\delta\left(\ln\frac{1}{\delta}\right)^{n-1}}.
	\end{equation*}
	Conversely, by Proposition \ref{prop:flag simplices} ii), there is a flag simplex
	$T_{\ell}$ of $P$ associated with $\mathbf{F}$ such that
	\begin{enumerate}
		\item[i)] $S\subset T_{\ell} \subset P$,
		\item[ii)] $S$ is a flag simplex of $T_{\ell}$.
	\end{enumerate}
	Then $\Phi(H^+\cap P) \geq \Phi(H^+\cap T_{\ell})$ and therefore $P_{\delta}^\phi\supset (T_{\ell})_{\delta}^{\phi}$, which yields
	\begin{equation*}
		\limsup_{\delta\to 0^+} \frac{\Psi(S \setminus P_\delta^\phi)}{\delta\left(\ln\frac{1}{\delta}\right)^{n-1}} 
		\leq\limsup_{\delta\to 0^+} \frac{\Psi(S\setminus (T_{\ell})_\delta^\phi)}{\delta\left(\ln\frac{1}{\delta}\right)^{n-1}}.
	\end{equation*}
	Hence, it is sufficient to prove the proposition for $n$-simplices $T$ and a associated flag simplex $S$, because then
	\begin{equation*}
		\lim_{\delta\to 0^+} \frac{\Psi(S\setminus (T_{\ell})_\delta^\phi)}{\delta\left(\ln\frac{1}{\delta}\right)^{n-1}}
		= \frac{1}{n!\,n^{n-1}} \frac{\psi(v)}{\phi(v)}
		= \lim_{\delta\to 0^+} \frac{\Psi\left(S \setminus (T_u)_\delta^{\widetilde{\phi}}\right)}{\delta\left(\ln\frac{1}{\delta}\right)^{n-1}},
	\end{equation*}
	where $v$ is the vertex of $P$ that is uniquely determined by $S$.
	
	In fact,  using an affine transformation and arguing as in \eqref{eqn:affine_trans},  we may further reduce the problem to the standard simplex $T_n$. So assume that $T=T_n$, 
	that $\phi,\psi:T_n\to (0,\infty)$ are continuous and that  $S$ is a flag simplex of $T_n$ with $v=e_0\in S$.  Proposition \ref{thm:main_theorem_weighted_local}  will follow once we show
	\begin{equation*}
		\lim_{\delta\to 0^+} \frac{\Psi(S \setminus (T_n)_\delta^\phi)}{\delta\left(\ln\frac{1}{\delta}\right)^{n-1}} 
		= \frac{1}{n!\,n^{n-1}} \frac{\psi(e_0)}{\phi(e_0)}.
	\end{equation*}

	First we prove a lower bound for an arbitrary $\varepsilon\in (0,1)$. With the same notation as in the proof of Lemma \ref{lem:main_simplex_upper_weighted},    
	we have
	\begin{equation*}
		T_n\setminus (T_n)_\delta^\phi 
		\overset{\eqref{eqn:wfloat_split}}{\supset} A^+_\phi(T_n,e_0,\delta) 
		\overset{\eqref{eqn:a_lower}}{\supset} A^+_\phi(\varepsilon T_n,e_0,\delta) 
		\overset{\eqref{eqn:a_eps_bound}}{\supset} A^+(\varepsilon T_n,e_0,\delta/\phi_\varepsilon^+(e_0)),
	\end{equation*}
	and by \eqref{eqn:phi_eps}, we conclude
	\begin{equation*}
		\Psi(S\setminus (T_n)_\delta^\phi) \geq \psi_\varepsilon^-(e_0) \, \lambda_n\left((S\cap A^+(\varepsilon T_n,e_0,\delta/\phi_\varepsilon^+(e_0))\right).
	\end{equation*}
	By \eqref{eqn:lambda_concentrate}, this yields
	\begin{equation*}
		\liminf_{\delta\to 0^+} \frac{\Psi(S\setminus (T_n)_\delta^\phi)}{\delta\left(\ln\frac{1}{\delta}\right)^{n-1}} 
		\geq \frac{1}{n!\,n^{n-1}} \frac{\psi_\varepsilon^-(e_0)}{\phi_\varepsilon^+(e_0)}.
	\end{equation*}
	Since $\varepsilon\in(0,1)$ was arbitrary, we obtain
	\begin{equation}\label{eqn:lower_bound}
		\liminf_{\delta\to 0^+} \frac{\Psi(S\setminus (T_n)_\delta^\phi)}{\delta\left(\ln\frac{1}{\delta}\right)^{n-1}} 
		\geq \frac{1}{n!\,n^{n-1}} \frac{\psi(e_0)}{\phi(e_0)}.
	\end{equation}
	
	To prove the upper bound we use a symmetrization of $\phi$ and $\psi$ in the following way:
	As in the proof of \eqref{eqn:lambda_concentrate} in Lemma \ref{lem:eucl_first} we consider the set $\Sigma(n)$ of all permutations $\sigma$ of the integers $\{1,\dotsc,n\}$.  For $\sigma\in \Sigma(n)$ there is an orthogonal transformation $\alpha_\sigma$ such that $\alpha_\sigma(e_i) = e_{\sigma(i)}$. In particular $\alpha_\sigma(T_n)=T_n$ and $\alpha_\sigma(e_0) = e_0$. 
	We define $\widetilde{\phi}, \widetilde{\psi}:T_n\to (0,\infty)$ by
	\begin{align*}
		\widetilde{\phi}(x) &:= \min_{\sigma\in \Sigma(n)} \phi(\alpha_\sigma(x)),&
		\widetilde{\psi}(x) &:= \max_{\sigma\in \Sigma(n)} \psi(\alpha_\sigma(x)).
	\end{align*}
	Then $\widetilde{\phi}$ and $\widetilde{\psi}$ are continuous and 
	\begin{enumerate}
		\item[i)] $\widetilde{\phi}\leq \phi$ and $\psi\leq \widetilde{\psi}$,
		\item[ii)] for all $\sigma \in \Sigma(n)$, $\widetilde{\phi}\circ\alpha_\sigma = \widetilde{\phi}$ and $\widetilde{\psi}\circ\alpha_\sigma = \widetilde{\psi}$, and
		\item[iii)] $\widetilde{\phi}(e_0) =\phi(e_0)$ and $\widetilde{\psi}(e_0)=\psi(e_0)$.
	\end{enumerate}
	We denote by $S_{\id}$ the simplex that has vertices $0,z_1,\dotsc,z_n$, where
	\begin{equation*}
		z_i := \frac{1}{i} \sum_{k=1}^i e_k,\quad \text{for $i=1,\dotsc,n$,}
	\end{equation*}
	and we set $S_\sigma := \alpha_{\sigma}(S_{\id})$, for $\sigma \in \Sigma(n)$.
	Then $(S_\sigma)_{\sigma\in\Sigma(n)}$ dissects $T_n$, i.e., $\interior S_\sigma\cap \interior S_{\sigma'} = \emptyset$ if $\sigma\neq \sigma'$ and
	\begin{equation*}
		T_n = \bigcup_{\sigma\in \Sigma(n)} S_{\sigma}.
	\end{equation*}
	Since $S$ is a flag simplex of $T_n$ with $e_0\in S$, we may assume that $S$ determines the flag $\mathbf{F}=(F_0,\dotsc,F_{n-1})$ where
	\begin{equation*}
		F_i := [e_0,\dotsc,e_i], \quad \text{for $i=0,\dotsc,n-1$.}
	\end{equation*}
	Now fix $\varepsilon\in (0,1)$.
	We have
	\begin{align}\label{eqn:cover}
	\begin{split}
		S\setminus (T_n)_\delta^\phi
		&\subset \left[(\varepsilon S_{\id})\setminus (T_n)_\delta^\phi\right]\\
		&\quad \uplus \left[(\varepsilon T_n)\cap (S\setminus (T_n)_\delta^\phi)\setminus (\varepsilon S_{\id})\right]\\
		&\quad \uplus \left[(S\setminus (T_n)_\delta^\phi)\setminus (\varepsilon T_n)\right].
	\end{split}
	\end{align}
	Since $S$ and $\varepsilon S_{\id}$ are flag simplices of $T_n$ associated with the same flag, we conclude with Theorem \ref{thm:main_theorem_weighted_upper} that
	\begin{equation*}
		\limsup_{\delta\to 0^+}\! \frac{\Psi\big((\varepsilon T_n)\cap (S\setminus (T_n)_\delta^\phi)\setminus (\varepsilon S_{\id})\big)}
			{\delta\left(\ln\frac{1}{\delta}\right)^{n-1}}\\
		\!\leq\! \limsup_{\delta\to 0^+} \frac{\Psi\big((S\setminus (T_n)_\delta^\phi)\setminus (\varepsilon S_{\id})\big)}
			{\delta\left(\ln\frac{1}{\delta}\right)^{n-1}} 
		\!=\! 0,
	\end{equation*}
	and
	\begin{equation*}
		\limsup_{\delta\to 0^+} \frac{\Psi\big((S\setminus (T_n)_\delta^\phi)\setminus (\varepsilon T_n)\big)}
			{\delta\left(\ln\frac{1}{\delta}\right)^{n-1}}
		\leq \limsup_{\delta\to 0^+} \frac{\Psi\big((S\setminus (T_n)_\delta^\phi)\setminus (\varepsilon S_{\id})\big)}
			{\delta\left(\ln\frac{1}{\delta}\right)^{n-1}} 
		=0.
	\end{equation*}
	Thus, \eqref{eqn:cover} gives 
	\begin{equation}\label{eqn:upper}
		\limsup_{\delta\to 0^+} \frac{\Psi(S\setminus (T_n)_\delta^\phi)}{\delta\left(\ln\frac{1}{\delta}\right)^{n-1}} 
		\leq \limsup_{\delta\to 0^+} \frac{\Psi((\varepsilon S_{\id})\setminus (T_n)_\delta^\phi)}{\delta\left(\ln\frac{1}{\delta}\right)^{n-1}}.
	\end{equation}
	Finally, since $\widetilde{\phi}\leq \phi$ and $\psi\leq \widetilde{\psi}$, we have
	\begin{multline}\label{eqn:upper2}
		\Psi((\varepsilon S_{\id})\setminus (T_n)_\delta^\phi) 
		\leq \widetilde{\Psi}((\varepsilon S_{\id})\setminus (T_n)_\delta^{\widetilde{\phi}}) \\
		= \frac{1}{|\Sigma(n)|} \sum_{\sigma\in \Sigma(n)} \widetilde{\Psi}((\varepsilon S_\sigma)\setminus (T_n)_\delta^{\widetilde{\phi}})
		= \frac{1}{n!} \widetilde{\Psi}((\varepsilon T_n)\setminus (T_n)_\delta^{\widetilde{\phi}}).
	\end{multline}
	By Lemma \ref{lem:main_simplex_upper_weighted}, this yields 
	\begin{align*}
		\limsup_{\delta\to 0^+} \frac{\Psi(S\setminus (T_n)_\delta^\phi)}{\delta\left(\ln\frac{1}{\delta}\right)^{n-1}} 
		&\overset{\eqref{eqn:upper}}{\leq} \limsup_{\delta\to 0^+} \frac{\Psi((\varepsilon S_{\id})\setminus (T_n)_\delta^\phi)}{\delta\left(\ln\frac{1}{\delta}\right)^{n-1}} \\
		&\overset{\eqref{eqn:upper2}}{\leq} \frac{1}{n!} \limsup_{\delta\to 0^+} \frac{\widetilde{\Psi}((\varepsilon T_n)\setminus (T_n)_\delta^{\widetilde{\phi}})}{\delta\left(\ln\frac{1}{\delta}\right)^{n-1}}\\
		&\overset{\phantom{(4.19}}{\leq} \frac{1}{n!\,n^{n-1}} \frac{\widetilde{\psi}(e_0)}{\widetilde{\phi}(e_0)} 
		= \frac{1}{n!\,n^{n-1}} \frac{\psi(e_0)}{\phi(e_0)}.
	\end{align*}
	This, together with \eqref{eqn:lower_bound}, implies
	\begin{equation*}
		\lim_{\delta\to 0^+} \frac{\Psi(S\setminus (T_n)_\delta^\phi)}{\delta\left(\ln\frac{1}{\delta}\right)^{n-1}} 
		= \frac{1}{n!\,n^{n-1}} \frac{\psi(e_0)}{\phi(e_0)},
	\end{equation*}
	which concludes the proof.
\end{proof}

The proof of our main Theorem \ref{thm:main_theorem} follows by combining the concentration in flag simplices, Proposition \ref{thm:main_theorem_weighted_upper}, with the limit theorem for flag simplices, Proposition \ref{thm:main_theorem_weighted_local}.
\begin{proof}[{Proof of Theorem \ref{thm:main_theorem}.}]
	Let $(S(\mathbf{F}))_{\mathbf{F}\in \flag(P)}$ be a family flag simplices of $P$ such that $S(\mathbf{F})$ is associated with $\mathbf{F}$. By Lemma \ref{lem:disjoint} we may further assume that $\interior S(\mathbf{F})\cap \interior S(\mathbf{F}')=\emptyset$, if $\mathbf{F}\neq \mathbf{F}'$.
	Then, by Proposition \ref{thm:main_theorem_weighted_local},
	\begin{equation}\label{eqn:main_proof1}
		\liminf_{\delta\to 0^+} \frac{\Psi\big(\!P\setminus P_\delta^\phi\!\big)}{\delta\left(\ln\frac{1}{\delta}\right)^{n-1}}
		\!\geq\! \sum_{\mathbf{F}\in\flag(P)}\!\! \liminf_{\delta\to 0^+}\! \frac{\Psi\big(\!S(\mathbf{F})\setminus P_\delta^\phi\!\big)}{\delta\left(\ln\frac{1}{\delta}\right)^{n-1}}
		\!=\!\frac{1}{n!\,n^{n-1}} \sum_{\mathbf{F}\in \flag(P)} \frac{\psi(v_{\mathbf{F}})}{\phi(v_{\mathbf{F}})},
	\end{equation}
	where $v_{\mathbf{F}}\in\vert P$ is the vertex associated with $\mathbf{F}$.
	Also
	\begin{multline*}
		\limsup_{\delta\to 0^+} \frac{\Psi\left(P\setminus P_\delta^\phi\right)}{\delta\left(\ln\frac{1}{\delta}\right)^{n-1}} 
		\leq \sum_{\mathbf{F}\in \flag(P)} 
			\limsup_{\delta\to 0^+} \frac{\Psi\left(S(\mathbf{F})\setminus P_\delta^\phi\right)}{\delta\left(\ln\frac{1}{\delta}\right)^{n-1}}\\
		+ \limsup_{\delta\to 0^+} \frac{\Psi\big((P\setminus P_\delta^\phi)\setminus \bigcup\{S(\mathbf{F}):\mathbf{F}\in\flag(P)\}\big)}{\delta\left(\ln\frac{1}{\delta}\right)^{n-1}}.
	\end{multline*}
	By Proposition \ref{thm:main_theorem_weighted_upper}, we have
	\begin{equation*}
		\limsup_{\delta\to 0^+} \frac{\Psi\big((P\setminus P_\delta^\phi)\setminus \bigcup\{S(\mathbf{F}):\mathbf{F}\in\flag(P)\}\big)}{\delta\left(\ln\frac{1}{\delta}\right)^{n-1}}
		= 0,
	\end{equation*}
	and therefore, again by Proposition \ref{thm:main_theorem_weighted_local}, we conclude
	\begin{equation}\label{eqn:main_proof2}
		\limsup_{\delta\to 0^+} \frac{\Psi\left(P\setminus P_\delta^\phi\right)}{\delta\left(\ln\frac{1}{\delta}\right)^{n-1}}  
		\leq  \frac{1}{n!\,n^{n-1}}  \sum_{\mathbf{F}\in \flag(P)} \frac{\psi(v_{\mathbf{F}})}{\phi(v_{\mathbf{F}})}.
	\end{equation}
	Now \eqref{eqn:main_proof1} and \eqref{eqn:main_proof2} yield
	\begin{equation*}
		\lim_{\delta\to 0^+} \frac{\Psi\left(P\setminus P_\delta^\phi\right)}{\delta\left(\ln\frac{1}{\delta}\right)^{n-1}}
		= \sum_{v\in\vert P} \frac{\psi(v)}{\phi(v)} \frac{\left|\flag_v(P)\right|}{n!\,n^{n-1}}. \qedhere
	\end{equation*}
\end{proof}

\section{Applications}
We now apply our main Theorem \ref{thm:main_theorem} to spherical and hyperbolic polytopes.
As a general reference for spherical and hyperbolic spaces and in particular convex polytopes in these spaces we refer to \cite[Ch.~6.3]{Ratcliffe:2006}.

\subsection{Spherical polytopes}
Let $\S^n$ be the Euclidean unit sphere in $\R^{n+1}$. 
A proper spherical polytope $P$ is the convex hull of a finite set of vertices contained in an open halfsphere.
Equivalently, $P\subset \S^n$ is a proper spherical convex polytope, if and only if $\operatorname{pos} P:=\{\lambda v : \lambda\geq 0, v\in P\}$ is a closed convex polyhedral cone that contains no line through the origin, i.e., is pointed.

The floating body of a spherical polytope $P$ is defined by
\begin{equation}\label{eqn:def_sphere_float}
	P_\delta^s := \bigcap \big\{H^+\cap P : \vol_n^s(H^-\cap P)\leq \delta\big\}
\end{equation}
where $\vol_n^s$ is the spherical Lebesgue measure and $H^{\pm}$ are the closed halfspheres bounded by the hypersphere $H$ (cf.\ \cite{BesWer:2015}).

Fix $e\in \S^n$. Given a proper spherical convex polytope $P$,  we may assume without loss of generality that $P\subset \S^+_e := \{v\in \S^n: v\cdot e>0\}$. The central projection $g:\S^+_e\to T_e\S^n$ of $\S^+_e$ to the tangent hyperplane $T_e\S^n\cong \R^n$ of $\S^n$ in $e$ is also called \emph{gnomonic projection},
\begin{equation*}
	g(x) = \frac{x}{x\cdot e}-e,
\end{equation*}
(cf.\ \cite{BesSchu:2016}). It maps proper spherical convex polytopes in $\S^+_e$ to Euclidean convex polytopes in $\R^n$. Furthermore, the push-forward of the spherical Lebesgue measure restricted to $\S^+_e$ is absolutely continuous with respect to the Lebesgue measure on $\R^n$ and the density is given by
\begin{equation*}
	\phi^s(x)= (1+\|x\|^2)^{-\frac{n+1}{2}},\quad \forall x\in\R^n.
\end{equation*}
That is, for a  Borel set $A\subset \S^+_e$, we have $\vol_n^s(A)=\int_{g(A)} \phi^s(x)\d{x}$.

Finally, the spherical floating body $P_\delta^s$ is mapped to the weighted floating body $g(P)_\delta^{\phi^s}$ and therefore we may apply our main Theorem \ref{thm:main_theorem} to obtain the following
\begin{theorem}
	Let $P$ be an $n$-dimensional spherical convex polytope in $\S^n$ that is contained in an open halfsphere. Then
	\begin{equation*}
		\lim_{\delta\to 0^+} \frac{\vol_n^s(P)-\vol_n^s(P^s_\delta)}{\delta \left(\ln\frac{1}{\delta}\right)^{n-1}} = \frac{\left|\flag(P)\right|}{n!\,n^{n-1}}.
	\end{equation*}
\end{theorem}
More generally $P\subset \S^n$ may be called a spherical polytope (or spherical polyhedra) if either
\begin{enumerate}
	\item[i)] $P$ is contained in an open halfsphere (proper spherical polytope), or
	\item[ii)]  $P$ is a great $k$-sphere of $\S^n$, or
	\item[iii)] $P$ is a $k$-lune, that is, $P$ is the convex hull of a great $k$-sphere $L$ and a proper spherical polytope $Q\subset L^\circ$.
\end{enumerate}
Here $L^\circ$ is the great $(n-k-1)$-sphere that is polar to $L$ (compare also \cite[Thm.\ 6.3.16]{Ratcliffe:2006}).
If $P$ is a $k$-sphere, then $\vol_n^s(P)=0$ and therefore we cannot define a spherical floating body of $P$. However, for $k$-lunes $P$ the Definition \eqref{eqn:def_sphere_float} still applies and we may consider the volume difference between $P$ and $P_\delta^s$.

For instance, $(n-1)$-lunes are closed halfspheres and for the halfsphere $\S^+_e$ the spherical floating body of $\S_e^+$ is a geodesic ball with center $e$ and radius
\begin{equation*}
	\alpha = \frac{\pi}{2}-\frac{2\pi\delta}{\vol_n^s(\S^n)}.
\end{equation*}
To see this, we just note that for the intersection of two closed halfspheres $\S_{e}^+$ and $\S_{v}^+$, we have 
\begin{equation*}
	\vol_n^s(\S_{e}^+\cap \S_{v}^+) = \frac{\vol_n^s(\S^n)}{2} \left(1- \frac{d(e,v)}{\pi}\right),
\end{equation*}
where $d(e,v)$ is the angle between $e$ and $v$.
Hence 
\begin{equation*}
	\vol_n^s(\S^+_e) - \vol_n^s((\S^+_e)_\delta^s) 
	\!=\! \vol_{n-1}^s(\S^{n-1}) \int_{\alpha}^{\frac{\pi}{2}} (\sin t)^{n-1} \d{t}
	\!=\! \frac{\vol_{n-1}^s(\S^{n-1})}{\vol_{n}^s(\S^n)} 2\pi \delta + o(\delta),
\end{equation*}
as $\delta\to 0^+$.
We conclude
\begin{equation*}
	\lim_{\delta \to 0^+} \frac{\vol_n^s(\S^+_e) - \vol_n^s((\S^+_e)_\delta^s)}{\delta} 
	= \frac{2\pi\vol_{n-1}^s(\S^{n-1})}{\vol_n^s(\S^n)} = 2\sqrt{\pi}\,\frac{\Gamma(\frac{n+1}{2})}{\Gamma(\frac{n}{2})}.
\end{equation*}
Thus the volume difference between an $(n-1)$-lune and its floating body is of order $\delta$. This is particularly interesting, since no Euclidean convex body has a similar asymptotic.
It would be very interesting to determine the asymptotic behavior of $k$-lunes for $k\in\{0,\dotsc,n-2\}$.

In \cite{BHRS:2017},  the expected volume difference of a closed halfsphere and a uniform random polytope inside the closed halfsphere were considered 
and similar phenomena were observed.
\subsection{Hyperbolic polytopes}

Let $\H^n$ be the hyperbolic $n$-space and denote by $\vol_n^h$ the natural volume measure.
It is a curiosity of hyperbolic space that geodesically convex closed subsets may be unbounded, but still have finite volume. Hence, concerning convex polytopes in hyperbolic space, we should distinguish between polytopes with finite volume and the subset of compact convex polytopes.
The (geodesically) convex hull of a finite number of points such that the convex hull has non-empty interior  is an $n$-polytopes in $\H^n$ (cf.\ \cite[Thm.~6.3.17]{Ratcliffe:2006}). Hence, an $n$-polytope in $\H^n$ is always bounded.

In the projective model of hyperbolic space,  $n$-polytopes in hyperbolic space can be identified with Euclidean $n$-polytopes in the open unit ball. To be more precise, the projective model is the metric space $(\B^n,d_h)$, where $\B^n\subset \R^n$ is the open unit ball and the hyperbolic distance $d_h$ between two points $x,y\in\B^n$ is defined by
\begin{equation*}
	d_h(x,y) = \frac{1}{2}\ln\left( \frac{\|y-p\|}{\|x-p\|}\frac{\|x-q\|}{\|y-q\|}\right),
\end{equation*}
where $p,q$ are the intersection points of the line spanned by $x$ and $y$ with the ``sphere at infinity'' $\S^{n-1}=\bd\B^n$ and the four points $p,x,y,q$ are in that order on the line. Also, $\|{\cdot}\|$ is just the standard Euclidean norm.
	
Geodesics in the projective model $(\B^n,d_h)$ are  straight lines of $\R^n$ intersected with $\B^n$. 
Therefore, the geodesically convex hull of a finite number of points in $\B^n$ is the same as the Euclidean convex hull. 
Hence any $n$-polytope of $\H^n$ can be identified with a polytope $P\in\mathcal{K}(\R^n)$ with $P\subset \B^n$.
The hyperbolic volume in $(\B^n,d_h)$ is absolutely continuous to $\lambda_n$ and the density is given by
\begin{equation*}
	\phi^h(x) := \frac{\d{\vol_n^h}}{\d\lambda_n}(x) = (1-\|x\|^2)^{-\frac{n+1}{2}}, \quad  \forall x\in \B^n,
\end{equation*}
(cf.\ \cite[(3.7)]{BesWer:2016}). 
The hyperbolic floating body in $(\B^n,d_h)$ is $P_{\delta}^h = P_{\delta}^{\phi^h}$.  See also \cite[Sec.~4]{BesWer:2016}). 
Thus, by applying our main Theorem \ref{thm:main_theorem} with $\phi=\psi=\phi^h$ in $(B^n,d_h)$ we conclude
\begin{theorem}\label{thm:main_theorem_hyperbolic}
	Let $P$ be an $n$-dimensional compact geodesically convex polytope in $\H^n$ and let
	\begin{equation*}
		P_\delta^h := \bigcap \big\{ H^- : \vol_n^h(P\cap H^+)\leq \delta\big\},
	\end{equation*}
	be the hyperbolic floating body of $P$ for $\delta\in(0,\frac{1}{2}\vol_n^h(P))$. 
	Here $H^{\pm}$ are the closed half-spaces that are bounded by the totally-geodesic hypersurface $H$.
	Then
	\begin{equation*}
		\lim_{\delta\to 0^+} \frac{\vol_n^h(P)-\vol_n^h(P^h_\delta)}{\delta \left(\ln\frac{1}{\delta}\right)^{n-1}} = \frac{\left|\flag(P)\right|}{n!\,n^{n-1}}.
	\end{equation*}
\end{theorem}

In \cite[p.\ 226]{Ratcliffe:2006} a generalized polytope $P$ in $(\B^n,d_h)$ is defined as the convex hull of a finite set of points in $\B^n\cup \S^{n-1}$.
A generalized polytope has finite volume, but may contain vertices at infinity and is therefore in general not bounded.
We call a generalized polytope ideal if all vertices of $P$ are at infinity, i.e., contained in $\S^{n-1}$.
We expect ideal polytopes to behave differently than compact polytopes. So far,   we can verify this in the hyperbolic plane.

\begin{theorem}\label{thm:ideal_triangle}
	For an ideal triangle $T$ in $(\B^2,d_h)$, there are $\delta_0>0$ and positive constants $C_1,C_2$ such that
	\begin{equation*}
		C_1\leq \frac{\vol_2^h(T)-\vol_2^h(T_\delta^h)}{\delta} \leq C_2,
	\end{equation*}
	for all $0<\delta<\delta_0$.
\end{theorem}

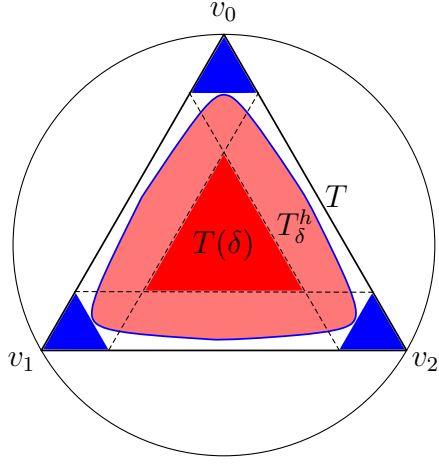
\begin{figure}
	\centering
	\begin{tikzpicture}[scale=0.50]
		\path[use as bounding box] (-6.3,-6.1) rectangle (6.3,6.5);
		\node at (0,0) {\includegraphics[width=0.44\textwidth]{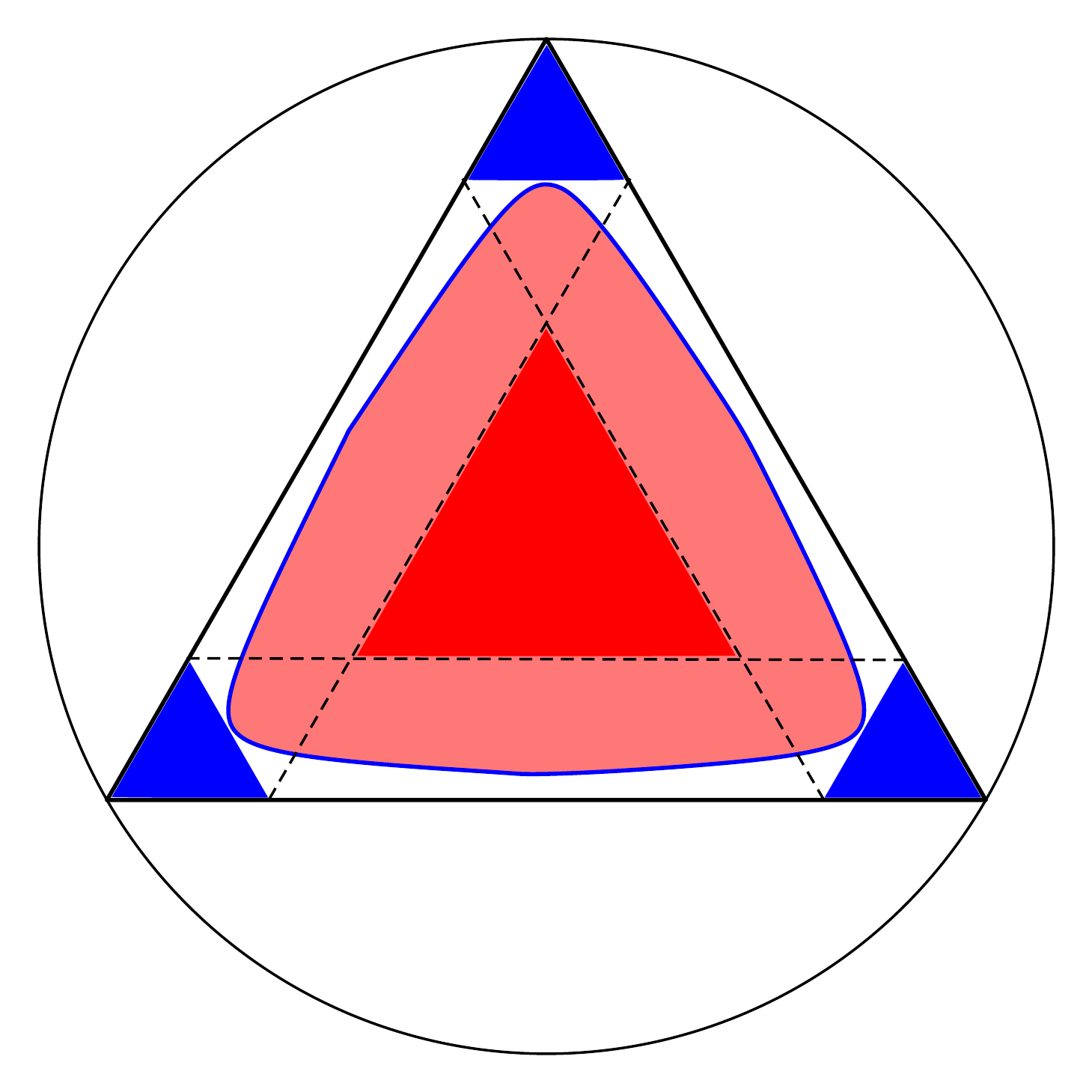}};
		\node[] at (0,0) {$T(\delta)$};
		\node[] at (1.9,0.5) {$T_\delta^h$};
		\node at (3,1.3) {$T$};
		\node at (0, 6.2) {$v_0$};
		\node at (210:6.2) {$v_1$};
		\node at (-30:6.2) {$v_2$};
	\end{tikzpicture}
	\caption{Constructions used in the proof of Theorem \ref{thm:ideal_triangle}. $T$ is the Euclidean regular triangle which we identify with the ideal hyperbolic triangles. $T_\delta^h$ is the hyperbolic floating body and $T(\delta)$ is a scaled triangle that is contained in $T_\delta^h$. Each of the blue colored caps $T\cap H^+(v_i,1-h_0(\delta))$,  $i=0,1,2$, cut off hyperbolic area $\delta$ from $T$ and touch the floating body $T_\delta^h$.}
	\label{fig:hyperbolic_triangle}
\end{figure}

\begin{proof}
	All ideal triangles in the hyperbolic plane are congruent, and therefore we may consider $T$ to be the Euclidean regular triangle spanned by the vertices
	$v_0 = (1,0)$, $v_1=(-\frac{1}{2},\frac{\sqrt{3}}{2})$, and $v_2=(-\frac{1}{2},-\frac{\sqrt{3}}{2})$.
	
	First, we establish the lower bound. Observe that $T\setminus T^h_\delta$ contains the caps $\biguplus_{i=0}^2 T\cap H^+(v_i,1-h(\delta))$, where $h_0(\delta)$ is the height determined by
	\begin{equation*}
		\delta = \vol_2^h\big(T\cap H^+(v_0,1-h_0(\delta))\big).
	\end{equation*}
	The caps $T\cap H^+(v_i,1-h_0(\delta))$, $i=0,1,2$, are all congruent by a rotation of $\pm \frac{2\pi}{3}$.
	We conclude
	\begin{equation*}
		\vol_2^h(T)-\vol_2^h(T_\delta^h) \geq \sum_{i=0}^2 \vol_2^h(T\cap H^+(v_i,1-h_0(\delta)) = 3\delta.
	\end{equation*}
	
	For the upper bound, let,  as above, $h_0(\delta)$  be the height of the cap $T\cap H^+(v_0,1-h_0(\delta))$ that has hyperbolic area $\delta$. The scaled triangle $T(\delta) := (1-2h_0(\delta))T$ is contained in the floating body $T_\delta^h$, because any half-plane $H^+$ that intersects $T(\delta)$ will contain at least one of the caps $T\cap H^+(v_i,1-h_0(\delta))$ and therefore $\vol_2^h(T\cap H^+)\geq \vol_2^h(T\cap H^+(v_i,1-h_0(\delta)))=\delta$ (see Figure \ref{fig:hyperbolic_triangle}). Thus $T(\delta)\subset T_\delta^h$. 
	Since
	\begin{align*}
		\vol_2^h\big(T\cap H^+(v_0,1-h_0(\delta))\big) 
		&= \int_{1-h_0(\delta)}^1 \int_{-\frac{1-x}{\sqrt{3}}}^{\frac{1-x}{\sqrt{3}}} \frac{1}{(1-x^2-y^2)^{\frac{3}{2}}}\d{y}\d{x}\\
		&= \int_{1-h_0(\delta)}^1 \frac{1}{(1-x^2)} \left.\left(\frac{y}{\sqrt{1-y^2}}\right)\right|_{y=-\frac{1}{\sqrt{3}} \sqrt{\frac{1-x}{1+x}}}^{\frac{1}{\sqrt{3}} \sqrt{\frac{1-x}{1+x}}}\d{x}\\
		&= \int_{0}^{h_0(\delta)} \frac{\sqrt{2}}{(2-s)\sqrt{s}\sqrt{3-2s}}\d{s} \geq \sqrt{\frac{2h_0(\delta)}{3}},
	\end{align*}
	we conclude $h_0(\delta) \leq \frac{3}{2} \delta^2$. Thus
	$T_\delta^h \supset T(\delta) \supset (1-3\delta^2) T$,
	which yields
	\begin{equation*}
		\vol_2^h(T)-\vol_2^h(T_\delta^h) \leq \vol_2^h(T) - \vol_2^h\left((1-3\delta^2) T\right).
	\end{equation*}
	The hyperbolic area of this annulus is bounded above by three times the hyperbolic area of the cap $T\cap H^+(n,\frac{1-3\delta^2}{2})$, where $n=-v_0=(-1,0)$ is the normal vector of an edge of $T$, that is,
	\begin{equation*}
		\vol_2^h(T) - \vol_2^h\left((1-3\delta^2) T\right) \leq 3 \vol_2^h\left(T\cap H^+\left(n,\frac{1-3\delta^2}{2}\right)\right).
	\end{equation*}
	We calculate
	\begin{align*}
		\vol_2^h\left(T\cap H^+\left(n,\frac{1-3\delta^2}{2}\right)\right) 
		&= \int_{-\frac{1}{2}}^{-\frac{1}{2}+\frac{3}{2}\delta^2} \int_{-\frac{1-x}{\sqrt{3}}}^{\frac{1-x}{\sqrt{3}}} \frac{1}{(1-x^2-y^2)^{\frac{3}{2}}} \d{y}\d{x}\\
		&= \int_{0}^{\frac{3}{2}\delta^2} \frac{1}{(\frac{1}{2}+s)\sqrt{s}\sqrt{\frac{3}{2}-s}} \d{s} 
		\leq 4\delta + 4\delta^2,
	\end{align*}
	for all $\delta\in\left(0,\frac{1}{2}\right)$, and conclude
	\begin{equation*}
		\limsup_{\delta\to 0^+} \frac{\vol_2^h(T)-\vol_2^h(T_\delta^h)}{\delta} \leq 12.\qedhere
	\end{equation*}
\end{proof}

From Theorem \ref{thm:ideal_triangle} it is clear that $\lim_{\delta\to 0^+} \frac{\vol_2^h(T)-\vol_2^h(T_\delta^h)}{\delta}$ exists. 
Refining the arguments, one may conclude, that for any ideal polygon $P$ of $(\B^2,d_h)$ the limit 
\begin{equation*}
	\lim_{\delta\to 0^+} \frac{\vol_2^h(P)-\vol_2^h(P_\delta^h)}{\delta}
\end{equation*}
exists. However, calculating the exact value appears to be difficult.
	
Note that for a regular ideal polygon $P$ the envelope $E^h_{[\delta]}$ of all lines that cut off hyperbolic area  $\delta$ from $P$ touches the vertices at infinity of $P$ if $\delta$ is small enough.  See Figure \ref{fig:hyperbolic_triangle_envelope} for sketches of an ideal triangle and an ideal regular pentagon in the projective model and in the Poincar{\'e} model of the hyperbolic plane. The reason for this behavior is, that the hyperbolic area of caps that cut off one vertex and caps that cut off one edge are asymptotic to $\sqrt{h_\delta}$, where $h_\delta$ is the height of the cap. This was observed for the ideal triangle in the proof of Theorem \ref{thm:ideal_triangle} and holds true for general ideal polygons. This behavior is of course vastly different from what we experience in Euclidean space, where the Euclidean area of a cap that cuts off a vertex is asymptotic to $h_\delta^2$ and a cap that cuts off an edge is asymptotic to $h_\delta$.

\begin{remark}[Floating body of regular ideal polygons]
	The floating body of ideal regular polygons has a remarkable behavior, as can already be inferred from Figure \ref{fig:hyperbolic_octagon} for the ideal regular octagon. For instance, unlike the Euclidean floating body which is always strictly convex \cite{SchuettWerner:1994}, the hyperbolic floating body of an ideal regular polygon is a hyperbolic polygon for special values of $\delta$. Apparently the behavior of hyperbolic floating bodies of polytopes has not been studied yet, and it seems that the hyperbolic floating body of ideal regular $n$-polytopes has very special role in this family and is therefore of particular interest.
\end{remark}

\bigskip
\footnotesize
\noindent\textit{Acknowledgments.}
F.~Besau is partly supported by DFG grant BE 2484/5-2 (Proj.-Nr.: 215529441). \\E.~M.~Werner is partly supported by NSF grant DMS-1504701.


\begin{figure}[p]
	\centering
	\begin{tikzpicture}[scale=0.33]
		\path[use as bounding box] (-10,-10) rectangle (30,10);
		\node at (0,0) {\includegraphics[width=0.45\textwidth]{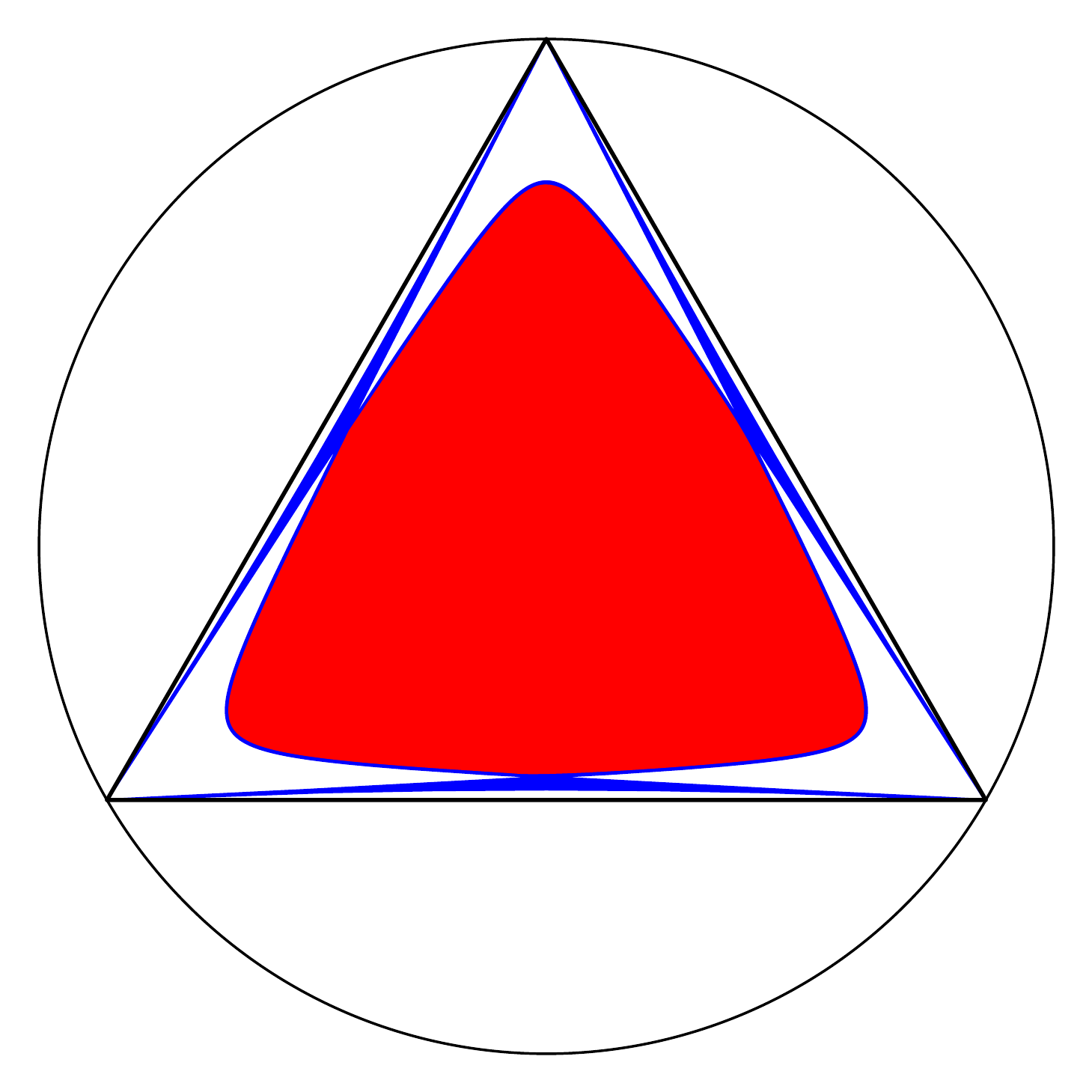}};
		\node at (20,0) {\includegraphics[width=0.45\textwidth]{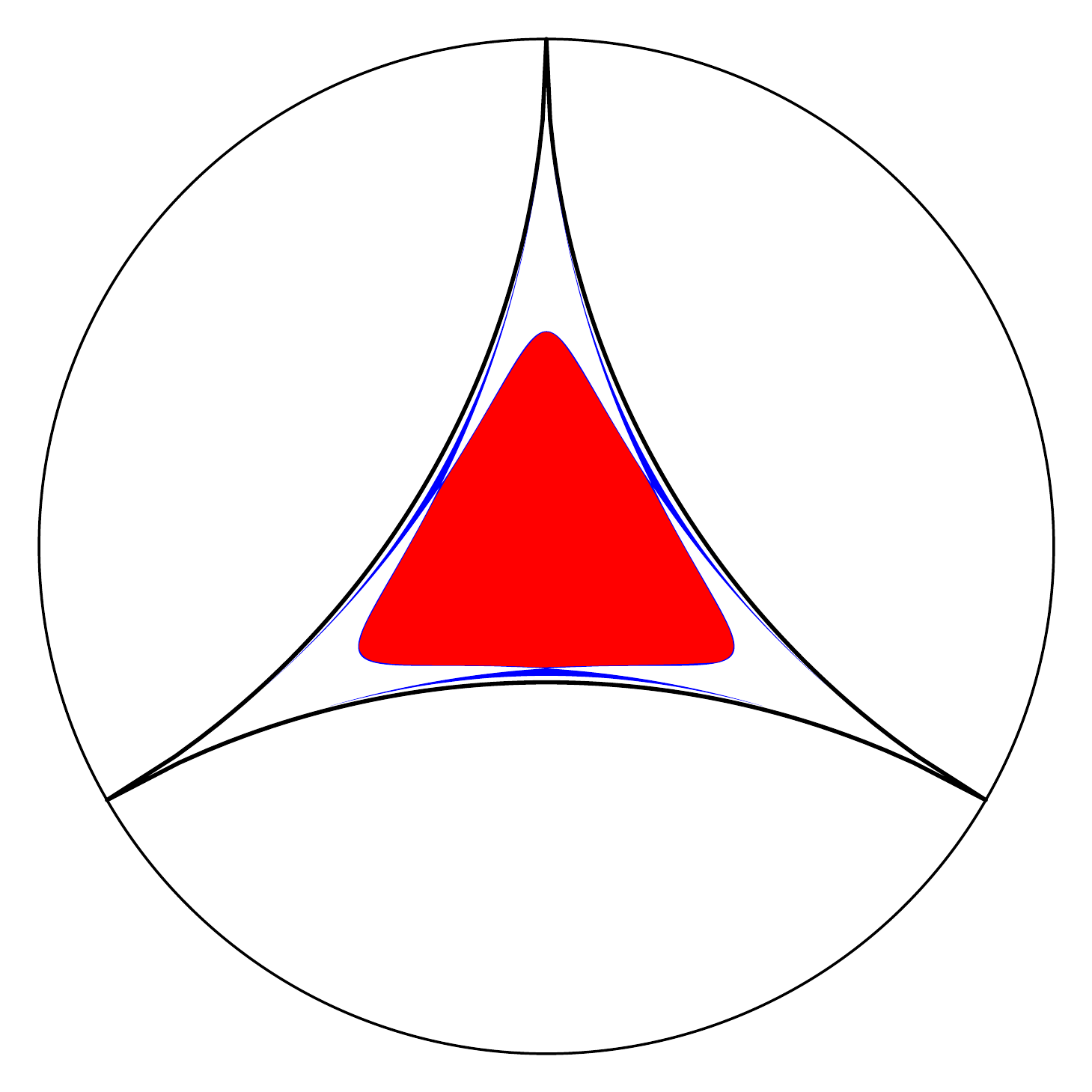}};
		\node[] at (0,0) {$T_\delta^h$};
		\node[] at (20,0) {$T_\delta^h$};
		\node at (6,2.3) {$T$};
		\node at (24,1.3) {$T$};
		\node[] at (-7.5,-0.5) {$E^h(\delta)$};
		\node[] at (16,2.5) {$E^h(\delta)$};
	\end{tikzpicture}\\
	\begin{tikzpicture}[scale=0.33]
		\path[use as bounding box] (-10,-10) rectangle (30,10);
		\node at (0,0) {\includegraphics[width=0.45\textwidth]{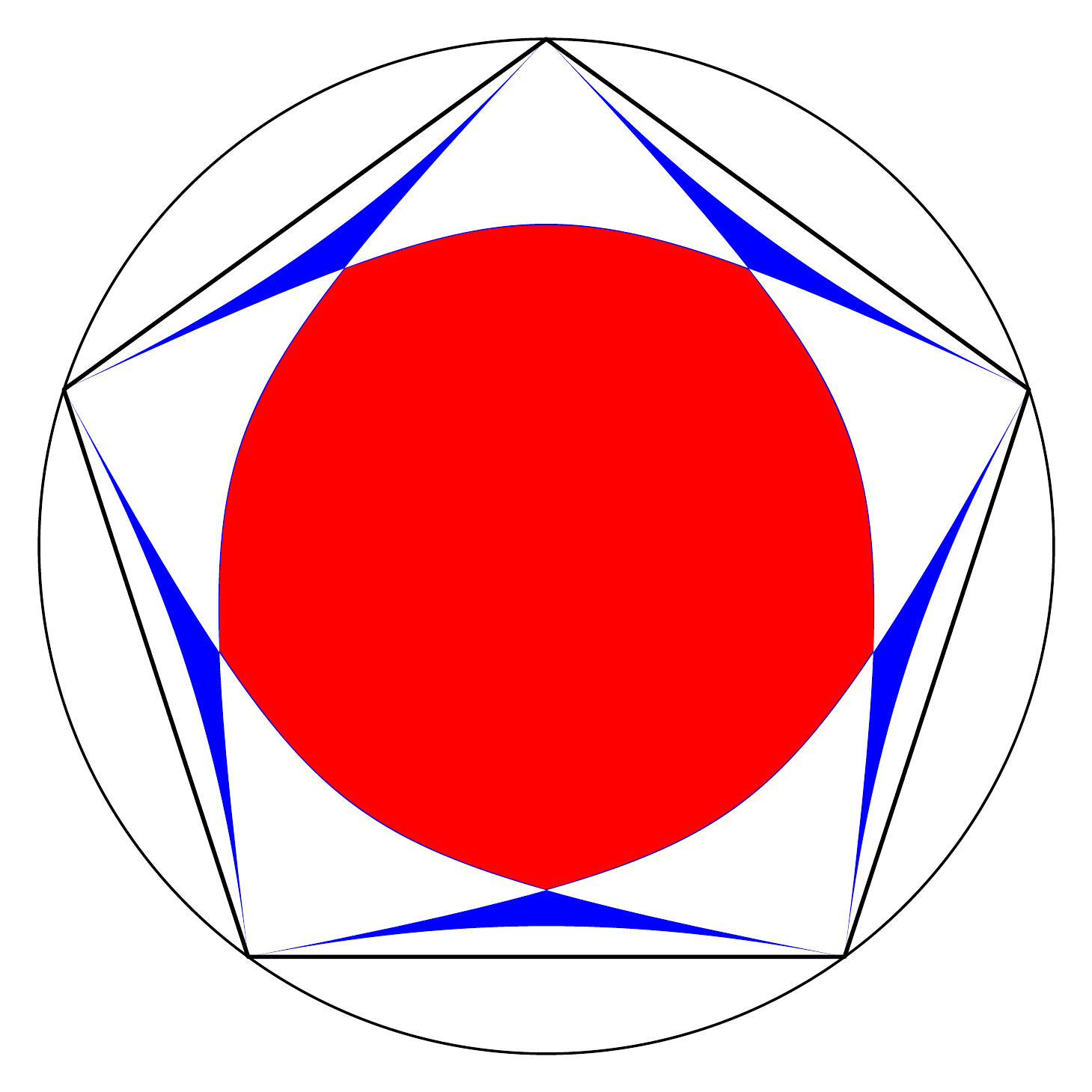}};
		\node at (20,0) {\includegraphics[width=0.45\textwidth]{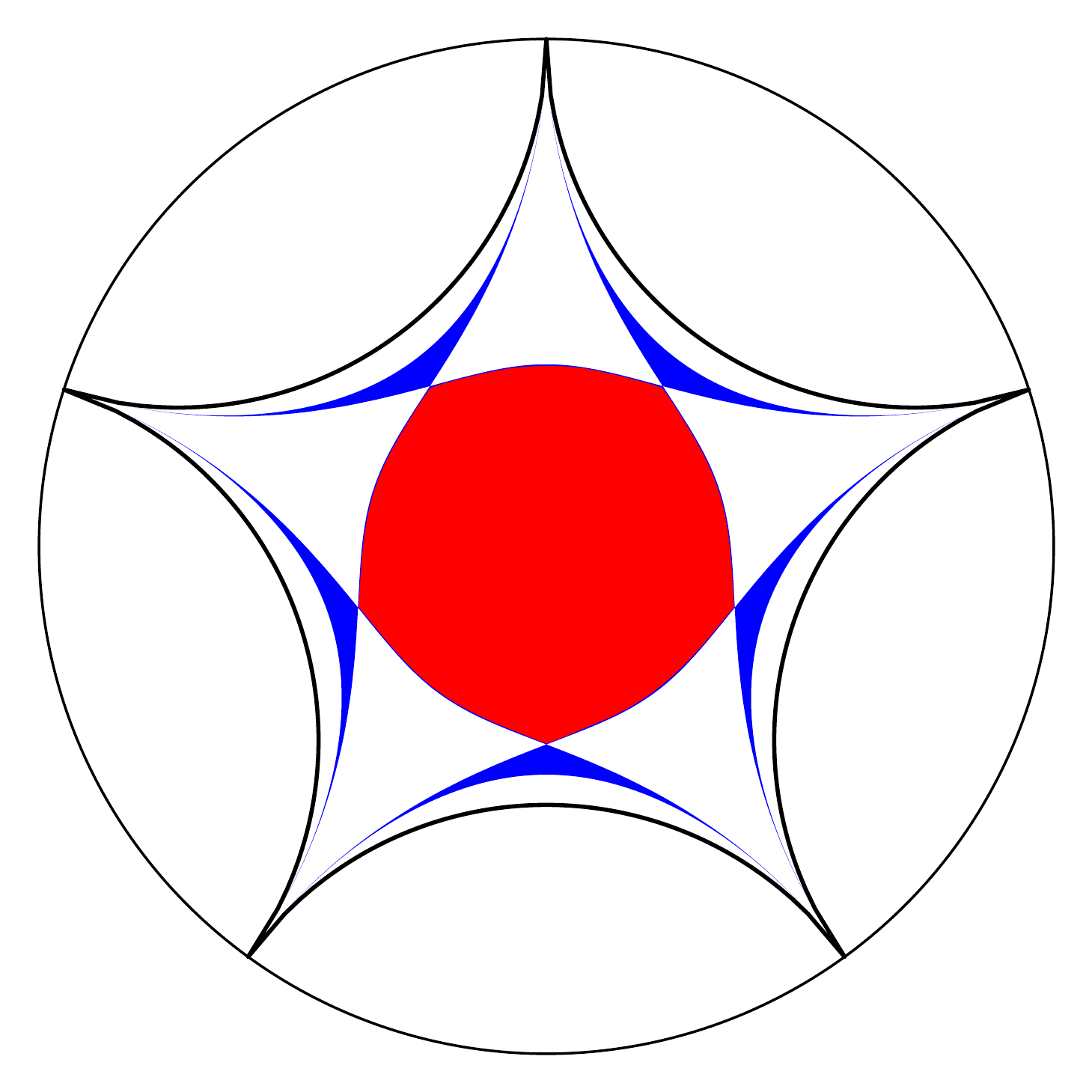}};
		\node[] at (0,0) {$P_\delta^h$};
		\node[] at (20,0) {$P_\delta^h$};
		\node at (6,6) {$P$};
		\node at (24,4.3) {$P$};
		\node[] at (-4,-6.5) {$E^h(\delta)$};
		\node[] at (16,5) {$E^h(\delta)$};
	\end{tikzpicture}
	\caption{The figures show in blue the envelope $E^h(\delta)$ in the projective model (on the left) and in the Poincar{\'e} disk model (on the right) of the hyperbolic plane. In red we see the convex hyperbolic floating body $T_\delta^h$ of the ideal triangle, respectively, the floating body $P_\delta^h$ of the ideal regular pentagon $P$.}
	\label{fig:hyperbolic_triangle_envelope}
\end{figure}

\begin{figure}[p]
	\centering
	\hfill
	\begin{tikzpicture}[scale=0.29]
		\path[use as bounding box] (-10,-10) rectangle (10,10);
		\node at (0,0) {\includegraphics[width=0.44\textwidth]{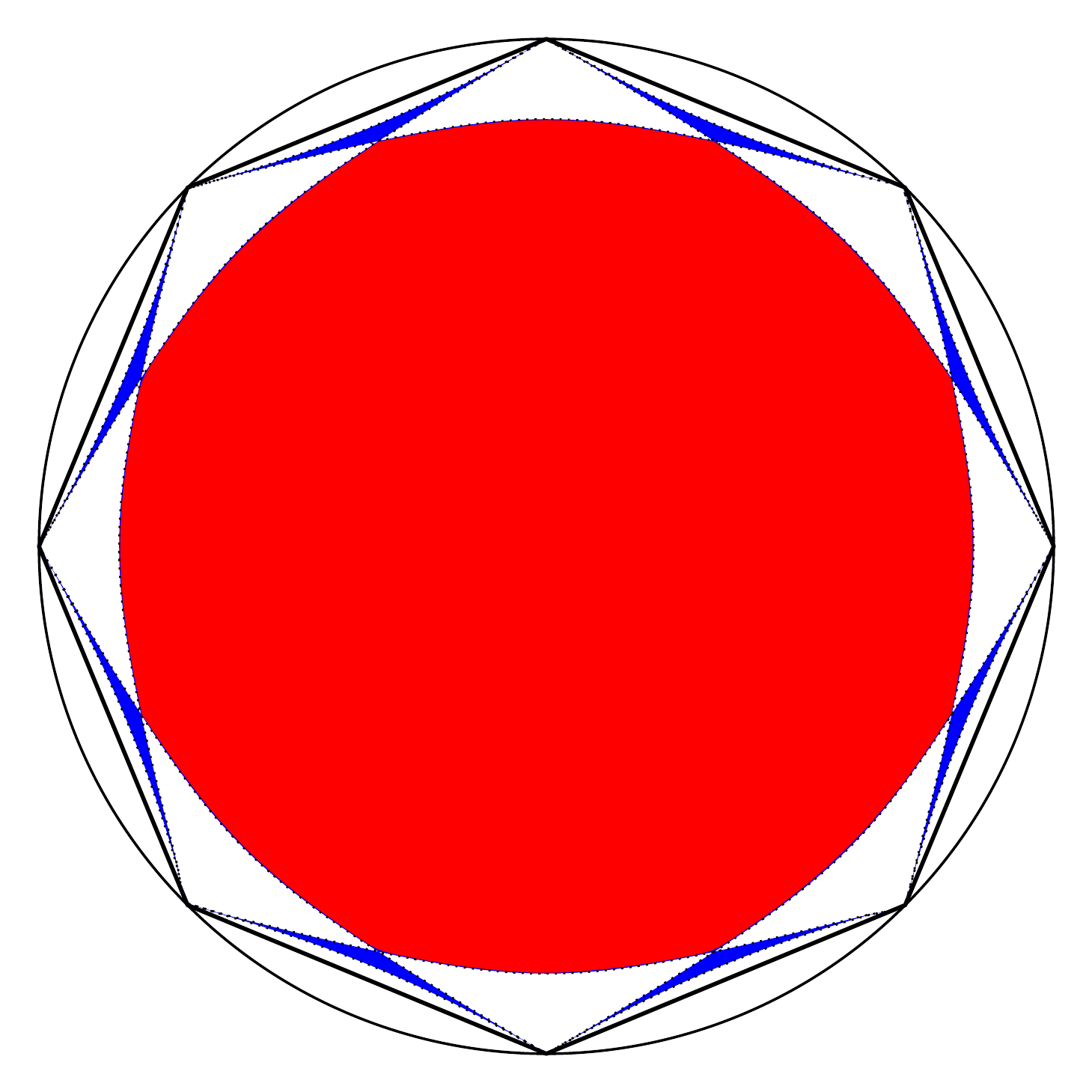}};
		\node[] at (0,0) {$O_{\pi/2}^h$};
		\node[] at (-4,-6.5) {$E^h(\pi/2)$};
	\end{tikzpicture}\hfill
	\begin{tikzpicture}[scale=0.29]
		\path[use as bounding box] (-10,-10) rectangle (10,10);
		\node at (0,0) {\includegraphics[width=0.44\textwidth]{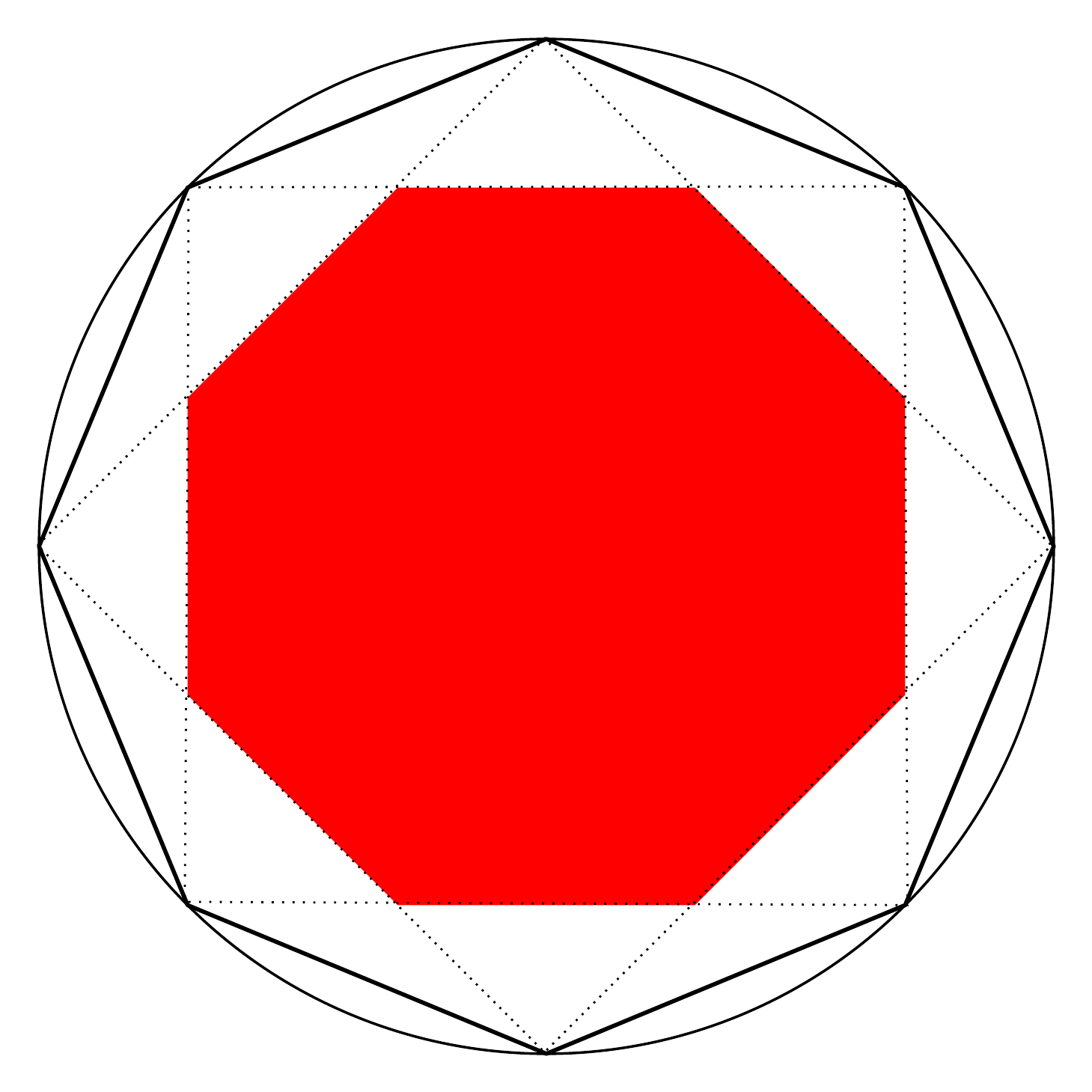}};
		\node[] at (0,0) {$O_{\pi}^h$};
	\end{tikzpicture}\hfill$ $\\[0.1cm]
	\hfill
	\begin{tikzpicture}[scale=0.29]
		\path[use as bounding box] (-10,-10) rectangle (10,10);
		\node at (0,0) {\includegraphics[width=0.44\textwidth]{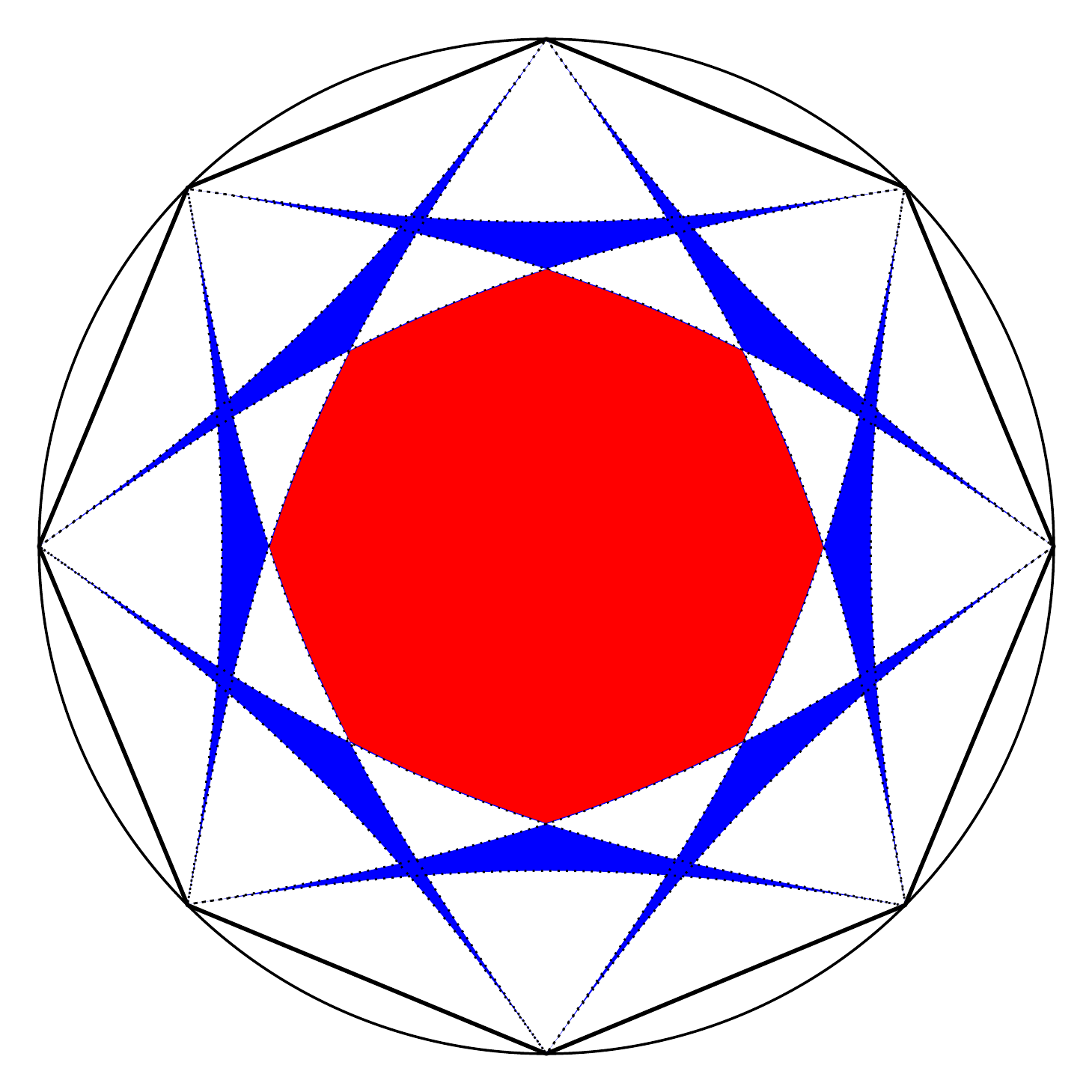}};
		\node[] at (0,0) {$O_{3\pi/2}^h$};
		\node[] at (-4,-6.5) {$E^h(3\pi/2)$};
	\end{tikzpicture}\hfill
	\begin{tikzpicture}[scale=0.29]
		\path[use as bounding box] (-10,-10) rectangle (10,10);
		\node at (0,0) {\includegraphics[width=0.44\textwidth]{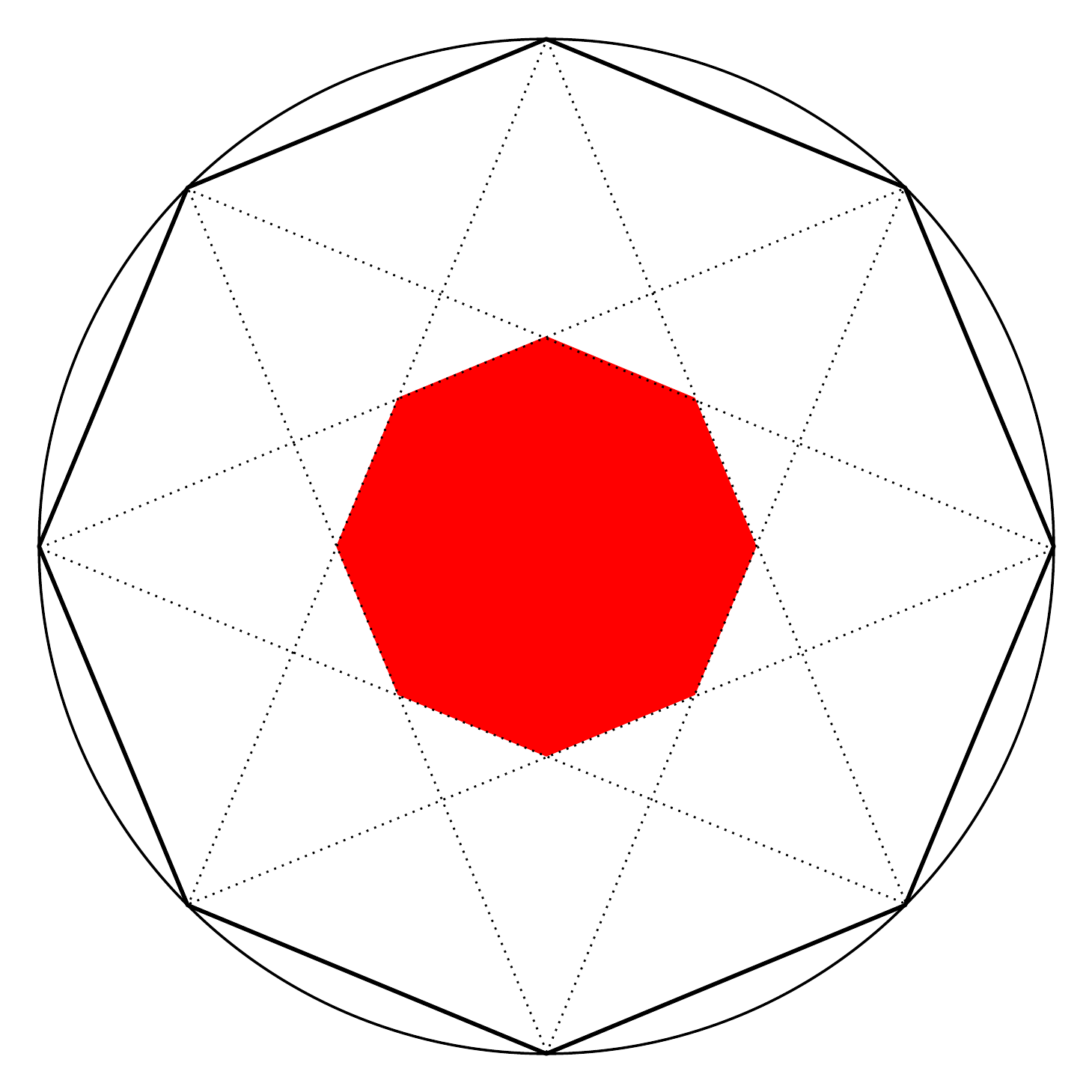}};
		\node[] at (0,0) {$O_{2\pi}^h$};
	\end{tikzpicture}\hfill$ $\\[0.1cm]
	\hfill
	\begin{tikzpicture}[scale=0.29]
		\path[use as bounding box] (-10,-10) rectangle (10,10);
		\node at (0,0) {\includegraphics[width=0.44\textwidth]{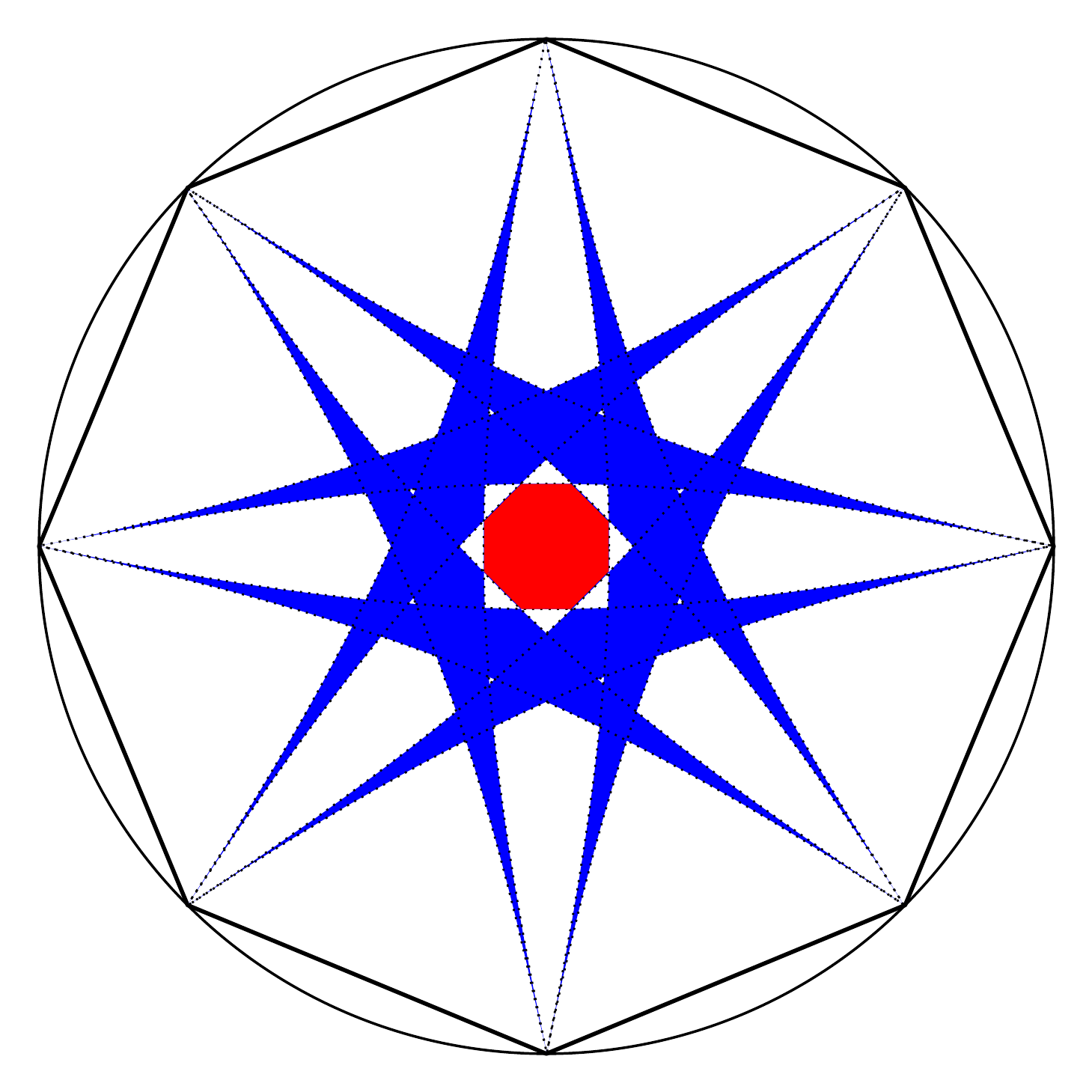}};
		\node[] at (0,0) {$O_{5\pi/2}^h$};
		\node[] at (-4,-6.5) {$E^h(5\pi/2)$};
	\end{tikzpicture}\hfill
	\begin{tikzpicture}[scale=0.29]
		\path[use as bounding box] (-10,-10) rectangle (10,10);
		\node at (0,0) {\includegraphics[width=0.44\textwidth]{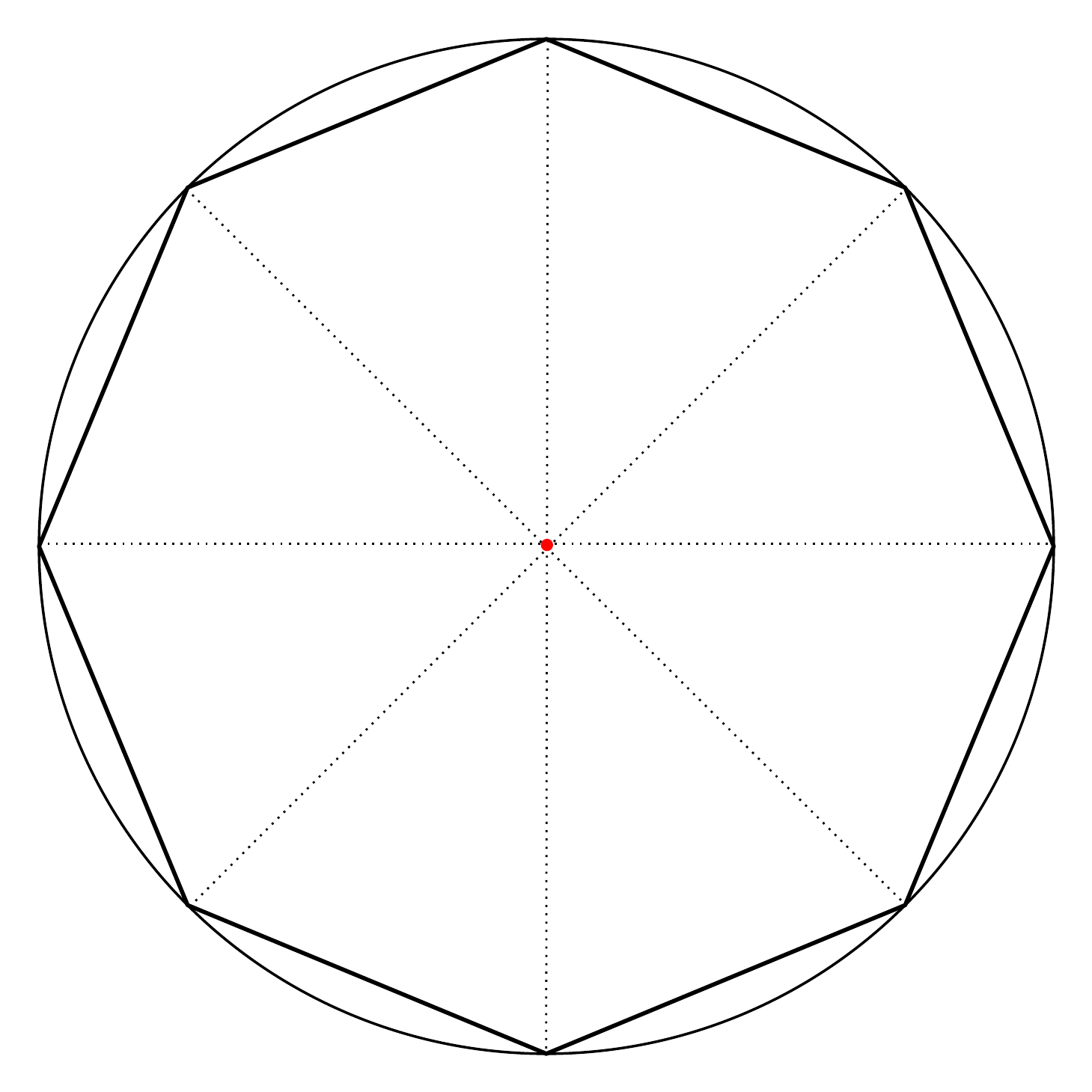}};
		\node[] at (1,1) {$O_{3\pi}^h$};
	\end{tikzpicture}\hfill$ $
	\caption{The envelope $E^h(\delta)$ (in blue) and the hyperbolic floating body $O_\delta^h$ (in red) of the regular ideal octagon $O$ for $\delta=\frac{\pi}{2}, \pi, \frac{3\pi}{2}, 2\pi, \frac{5\pi}{2}, 3\pi$ ($\vol_2^h(O)=6\pi$).}
	\label{fig:hyperbolic_octagon}
\end{figure}


\begin{thebibliography}{99}

\normalsize
\baselineskip=17pt

\bibitem{AFZ:2017}
	M.~Alexander, M.~Fradelizi and A.~Zvavitch,
	\emph{Polytopes of maximal volume product},
	2017, arXiv: \url{https://arxiv.org/abs/1708.07914}.
	
\bibitem{AAKSW:2012}
	S.~Artstein-Avidan, B.~Klartag, C.~Schütt and E.~M.~Werner,
	\emph{Functional affine-isoperimetry and an inverse logarithmic Sobolev inequality},
	J.\ Funct.\ Anal.\ \textbf{262} (2012), 4181--4204. 

\bibitem{BarBu:1993}
	I.~B{\'a}r{\'a}ny and C.~Buchta,
	\emph{Random polytopes in a convex polytope, independence of shape, and concentration of vertices},
	Math.\ Ann.\ \textbf{297} (1993), 467--497. 
  
\bibitem{BHRS:2017}  
	I.~B\'ar\'any, D.~Hug, M.~Reitzner and R.~Schneider, 
	\emph{Random points in halfspheres}, 
	Random Structures Algorithms \textbf{50} (2017), 3--22.

\bibitem{BarLar:1988}
	I.~B{\'a}r{\'a}ny and D.~G. Larman, 
	\emph{Convex bodies, economic cap coverings, random polytopes}, 
	Mathematika \textbf{35} (1988), 274--291.

\bibitem{BarLov:1982}
	I.~B{\'a}r{\'a}ny and L.~Lov{\'a}sz,
	\emph{Borsuk's theorem and the number of facets of centrally symmetric polytopes}, 
	Acta Math.\ Acad.\ Sci.\ Hungar.\ \textbf{40} (1982), 323--329. 

\bibitem{BBF:2014}
	F.~Barthe, K.~J.~Böröczky and M.~Fradelizi,
	\emph{Stability of the functional forms of the Blaschke-Santal{\'o} inequality},
	Monatsh.\ Math.\ \textbf{173} (2014), 135--159. 

\bibitem{BN:2008}
	A.~Barvinok and I.~Novik,
	\emph{A centrally symmetric version of the cyclic polytope},
	Discrete Comput.\ Geom.\ \textbf{39} (2008), 76--99. 
	
\bibitem{BLN:2013}
	A.~Barvinok, S.~J.~Lee and I.~Novik,
	\emph{Centrally symmetric polytopes with many faces},
	Israel J.\ Math.\ \textbf{195} (2013), 457--472. 
	
\bibitem{Bayer:1987}
	M.~Bayer,
	\emph{The extended $f$-vectors of $4$-polytopes},
	J.\ Combin.\ Theory Ser.\ A \textbf{44} (1987), 141--151.
	
\bibitem{BayerBillera:1985}
	M.~Bayer and L.~J.~Billera,
	\emph{Generalized Dehn-Sommerville relations for polytopes, spheres and Eulerian partially ordered sets},
	Invent.\ Math.\ \textbf{79} (1985), 143--157. 

\bibitem{BesSchu:2016}
	F.~Besau and F.~E.~Schuster, 
	\emph{Binary operations in spherical convex geometry},
	Indiana Univ.\ Math.\ J.\ \textbf{65} (2016), 1263--1288.
  
\bibitem{BesWer:2015}
	F.~Besau and E.~M.~Werner,
	\emph{The spherical convex floating body},
	Adv.\ Math.\ \textbf{301} (2016), 867--901.

\bibitem{BesWer:2016}
	F.~Besau and E.~M.~Werner,
	\emph{The floating body in real space forms},
	J.\ Differential Geom., in press.
	
\bibitem{BLW:2017}
	F.~Besau, M.~Ludwig, and E.~M.~Werner,
	\emph{Weighted floating bodies and polytopal approximation},
	Trans.\ Amer.\ Math.\ Soc., in press.

\bibitem{Billera:2010}
	L.~J.~Billera,
	\emph{Flag enumeration in polytopes, Eulerian partially ordered sets and Coxeter groups},
	Proceedings of the International Congress of Mathematicians (2010), 2389--2415.
	
\bibitem{BilleraEhrenborg:2000}
	L.~J.~Billera and R.~Ehrenborg,
	\emph{Monotonicity of the cd-index for polytopes},
	Math.\ Z.\ \textbf{233} (2000), 421--441. 
	
\bibitem{BilleraLee:1980}
	L.~J.~Billera and C.~W.~Lee,
	\emph{Sufficiency of McMullen's conditions for $f$-vectors of simplicial polytopes}, 
	Bull.\ Amer.\ Math.\ Soc.\ (N.S.) \textbf{2} (1980), 181--185. 
	
\bibitem{Blaschke:1923}
	W.~Blaschke, 
	\emph{Vorlesungen {\"u}ber {D}ifferentialgeometrie und geometrische {G}rundlagen von {E}insteins {R}elativit{\"a}tstheorie {II}.
	{A}ffine {D}ifferentialgeometrie ({G}erman)},
	Grundlehren der mathematischen Wissenschaften \textbf{7} (1923), Springer.

\bibitem{Boroczky:2010}
	K.~J.~B\"or\"oczky,
	\emph{Stability of the {B}laschke--{S}antal{\'o} and the affine isoperimetric inequality},
	Adv.\ Math.\ \textbf{225} (2010), 1914--1928.
	
\bibitem{BMMR:2013}
	K.~J.~Böröczky, E.~Jr.~Makai, M.~Meyer and S.~Reisner,
	\emph{On the volume product of planar polar convex bodies--lower estimates with stability},
	Studia Sci.\ Math.\ Hungar.\ \textbf{50} (2013), 159--198.

\bibitem{BouMil:1987}
	J.~Bourgain and V.~D.~Milman,
	\emph{New volume ratio properties for convex symmetric bodies in $R^n$}, 
	Invent.\ Math.\ \textbf{88} (1987), 319--340.
	
\bibitem{CagWer:2014}
	U.~Caglar and E.~M.~Werner,
	\emph{Divergence for $s$-concave and log concave functions},
	Adv.\ Math.\ \textbf{257} (2014), 219--247.
	
\bibitem{CagYe:2016}
	U.~Caglar and D.~Ye,
	\emph{Affine isoperimetric inequalities in the functional Orlicz-Brunn-Minkowski theory},
	Adv.\ in Appl.\ Math.\ \textbf{81} (2016), 78--114.
	
	
\bibitem{CFGLSW:2016}
	U.~Caglar, M.~Fradelizi, O.~Gu\'edon, J.~Lehec, C.~Sch\"utt and E.~M.~Werner,
	\emph{Functional versions of $L_p$-affine surface area and entropy inequalities},
	Int.\ Math.\ Res.\ Not.\ IMRN (2016), no.\ 4, 1223--1250.
	
	
\bibitem{DKY:2018}
	S.~Dann, J.~Kim and V.~Yaskin,
	\emph{Busemann's intersection inequality in hyperbolic and spherical spaces},
	Adv.\ Math.\ \textbf{326} (2018), 521--560.
	
\bibitem{DPP:2016}
	S.~Dann, G.~Paouris and P.~Pivovarov,
	\emph{Bounding marginal densities via affine isoperimetry},
	Proc.\ Lond.\ Math.\ Soc.\ \textbf{113} (2016), 140--162.
	
\bibitem{DeLRS:2010}
	J.~A.~De Loera, J.~Rambau and F.~Santos,
	\emph{Triangulations},
	Algorithms and Computation in Mathematics \textbf{25} (2010), Springer.

\bibitem{Ewald:1996}
 	G.~Ewald,
 	\emph{Combinatorial {C}onvexity and {A}lgebraic {G}eometry},
 	Graduate Texts in Mathematics \textbf{168} (1996), Springer.

\bibitem{FLM:1977}
	T.~Figiel, J.~Lindenstrauss and V.~D.~Milman,
	\emph{The dimension of almost spherical sections of convex bodies}, 
	Acta Math.\ \textbf{139} (1977), 53--94.

\bibitem{FHSZ:2013}
	R.~Freij, M.~Henze, M.~W.~Schmitt and G.~M.~Ziegler,
	\emph{Face numbers of centrally symmetric polytopes produced from split graphs},
	Electron.\ J.\ Combin.\ \textbf{20} (2013), Paper 32, 15 pp.
	
\bibitem{Gardner:2006}
	R.~J.~Gardner, 
	\emph{Geometric {T}omography},
	Encyclopedia of Mathematics and its Applications \textbf{58} (2006), Cambridge University Press.
	
%
	
\bibitem{GPV:2014}
	A.~Giannopoulos, G.~Paouris, B.-H.~Vritsiou,
	\emph{The isotropic position and the reverse Santal{\'o} inequality},
	Israel J.\ Math.\ \textbf{203} (2014), 1--22. 

\bibitem{Gruber:2007}
	P.~M.~Gruber, 
	\emph{Convex and {D}iscrete {G}eometry}, 
	Grundlehren der Mathematischen Wissenschaften \textbf{336} (2007), Springer.
	
\bibitem{Grunbaum:2003}
	B.~Grünbaum,
	\emph{Convex {P}olytopes},
	Graduate Texts in Mathematics \textbf{221} (2003), Springer.
	
\bibitem{HabPar:2014}
	C.~Haberl and L.~Parapatits,
	\emph{The centro-affine Hadwiger theorem},
	J.\ Amer.\ Math.\ Soc.\ \textbf{27} (2014), 685--705. 
	
\bibitem{HaberlSchuster:2009}
	C.~Haberl and F.~E.~Schuster,
	\emph{General $L_p$ affine isoperimetric inequalities},
	J.\ Differential Geom.\ \textbf{83} (2009), 1--26. 
	
\bibitem{Hanner:1956}
	O.~Hanner,
	\emph{Intersections of translates of convex bodies}, 
	Math.\ Scand.\ \textbf{4} (1956), 65--87.
	
%

\bibitem{IriShi:2017}
	H.~Iriyeh and M.~Shibata,
	\emph{Symmetric Mahler's conjecture for the volume product in the three dimensional case},
	2017, arXiv: \url{https://arxiv.org/abs/1706.01749}.
	
\bibitem{Ivaki:2015}
	M.~N.~Ivaki,
	\emph{Convex bodies with pinched Mahler volume under the centro-affine normal flows},
	Calc.\ Var.\ Partial Differential Equations \textbf{54} (2015), 831--846.
	
\bibitem{Kalai:1989}
	G.~Kalai,
	\emph{The number of faces of centrally-symmetric polytopes},
	Graphs Combin.\ \textbf{5} (1989), 389--391.

\bibitem{Kim:2014}
	J.~Kim,
	\emph{Minimal volume product near Hanner polytopes},
	J.\ Funct.\ Anal.\ \textbf{266} (2014), 2360--2402. 
	
\bibitem{Kuperberg:2008}
	G.~Kuperberg, 
	\emph{From the Mahler conjecture to Gauss linking integrals}, 
	Geom.\ Funct.\ Anal.\ \textbf{18} (2008), 870--892.

\bibitem{Leichtweiss:1986a}
	K.~Leichtweiss,
	\emph{Zur {A}ffinoberfl{\"a}che konvexer {K}{\"o}rper ({G}erman)},
	Manu\-scripta Math.\ \textbf{56} (1986), 429--464.

\bibitem{Leichtweiss:1988}
	K.~Leichtweiss,
	\emph{{\"U}ber einige {E}igenschaften der {A}ffinoberfl{\"a}che beliebiger konvexer {K}{\"o}rper ({G}erman)},
	Results Math.\ \textbf{13} (1988), 255--282.

\bibitem{Leichtweiss:1989}
	K.~Leichtweiss,
	\emph{{B}emerkungen zur {D}efinition einer erweiterten {A}ffinoberfl{\"a}che von {E}.\ {L}utwak ({G}erman)},
	Manuscripta Math.\ \textbf{65} (1989), 181--197.

\bibitem{LSW:2017}
	B.~Li, C. ~Sch\"utt and E.~M.~Werner,
	\emph{Floating functions},
	preprint (2017), arXiv: \url{https://arxiv.org/abs/1711.11088}.

\bibitem{LudRei:1999}
	M.~Ludwig and M.~Reitzner, 
	\emph{A characterization of affine surface area},
	Adv.\ Math.\ \textbf{147} (1999), 138--172.

\bibitem{LudRei:2010}
	M.~Ludwig and M.~Reitzner, 
	\emph{A classification of \,{${\rm SL}(n)$} invariant valuations}, 
	Ann.\ of Math.\ (2) \textbf{172} (2010), 1219--1267.

\bibitem{Lutwak:1985}
	E.~Lutwak,
	\emph{On the Blaschke-Santal{\'o} inequality},
	Ann.\ New York Acad.\ Sci.\ \textbf{440} (1985), 106--112.

\bibitem{Lutwak:1991}
	E.~Lutwak, 
	\emph{Extended affine surface area}, 
	Adv.\ Math.\ \textbf{85} (1991), 39--68.

\bibitem{Lutwak:1996}
	E.~Lutwak, 
	\emph{The {B}runn-{M}inkowski-{F}irey theory. {II}: {A}ffine and geominimal surface areas}, 
	Adv.\ Math.\ \textbf{118} (1996), 244--294.
	
\bibitem{Mahler:1939}
	K.~Mahler, 
	\emph{Ein Minimalproblem für konvexe Polygone (German)},
	Mathematica, Zutphen.\ B.\ \textbf{7} (1939), 118--127.
	
\bibitem{Mahler:1939a}
	K.~Mahler,
	\emph{Ein Übertragungsprinzip für konvexe Körper (German)},
	{\v C}asopis P{\v e}st.\ Mat.\ Fys.\ \textbf{68} (1939), 93--102. 
	
\bibitem{McMullen:1971}
	P.~McMullen,
	\emph{The numbers of faces of simplicial polytopes},
	Israel J.\ Math.\ \textbf{9} (1971), 559--570.
	
\bibitem{McMuShep:1971}
	P.~McMullen and G.~C.~Shephard,
	\emph{Convex polytopes and the upper bound conjecture},
	London Mathematical Society Lecture Note Series \textbf{3} (1971), Cambridge University Press.
	
\bibitem{Meyer:1991}
	M.~Meyer, 
	\emph{Convex bodies with minimal volume product in $R^2$},
	Monatsh.\ Math.\ \textbf{112} (1991), 297--301.
	
\bibitem{MeyerWerner:2000}
	M.~Meyer and  E.~M.~Werner,
	\emph{On the $p$-affine surface area}, 
	Adv.\ Math.\ \textbf{152} (2000), 288--313.

\bibitem{MordhorstWerner:2017}
	O.~Mordhorst and E.~M.~Werner,
	\emph{Duality of floating and illumination bodies},
	2017, arXiv: \url{https://arxiv.org/abs/1709.02424}.
	
\bibitem{MordhorstWerner:2017a}
	O.~Mordhorst and E.~M.~Werner,
	\emph{Duality of floating and illumination bodies for polytopes},
	2017, arXiv: \url{https://arxiv.org/abs/1709.02429}.

\bibitem{Ratcliffe:2006}
	J.~G.~Ratcliffe,
	\emph{Foundations of {H}yperbolic {M}anifolds}, 
	Graduate Texts in Mathematics \textbf{149} (2006), Springer.
	
\bibitem{Reisner:1986}
	S.~Reisner,
	\emph{Zonoids with minimal volume-product},
	Math.\ Z.\ \textbf{192} (1986), 339--346. 

\bibitem{SWZ:2009}
	R.~Sanyal, A.~Werner and G.~M.~Ziegler,
	\emph{On Kalai's conjectures concerning centrally symmetric polytopes},
	Discrete Comput.\ Geom.\ \textbf{41} (2009), 183--198.

\bibitem{Schneider:2014}
	R.~Schneider, 
	\emph{Convex {B}odies: the {B}runn-{M}inkowski {T}heory},
	Encyclopedia of Mathematics and its Applications \textbf{151} (2014),
	Cambridge University Press.

\bibitem{Schuett:1991}
	C.~Sch\"utt, 
	\emph{The convex floating body and polyhedral approximation},
	Israel J.\ Math.\ \textbf{73} (1991), 65--77.

\bibitem{SchuettWerner:1990}
	C.~Sch\"utt and E.~M.~Werner, 
	\emph{The convex floating body}, 
	Math.\ Scand.\ \textbf{66} (1990), 275--290.

\bibitem{SchuettWerner:1994}
	C.~Sch\"utt and E.~M.~Werner, 
	\emph{Homothetic floating bodies},
	Geom.\ Dedicata  \textbf{49} (1994), 335--348.

\bibitem{SchuettWerner:2004}
	C.~Sch\"utt and E.~M.~Werner, 
 	\emph{Surface bodies and {$p$}-affine surface area},
 	Adv.\ Math.\ \textbf{187} (2004), 98--145.

\bibitem{SjobergZiegler:2017}
	H.~Sjöberg and G.~M.~Ziegler,
	\emph{Semi-algebraic sets of $f$-vectors},
	2017, arXiv: \url{https://arxiv.org/abs/1711.01864}.
	
\bibitem{SjobergZiegler:2018}
	H.~Sjöberg, G.~M.~Ziegler,
	\emph{Characterizing face and flag vector pairs for polytopes},
	2018, arXiv: \url{https://arxiv.org/abs/1803.04801}.

\bibitem{Stancu:2012}
	A.~Stancu,
	\emph{Centro-affine invariants for smooth convex bodies},
	Int.\ Math.\ Res.\ Not.\ IMRN (2012), no.\ 10, 2289--2320. 
	
\bibitem{Stanley:1980}
	R.~P.~Stanley,
	\emph{The number of faces of a simplicial convex polytope}, 
	Adv.\ in Math.\ \textbf{35} (1980), 236--238. 
	
\bibitem{Stanley:1987}
	R.~P.~Stanley,
	\emph{On the number of faces of centrally-symmetric simplicial polytopes},
	Graphs Combin.\ \textbf{3} (1987), 55--66.
 
\bibitem{Steinitz:1906}
	E.~Steinitz,
	\textit{{\"U}ber die Eulerschen Polyederrelationen},
	Archiv f{\"u}r Mathematik und Physik \textbf{11} (1906), 86--88.

\bibitem{Werner:2002}
	E.~M.~Werner, 
	\emph{The {$p$}-affine surface area and geometric interpretations},
	Rend.\ Circ.\ Mat.\ Palermo (2) Suppl.\ (2002), 367--382.
	
\bibitem{Werner:2012}
	E.~M.~Werner, 
	\emph{R\'enyi Divergence and $L_p$-affine surface area for convex bodies},
	Adv.\  Math.\ \textbf{230}, (2012), 1040--1059.

\bibitem{WernerYe:2008}
	E.~M.~Werner and D.~Ye,
 	\emph{New $L_p$ affine isoperimetric inequalities},
	Adv.\ Math.\ \textbf{218} (2008), 762--780.

\bibitem{Ye:2016}
	D.~Ye,
	\emph{Dual Orlicz-Brunn-Minkowski theory: dual Orlicz $L_\phi$ affine and geominimal surface areas},
	J.\ Math.\ Anal.\ Appl.\ \textbf{443} (2016), 352--371.

\bibitem{Zhao:2016}
	Y.~Zhao,
	\emph{On $L_p$-affine surface area and curvature measures},
	Int.\ Math.\ Res.\ Not.\ IMRN (2016), no.\ 5, 1387--1423.
	
\bibitem{Ziegler:1995}
	G.~M.~Ziegler,
	\emph{Lectures on {P}olytopes},
	Graduate Texts in Mathematics \textbf{152} (1995), Springer.
\end{thebibliography}
\end{document}